\documentclass[11pt]{amsart}
\usepackage{amsfonts}
\usepackage{amsmath}
\usepackage{amssymb}
\usepackage{amscd}
\usepackage[mathscr]{eucal}
\usepackage{url}

\allowdisplaybreaks[1]

\newcommand{\ph}[2]{{\left({#1}\right)}_{#2}}

\renewcommand*{\bar}{\overline}
\newcommand{\gfp}[1]{\Gamma_p{\left({#1}\right)}}
\newcommand{\biggfp}[1]{\Gamma_p{\bigl({#1}\bigr)}}
\newcommand{\Biggfp}[1]{\Gamma_p{\Bigl({#1}\Bigr)}}

\newcommand{\bin}[2]{\left({\genfrac{}{}{0pt}{}{#1}{#2}}\right)}

\newcommand{\biggbin}[2]{\biggl({\genfrac{}{}{0pt}{}{#1}{#2}}\biggr)}

\theoremstyle{plain}
\newtheorem{theorem}{Theorem}[section]
\newtheorem{lemma}[theorem]{Lemma}
\newtheorem{prop}[theorem]{Proposition}
\newtheorem{cor}[theorem]{Corollary}
\theoremstyle{definition}
\newtheorem{defi}[theorem]{Definition}

\newtheorem{conj}[theorem]{Conjecture}

\numberwithin{equation}{section}

\begin{document}

\title[Supercongruence conjectures of Rodriguez-Villegas]{Supercongruence Conjectures of Rodriguez-Villegas}
\author{Dermot M\lowercase{c}Carthy}  

\address{School of Mathematical Sciences, University College Dublin, Belfield, Dublin 4, Ireland}

\email{dermot.mc-carthy@ucdconnect.ie}



\date{July 20, 2009}

\begin{abstract}
In examining the relationship between the number of points over $\mathbb{F}_p$ on certain Calabi-Yau manifolds and hypergeometric series which correspond to a particular period of the manifold, Rodriguez-Villegas identified 22 possible supercongruences. We provide a framework of congruences covering all 22 cases. Using this framework we prove one of the outstanding supercongruence conjectures between a special value of a truncated ordinary hypergeometric series and the $p$-th Fourier coefficient of a modular form. In the course of this work we also establish two new binomial coefficient-harmonic sum identities.
\end{abstract}

\maketitle

\section{Introduction}
The term \emph{supercongruence} was first introduced by Beukers in \cite{B}. Let $A(n)$ and $B(n)$ be the the numbers defined by 
\begin{align*}
A(n):= \sum_{j=0}^{n} \bin{n+j}{j}^2 \bin{n}{j}^2
&& \textup{and} && B(n):= \sum_{j=0}^{n} \bin{n+j}{j} \bin{n}{j}^2
\end{align*}
for $n\geq0$. These numbers were used by Ap\'ery in his proof of the irrationality of $\zeta(3)$ and $\zeta(2)$ respectively \cite{Ap},\cite{vdP} and are commonly known as the Ap\'ery numbers. Beukers proved that, for $m,r \in \mathbb{Z^+}$ and $p\geq5$ a prime,
\begin{align*}
A(mp^r-1)\equiv A(mp^{r-1}-1) \pmod{p^{3r}}
&& \textup{and} &&
B(mp^r-1)\equiv B(mp^{r-1}-1) \pmod{p^{3r}}.
\end{align*}
He noted also that a weaker version of these congruences, modulo $p^r$, arose in a natural way from formal groups. As the congruences modulo $p^{3r}$ were stronger than those suggested by formal group theory he named them supercongruences. The term has since come to cover individual congruences modulo $p^k$ where $k>1$. For example, Beukers conjectured \cite{B1} that for $p\geq3$ a prime,
\begin{equation}\label{for_ANS}
A\left(\tfrac{p-1}{2}\right)\equiv \gamma(p) \pmod{p^2}
\end{equation}
where $\gamma(p)$ is given by
\begin{equation*}
\eta^{4}(2z) \eta^{4}(4z)=\sum_{n=1}^{\infty} \gamma(n)q^n,
\end{equation*}
$q:=e^{2 \pi i z}$ and
\begin{equation*}
\eta(z):=q^{\frac{1}{24}} \prod_{n=1}^{\infty}(1-q^n)
\end{equation*}
is Dedekind's eta function. This became known as the Ap\'ery number supercongruence and was proved by Ahlgren and Ono \cite{AO}.

In \cite{R} Rodriguez-Villegas examined the relationship between the number of points over $\mathbb{F}_p$ on certain Calabi-Yau manifolds and (truncated) hypergeometric series which correspond to a particular period of the manifold. In doing so, he identified 22 possible supercongruences which can be categorised by the dimension, $D$, of the manifold as outlined below.

We first define the truncated hypergeometric series by
\begin{equation*}
{{_pF_q} \left[ \begin{array}{ccccc} a_1, & a_2, & a_3, & \dotsc, & a_p \vspace{.05in}\\
\phantom{a_1} & b_1, & b_2, & \dotsc, & b_q \end{array}
\Big| \; z \right]}_{m}
:=\sum^{m}_{n=0}
\frac{\ph{a_1}{n} \ph{a_2}{n} \ph{a_3}{n} \dotsm \ph{a_p}{n}}
{\ph{b_1}{n} \ph{b_2}{n} \dotsm \ph{b_q}{n}}
\; \frac{z^n}{{n!}}
\end{equation*}
where $a_i$, $b_i$ and $z$ are complex numbers, with none of the $b_i$ being negative integers or zero, $\ph{a}{0}:=1$ and $\ph{a}{n} := a(a+1)(a+2)\dotsm(a+n-1)$ for positive integers $n$, and, $m$, $p$ and $q$ are positive integers. We also let $\phi(\cdot)$ denote Euler's totient function and $\left(\frac{\cdot}{p}\right)$ the Legendre symbol modulo $p$.

For $D=1$, associated to certain elliptic curves, 4 supercongruences were identified. They were all of the form 
\begin{equation*}
{{_{2}F_1} \left[ \begin{array}{cc} \frac{1}{d}, & 1-\frac{1}{d}\vspace{.05in}\\
\phantom{\frac{1}{d}} & 1 \end{array}
\Big| \; 1 \right]}_{p-1}
\equiv
\left(\frac{-t}{p}\right)
\pmod{p^2}
\end{equation*}
where $\phi(d) \leq 2$, $1\leq t \leq4$ and $p$ is a prime not dividing $d$. These cases have been proven by Mortenson  \cite{M1}, \cite{M2}.

For $D=2$ another 4 supercongruences were identified which related to certain modular K3 surfaces. These were all of the form 
\begin{equation*}
{{_{3}F_2} \left[ \begin{array}{ccc} \frac{1}{2}, & \frac{1}{d}, & 1-\frac{1}{d}\vspace{.05in}\\
\phantom{\frac{1}{d}} & 1, &1 \end{array}
\Big| \; 1 \right]}_{p-1}
\equiv
a(p)
\pmod{p^2}
\end{equation*}
where $\phi(d) \leq 2$, $p$ is a prime not dividing $d$ and $a(p)$ is the $p^{th}$ Fourier coefficient of a weight three modular form on a congruence subgroup of $SL_2(\mathbb{Z})$. For the case $d=2$, $a(p)$ is given by
\begin{equation*}
\eta^{6}(4z)=\sum_{n=1}^{\infty} a(n)q^n ,
\end{equation*}
and the congruence can also be written as
\begin{equation*}
B\left(\frac{p-1}{2}\right) 
\equiv
a(p)
\pmod{p^2}.
\end{equation*}
This supercongruence was first conjectured by Beukers and Stienstra \cite{BS} and was proved by Ahlgren \cite{A}, Ishikawa \cite{I} and Van Hamme \cite{VH}. The other $D=2$ cases are dealt with by Mortenson \cite{M} where they have been proven for $p\equiv 1 \pmod d$ and up to sign otherwise. 

The remaining 14 supercongruence conjectures relate to Calabi-Yau threefolds (i.e. $D=3$). The threefolds in question are complete intersections of hypersurfaces, of which 13 are discussed by Batyrev and van Straten in \cite{BvS}. The supercongruences can be expressed as either
\begin{equation*}
{{_{4}F_3} \left[ \begin{array}{cccc} \frac{1}{d_1}, & 1-\frac{1}{d_1}, & \frac{1}{d_2}, & 1-\frac{1}{d_2}\vspace{.05in}\\
\phantom{\frac{1}{d_1}} & 1, & 1, & 1 \end{array}
\Big| \; 1 \right]}_{p-1}
\equiv
b(p)
\pmod{p^3}
\end{equation*}
where $\phi(d_i) \leq 2$ and $p$ is a prime not dividing $d_i$, or
\begin{equation}\label{type2}
{{_{4}F_3} \left[ \begin{array}{cccc} \frac{1}{d}, & \frac{r}{d}, & 1-\frac{r}{d}, & 1-\frac{1}{d}\vspace{.05in}\\
\phantom{\frac{1}{d_1}} & 1, & 1, & 1 \end{array}
\Big| \; 1 \right]}_{p-1}
\equiv
b(p)
\pmod{p^3}
\end{equation}
where $\phi(d) =4$, $gcd(r,d)=1$, $p$ is a prime not dividing $d$ and $b(p)$ is the $p^{th}$ Fourier coefficient of a weight four modular form on a congruence subgroup of $SL_2(\mathbb{Z})$. To date only one of these cases has been proven (see \cite{K}). It is of the first type with $d_1=d_2=2$ and is an extension of the Ap\'ery number supercongruence (\ref{for_ANS}). Let 
\begin{equation}\label{for_ModForm}
f(z):= f_1(z)+5f_2(z)+20f_3(z)+25f_4(z)+25f_5(z)=\sum_{n=0}^{\infty} c(n) q^n
\end{equation}
where
\begin{align*}
f_1(z)&:=\eta^4(z) \hspace{2pt} \eta^4(5z),\\
f_2(z)&:=\eta^3(z) \hspace{2pt} \eta^4(5z) \hspace{2pt} \eta(25z),\\
f_3(z)&:=\eta^2(z) \hspace{2pt} \eta^4(5z) \hspace{2pt} \eta^2(25z),\\
f_4(z)&:=\eta(z) \hspace{2pt} \eta^4(5z) \hspace{2pt} \eta^3(25z),\\ 
\intertext{and}
f_5(z)&:=\eta^4(5z) \hspace{2pt} \eta^4(25z).
\end{align*}
Then $f$ is a weight four newform in the space of weight four cusp forms on the congruence subgroup $\Gamma_0(25)$.  We now list one of the outstanding conjectures of type (\ref{type2}).
\begin{conj}\label{conj_RV} For a prime $p \neq 5$ 
\begin{equation*}
{_4F_3} \left[ \begin{array}{cccc} \frac{1}{5}, & \frac{2}{5}, & \frac{3}{5}, & \frac{4}{5} \\
\phantom{\frac{1}{5},} & 1, & 1, & 1\end{array}
\Big| \; 1 \right]_{p-1}
\equiv c(p) \pmod {p^3}.
\end{equation*}
\end{conj}
\noindent One of the main results of this paper is the following theorem.
\begin{theorem}\label{thm_DMCMain} Conjecture \ref{conj_RV} is true.
\end{theorem}

In addition to proving Theorem \ref{thm_DMCMain}, the purpose of this paper is to develop a framework of congruences which cover all 22 supercongruence cases above. This will be the subject of Section 4 and will form a key part in proving Theorem \ref{thm_DMCMain}. Section 2 recalls some properties of Gauss and Jacobi sums, the $p$-adic gamma function and related ideas. In Section 3 we develop two new binomial coefficient-harmonic sum identities which are needed in Section 4. Finally, the proof of Theorem \ref{thm_DMCMain} appears in Section 5.


\section{Preliminaries}
We briefly recall some properties of Gauss and Jacobi sums and the $p$-adic gamma function and also develop some preliminary results which we use in Sections 4 and 5. Throughout, we let $\mathbb{F}_{p}$ denote the finite field with $p$ elements. We extend the domain of all characters $\chi$ on $\mathbb{F}^{*}_{p}$ to $\mathbb{F}_{p}$, by defining $\chi(0):=0$ (including the trivial character $\varepsilon$). 

\subsection{Gauss and Jacobi Sums}
For further details see \cite{BEW} and \cite{IR}, noting that we have adjusted results to take into account $\varepsilon(0)=0$. 
We let $T$ denote a generator for the group of characters on $\mathbb{F}_p$.
Then we have the following orthogonal relations.
\begin{prop}For a character $\chi=T^n$ on $\mathbb{F}_{p}$ we have
\begin{equation}\label{for_TOrthEl}
\sum_{x \in \mathbb{F}_p} \chi(x)=
\sum_{x \in \mathbb{F}_p} T^n(x) =
\begin{cases}
p-1 & \text{if $\chi=T^n = \varepsilon$}  ,\\
0 & \text{if $\chi=T^n \neq \varepsilon$}  ,
\end{cases}
\end{equation}
and
\begin{equation}\label{for_TOrthCh}
\sum_{\chi} \chi(x)=
\sum_{n=0}^{p-2} T^n(x) =
\begin{cases}
p-1 & \text{if $x=1$}  ,\\
0 & \text{if $x \neq 1$}  .
\end{cases}
\end{equation}
\end{prop}
\noindent We define the additive character $\theta : \mathbb{F}_p \rightarrow \mathbb{C}$ by $\theta(\alpha):=e^{\frac{2 \pi i \alpha}{p}}$.
\noindent It is easy to see that
\begin{equation}\label{for_AddProp}
\theta(a+b)=\theta(a) \theta(b)
\end{equation}
and
\begin{equation}\label{sum_AddChar}
\sum_{x \in \mathbb{F}_p} \theta(x) = 0.
\end{equation}
Recall that for a character $\chi$ on $\mathbb{F}_p$, the Gauss sum $G(\chi)$ is defined by 
\begin{equation*}
G(\chi):= \sum_{x \in \mathbb{F}_p} \chi(x) \theta(x)  .
\end{equation*}
The following important result gives a simple expression for the product of two Gauss sums.
\begin{prop}\label{prop_GaussConj}

\begin{equation}\label{for_GaussConj}
G(\chi)G(\bar{\chi})=
\begin{cases}
\chi(-1) p & \text{if } \chi \neq \varepsilon,\\
1 & \text{if } \chi= \varepsilon.
\end{cases}
\end{equation}
\end{prop}
\noindent Another important product formula for Gauss sums is the Hasse-Davenport formula.
\begin{theorem}[Hasse, Davenport]\label{thm_HD}
Let $\chi$ be a character of order $m$ on $\mathbb{F}_p$ for some positive integer $m$. For a character $\psi$ on $\mathbb{F}_p$ we have
\begin{equation*}
\prod_{i=0}^{m-1} G(\chi^i \psi) = G(\psi^m) \psi^{-m}(m)\prod_{i=1}^{m-1} G(\chi^i).
\end{equation*}
\end{theorem}

\noindent We let $G_m := G(T^m)$ for brevity and note that $G_0=G(\varepsilon)=-1$. The next result expresses the additive character as a sum of Gauss sums (see \cite[Lemma II.1.5, page 7]{F}).
\begin{prop}
For $\alpha \in \mathbb{F}_p^*$\hspace{1pt},
\begin{equation}\label{for_AddtoGauss}
\theta(\alpha)= \frac{1}{p-1} \sum_{m=0}^{p-2} G_{-m} T^m(\alpha).
\end{equation}
\end{prop}

We now introduce generalised Jacobi sums.
Let $\chi_1, \chi_2, \dotsc, \chi_k$ be characters on $\mathbb{F}_{p}$. Then the generalised Jacobi sum $J(\chi_1, \chi_2, \dotsc, \chi_k)$ is defined by
\begin{equation}\label{for_GenJacSum}
J(\chi_1, \chi_2, \dotsc, \chi_k):= \sum_{\substack{t_1+t_2+ \dotsm + t_k=1\\ t_i \in \mathbb{F}_p}} \chi_1(t_1) \chi_2(t_2) \dotsm \chi_k(t_k).
\end{equation}
When $k=2$ this reduces to the ordinary Jacobi sum which has the following properties.
\begin{prop}\label{prop_JacBasic}
For a non-trivial character $\chi$ and trivial character $\varepsilon$ we have\\[3pt]
\noindent \textup{(1)}
$J(\varepsilon,\varepsilon)=p-2$;\\[3pt]
\noindent \textup{(2)}
$J(\varepsilon, \chi)=-1$; and\\[3pt]
\noindent \textup{(3)}
$J(\chi,\bar{\chi}) = -\chi(-1)$.
\end{prop}
\noindent The following proposition gives a reduction formula in cases where $k \geq 2$.
\begin{prop}\label{prop_JacRed}
\begin{equation*}
J(\chi_1, \chi_2, \dotsc, \chi_k)=
\begin{cases}
J(\chi_1 \chi_2 \dotsm \chi_{k-1}, \chi_k)J(\chi_1, \chi_2, \dotsc, \chi_{k-1})
& \quad \text{if } \chi_1\chi_2\dotsm\chi_{k-1} \neq \varepsilon, \\[18pt]
(p-1)^{k-1}-
\begin{cases}
p \; J(\chi_1, \chi_2, \dotsc, \chi_{k-1}) & \text{if } \chi_k \neq \varepsilon,\\[9pt]
J(\chi_1, \chi_2, \dotsc, \chi_{k-1})  & \text{if } \chi_k=\varepsilon.
\end{cases}\\[-34pt]
&\quad \text{if } \chi_1, \chi_2, \dotsc, \chi_{k-1} \text{ }\\[0pt]
&\quad \text{are all trivial},\\[18pt]
\begin{cases}
-p \; J(\chi_1, \chi_2, \dotsc, \chi_{k-1}) & \text{if }\chi_k \neq \varepsilon,\\[9pt]
-J(\chi_1, \chi_2, \dotsc, \chi_{k-1})  & \text{if }\chi_k=\varepsilon.
\end{cases}
& \quad \text{otherwise}.
\end{cases}\\[12pt]
\end{equation*}
\end{prop}
\noindent Special cases of this proposition yield the following corollaries.
\begin{cor}\label{cor_JacRedSmall}
For $\chi_1\chi_2\dotsm\chi_k$ trivial but at least one of $\chi_1, \chi_2, \dotsc, \chi_k$ non-trivial,
\begin{equation*}
 J(\chi_1, \chi_2, \dotsc, \chi_k)=-\chi_k(-1)  J(\chi_1, \chi_2, \dotsc, \chi_{k-1}) \; .
 \end{equation*}
\end{cor}

\begin{cor}For $\chi_1, \chi_2, \dotsc, \chi_k$ all trivial,
\begin{equation*}
J(\chi_1, \chi_2, \dotsc, \chi_k) =\frac{(p-1)^k+(-1)^{k+1}}{p}.
\end{equation*}
\end{cor}

\noindent We can relate generalised Jacobi sums to Gauss sums via the following.
\begin{prop}\label{prop_JactoGauss}
For $\chi_1, \chi_2, \dotsc, \chi_k$ not all trivial,\\
\begin{equation*}
J(\chi_1, \chi_2, \dotsc, \chi_k)=
\begin{cases}
\dfrac{G(\chi_1)G(\chi_2)\dotsc G(\chi_k)}{G(\chi_1 \chi_2 \dotsm \chi_k)}
& \qquad \text{if } \chi_1 \chi_2 \dotsm \chi_k \neq \varepsilon,\\[18pt]
-\dfrac{G(\chi_1)G(\chi_2)\dotsc G(\chi_k)}{p}
&\qquad \text{if }\chi_1 \chi_2 \dotsm \chi_k = \varepsilon \: .
\end{cases}
\end{equation*}
\end{prop}

\noindent We now develop some results involving generalised Jacobi sums which we use in Section 5.
\begin{lemma}\label{lem_JacConsMul}
Let $T$ be a generator for the group of characters on $\mathbb{F}_p$. For a prime $p\equiv1 \pmod5$ and $a,b,c,r \in \mathbb{Z}\setminus\{0\}$,
\begin{equation*}
\sum_{k=1}^4 J\left(T^{akt},T^{bkt},T^{ckt}\right)=\sum_{k=1}^4 J\left(T^{arkt},T^{brkt},T^{crkt}\right)
\end{equation*}
where $t=\frac{p-1}{5}$.
\end{lemma}
\begin{proof}
The result follows from the facts that $T^{nt}=T^{mt}$ when $n\equiv m \pmod 5$ and $\mathbb{F}_5^*$ is a multiplicative group.
\end{proof}

\begin{lemma}\label{lem_JacTsum}
Let $T$ be a generator for the group of characters on $\mathbb{F}_p$. For a prime $p\equiv1 \pmod5$ and $a,b,c \in \mathbb{Z}$ such that $a+c, b+c \not\equiv 0 \pmod5$,
\begin{equation*}
\sum_{e=0}^{p-2} T^e(-1) J(T^{-e+at},T^{-e+bt},T^{e+ct})=-(p-1)
\end{equation*}
where $t=\frac{p-1}{5}$.
\end{lemma}
\begin{proof}
\begin{align*}
\sum_{e=0}^{p-2} T^e(-1) J(T^{-e+at},T^{-e+bt},T^{e+ct})
&=\sum_{e=0}^{p-2} T^e(-1)  \sum_{\substack{t_1+t_2+ t_3=1\\ t_i \in \mathbb{F}_p^*}} T^{-e+at}(t_1) T^{-e+bt}(t_2) T^{e+ct}(t_3)\\
&=\sum_{e=0}^{p-2} T^e(-1)  \sum_{\substack{t_1+t_2+ t_3=1\\ t_i \in \mathbb{F}_p^*}}T^{-e}\left(\frac{t_1t_2}{t_3}\right) T^t({t_1}^a{ t_2}^b {t_3}^c)\\
&= \sum_{\substack{t_1+t_2+ t_3=1\\ t_i \in \mathbb{F}_p^*}} T^t({t_1}^a{ t_2}^b {t_3}^c) \sum_{e=0}^{p-2} T^{-e}\left(-\frac{t_1t_2}{t_3}\right)\\
&= (p-1) \sum_{\substack{t_1+t_2+ t_3=1\\ t_i \in \mathbb{F}_p^*\\-\frac{t_1t_2}{t_3}=1}} T^t({t_1}^a{ t_2}^b {t_3}^c) \qquad \qquad \qquad \text{by (\ref{for_TOrthCh}).}
\end{align*}
All possible triples $(t_1,t_2,t_3)$ satisfying the conditions of the summation can be represented by $(t_1,1,-t_1)$ and $(1,t_2,-t_2)$, not counting $(1,1,-1)$ twice. So
\begin{align*}
\sum_{e=0}^{p-2} T^e(-1) & J(T^{-e+at},T^{-e+bt},T^{e+ct})\\
&\qquad \qquad \qquad =
(p-1) \left[\sum_{t_1\in \mathbb{F}_p^*} T^t({t_1}^{a+c}\: (-1)^c) + \sum_{t_2\in \mathbb{F}_p^*} T^t({t_2}^{b+c} \:(-1)^c) - T^t((-1)^c) \right]\\
&\qquad \qquad \qquad=(p-1)\; T((-1)^{tc}) \left[\sum_{t_1\in \mathbb{F}_p^*} T^{(a+c)t}(t_1) + \sum_{t_2\in \mathbb{F}_p^*} T^{(b+c)t}(t_2) - 1 \right].
\end{align*}
Now $a+c, b+c \not\equiv 0 \pmod5$ so $T^{(a+c)t}$, $T^{(b+c)t} \neq \varepsilon$ and $\displaystyle \sum_{t_1\in \mathbb{F}_p^*} T^{(a+c)t}(t_1), \sum_{t_2\in \mathbb{F}_p^*} T^{(b+c)t}(t_2)$ both equal $0$ by (\ref{for_TOrthEl}). For $p \equiv 1 \pmod 5$, $t$ is even and so $T((-1)^{tc})=1$. Overall we get
\begin{equation*}
\sum_{e=0}^{p-2} T^e(-1) J(T^{-e+at},T^{-e+bt},T^{e+ct})=-(p-1)
\end{equation*}
as required.
\end{proof}

\begin{cor}\label{cor_GaussTsum}
Let $T$ be a generator for the group of characters on $\mathbb{F}_p$. For a prime $p\equiv1 \pmod5$ and $a,b,c \in \mathbb{Z}$ such that $a+c, b+c \not\equiv 0 \pmod5$,
\begin{equation*}
\sum_{e=0}^{p-2} G_{-e+at} G_{-e+bt} G_{e+ct} G_{e-(a+b+c)t)}=-p(p-1)
\end{equation*}
where $t=\frac{p-1}{5}$.
\end{cor}
\begin{proof}
Using Proposition \ref{prop_JactoGauss} and Corollary \ref{cor_JacRedSmall} we see that
\begin{align*}
\sum_{e=0}^{p-2} G_{-e+at} G_{-e+bt} G_{e+ct} G_{e-(a+b+c)t)}
&=-p\sum_{e=0}^{p-2} J\left(T^{-e+at}, T^{-e+bt}, T^{e+ct}, T^{e-(a+b+c)t)}\right) \\
&=p\sum_{e=0}^{p-2} T^{e-(a+b+c)t}(-1) J\left(T^{-e+at}, T^{-e+bt}, T^{e+ct}\right) \\
&=p\sum_{e=0}^{p-2} T^{e}(-1) J\left(T^{-e+at}, T^{-e+bt}, T^{e+ct}\right) \\
&=-p(p-1)
\end{align*}
as at least one of $T^{-e+at}, T^{-e+bt}$ or $T^{e+ct}$ is non-trivial due to the conditions imposed on $a,b$ and $c$.
\end{proof}

We now recall a formula for counting points using the additive character.
If $f(x_1, x_2, \ldots x_n) \in \mathbb{F}_p[x_1, x_2, \ldots x_n]$, then the number of points, $N_p^*$, in $\mathbb{A}^n(\mathbb{F}_p)$ satisfying
$f(x_1, x_2, \ldots x_n) =0$ is given by
\begin{equation}\label{for_CtgPts}
p N_p^* = p^n +\sum_{y \in \mathbb{F}_p^*} \sum_{x_1, x_2, \ldots x_n \in \mathbb{F}_p}
\theta(y \: f(x_1, x_2, \ldots x_n)) \; .
\end{equation}
Notice also that if  $f(x_0, x_1, x_2, \ldots x_n) \in \mathbb{F}_p[x_0, x_1, x_2, \ldots x_n]$, then
the number of points, $N_p$, in $\mathbb{P}^n(\mathbb{F}_p)$, satisfying $f(x_0, x_1, x_2, \ldots x_n)=0$ can be determined by
\begin{equation}\label{for_CtgPtsSplit}
N_p=N_p^0 +N_p^1
\end{equation}
where $N_p^0$ is the number of points in $\mathbb{P}^{n-1}(\mathbb{F}_p)$ on $f(0, x_1, x_2 \ldots x_n)=0$ and $N_p^1$ is the number of points in $\mathbb{A}^n(\mathbb{F}_p)$ on $f(1, x_1, x_2 \ldots x_n)=0$.

\subsection{$p$-adic preliminaries}
We first define the Teichm\"{u}ller character to be the primitive character $\omega: \mathbb{F}_p \rightarrow \mathbb{Z}_p$ satisfying $\omega(x) \equiv x \pmod p$ for all $x \in \{0,1, \ldots, p-1\}$.
It also satisfies
\begin{equation}\label{prop_TechChar}
\omega(x) \equiv x^{p^{n-1}} \pmod {p^n}.
\end{equation}
We now recall the $p$-adic gamma function. For further details see \cite{Ko}.
Let $p$ be an odd prime.  For $n \in \mathbb{N}$, we define the $p$-adic gamma function as
\begin{align*}
\gfp{n} &:= {(-1)}^n \prod_{\substack{j<n\\p \nmid j}} j \\
\intertext{and extend to all $x \in\mathbb{Z}_p$ by setting}
\gfp{x} &:= \lim_{n \rightarrow x} \gfp{n}
\end{align*}
where $n$ runs through any sequence of positive integers $p$-adically approaching $x$ and 
$\gfp{0}:=1$. 
This limit exists, is independent of how $n$ approaches $x$ and determines a continuous function
on $\mathbb{Z}_p$.
We now state some basic properties of the $p$-adic gamma function.
\begin{prop}\label{prop_pGamma}
Let $x, y \in \mathbb{Z}_{p}$ and $n \in \mathbb{N}$. Then\\[6pt]
\textup{(1)}
$\gfp{x+1}=
\begin{cases}
-x \hspace{1pt} {\gfp{x}} & \quad \text{if } x \in \mathbb{Z}_p^* \; ,\\
- {\gfp{x}} & \quad \text{otherwise} \;.
\end{cases}$
\\[6pt]

\noindent \textup{(2)}
$\Gamma_p(x)\Gamma_p(1-x) = {(-1)}^{x_0}$,
where $x_0 \in \{1,2, \dotsc, {p}\}$ satisfies $x_0 \equiv x \pmod {p}$ .
\\

\noindent \textup{(3)}
If $x \equiv y \pmod{p^n}$, then $\gfp{x} \equiv \gfp{y} \pmod{p^n}$.
\end{prop}

\noindent \textup{(4)}
For $m\in\mathbb{Z}^{+}$, $p \nmid m$ and $x=\frac{r}{p-1}$ with $0\leq r \leq p-1$ then
\begin{equation*}
\prod_{h=0}^{m-1} \gfp{\tfrac{x+h}{m}}=\omega\left(m^{(1-x)(1-p)}\right)
\gfp{x} \prod_{h=1}^{m-1} \gfp{\tfrac{h}{m}}\;.
\end{equation*}

\noindent
We consider the logarithmic derivatives of $\Gamma_p$. For $x \in \mathbb{Z}_{p}$, define
$$G_1(x):= \dfrac{\Gamma_p^{\prime}(x)}{\gfp{x}} \qquad \textup{and} \qquad
 G_2(x):= \dfrac{\Gamma_p^{\prime \prime}(x)}{\gfp{x}}.$$ These also satisfy some basic properties  which we state below (see \cite{AO}, \cite{CDE} and \cite{K}).

\begin{prop}\label{prop_pGammaG}
Let $x \in \mathbb{Z}_{p}$. Then\\[6pt]
\textup{(1)}
$G_1(x+1) - G_1(x) =
\begin{cases}
 1/ x & \quad \text{if } x \in \mathbb{Z}_p^* \;,\\
0 & \quad \text{otherwise}   \; .
\end{cases}$
\\[6pt]

\noindent \textup{(2)}
$G_1(x+1)^2 - G_2(x+1) - G_1(x)^2 + G_2(x)=
\begin{cases}
 1/ x^2 & \quad \text{if } x \in \mathbb{Z}_p^* \;,\\
0 & \quad \text{otherwise}   \; .
\end{cases}$
\\[6pt]

\noindent \textup{(3)}
$G_1(x) = G_1(1-x)$.
\\

\noindent \textup{(4)}
$G_1(x)^2 - G_2(x) = - G_1(1-x)^2 + G_2(1-x)$.\\
\end{prop}

\begin{proof}
(1) and (3) are obtained from differentiating the results in Proposition \ref{prop_pGamma} (1) and (2) respectively, while (2) and (4) follow from differentiating (1) and (3).
\end{proof}

\noindent We also have some congruence properties of the $p$-adic gamma function and its logarithmic derivatives as follows.
\begin{prop}\label{prop_pGammaCong}
Let $p \geq 7$ be a prime, $x \in \mathbb{Z}_{p}$ and $z \in p  \hspace{1pt} \mathbb{Z}_{p}$ . Then\\[6pt]
\textup{(1)}
$G_1(x)$, $G_2(x) \in \mathbb{Z}_p$.\\[-3pt]

\noindent \textup{(2)}
$\Gamma_p(x+z) \equiv \Gamma_p(x) \left(1+zG_1(x) +\frac{z^2}{2} G_2(x) \right) \pmod{p^3}$.\\[-3pt]

\noindent \textup{(3)}
$\Gamma_p^{\prime}(x+z) \equiv \Gamma_p^{\prime}(x) +z \Gamma_p^{\prime \prime} (x) \pmod{p^2}$.
\end{prop}

\begin{cor}\label{cor_pGammaCong1}
Let $p \geq 7$ be a prime, $x \in \mathbb{Z}_{p}$ and $z \in p  \hspace{1pt} \mathbb{Z}_{p}$ . Then\\[6pt]
\noindent \textup{(1)}
$\Gamma_p^{\prime}(x+z) \equiv \Gamma_p^{\prime}(x)  \pmod{p}$.\\[-3pt]

\noindent \textup{(2)}
$\Gamma_p^{\prime\prime}(x+z) \equiv \Gamma_p^{\prime\prime}(x)  \pmod{p}$.\\[-3pt]

\noindent \textup{(3)}
$G_1(x+z) \equiv G_1(x)  \pmod{p}$.\\[-3pt]

\noindent \textup{(4)}
$G_2(x+z) \equiv G_2(x)  \pmod{p}$.
\end{cor}

\begin{proof}
By definition, $\Gamma_p(x) \in \mathbb{Z}_p^{*}$. Thus, by Proposition \ref{prop_pGammaCong} (1) and the definitions of $G_1(x)$ and $G_2(x)$, we see that $\Gamma_p^{\prime}(x)$ and $\Gamma_p^{\prime\prime}(x) \in \mathbb{Z}_p$.
Observe that (1) then follows from Proposition \ref{prop_pGammaCong} (3).
For (2), one uses similar methods to Proposition 2.3 in \cite{K}.
Finally (3) \& (4) follow from (1)  and (2) and the definitions of $G_1(x)$ and $G_2(x)$.
\end{proof}

\begin{cor}\label{cor_pGammaCong2}
Let $p \geq 7$ be a prime, $x \in \mathbb{Z}_{p}$ and $z \in p  \hspace{1pt} \mathbb{Z}_{p}$ . Then
\begin{equation*}
G_1(x)\equiv G_1(x+z)+z\left(G_1(x+z)^2 - G_2(x+z)\right)
\pmod{p^2}.
\end{equation*}
\end{cor}

\begin{proof}
By Proposition \ref{cor_pGammaCong1} we see that
\begin{align*}
G_1(x)= \frac{\Gamma_p^{\prime}(x)}{\gfp{x}}
\equiv \frac{\Gamma_p^{\prime}(x+z) - z \Gamma_p^{\prime \prime} (x)}{\Gamma_p(x+z) - z \Gamma_p^{\prime}(x)}
\pmod{p^2}.
\end{align*}
Multiplying the numerator and denominator by ${\Gamma_p(x+z) + z \Gamma_p^{\prime}(x)}$ we get that
\begin{align*}
G_1(x)
&\equiv \frac{
\Gamma_p^{\prime}(x+z) \Gamma_p(x+z) + z 
\left(\Gamma_p^{\prime}(x) \Gamma_p^{\prime}(x+z)-\Gamma_p(x+z) \Gamma_p^{\prime \prime} (x)\right)
}
{\Gamma_p(x+z)^2}\\
&\equiv \frac{
\Gamma_p^{\prime}(x+z) \Gamma_p(x+z) + z 
\left(\Gamma_p^{\prime}(x+z) \Gamma_p^{\prime}(x+z)-\Gamma_p(x+z) \Gamma_p^{\prime \prime} (x+z)\right)
}
{\Gamma_p(x+z)^2}\\
&\equiv G_1(x+z) +z \left(G_1(x+z)^2-G_2(x+z)\right)
\pmod{p^2}.
\end{align*}
\end{proof}

The Gross-Koblitz formula \cite{GK} allows us to convert between Gauss sums and the $p$-adic gamma function. Let $\zeta_p=\theta(1)=e^{\frac{2 \pi i}{p}}$ and $\pi \in \mathbb{C}_p$ be the fixed root of $x^{p-1}+p=0$ satisfying $\pi \equiv \zeta_p-1 \pmod{{(\zeta_p-1)}^2}$. Then we have the following result.
\begin{theorem}[Gross, Koblitz]\label{thm_GrossKoblitz}
\begin{equation*}
G(\bar{\omega}^j)=-\pi^j \: \gfp{\tfrac{j}{p-1}}
\end{equation*}
for $0 \leq j \leq p-2$.
\end{theorem}

We now introduce some notation for a rational number's basic representative modulo $p$.
\begin{defi}
For a prime $p$ and $\frac{a}{b} \in \mathbb{Q} \cap \mathbb{Z}_p^*$ we define $rep_p(\frac{a}{b}) \in \{0,1,\cdots, p-1\}$ such that
\begin{equation*}
rep_p(\tfrac{a}{b}) \equiv \tfrac{a}{b} \pmod p.
\end{equation*}
\end{defi}
\noindent We will drop the subscript $p$ when it is clear from the context. We have the following basic properties of $rep(\cdot)$.

\begin{prop}\label{prop_repOneminus}
Let $p$ be a prime with $1\leq m < d < p$. Then
\begin{equation*}
rep(1-\tfrac{m}{d})=p+1-rep(\tfrac{m}{d}).
\end{equation*}
\end{prop}

\begin{proof}
As $1\leq m <p$, $m \not\equiv 0 \pmod p$ and hence $rep(\tfrac{m}{d}) \neq 0$. Also $rep(\tfrac{m}{d}) \neq 1$, as otherwise $m \equiv d \pmod p$ which contradicts $m<d<p$. Therefore $2\leq rep(\tfrac{m}{d}) \leq p-1$ and $2\leq p+1-rep(\tfrac{m}{d}) \leq p-1$. Also
\begin{equation*}
p+1-rep(\tfrac{m}{d}) \equiv 1-rep(\tfrac{m}{d}) \equiv 1- \tfrac{m}{d} \pmod p.
\end{equation*}
\end{proof}

\begin{lemma}\label{lem_repGenformula}
Let $p$ be a prime with $1\leq m < d < p$. If $p\equiv a \pmod d$ then
\begin{equation*}
rep(\tfrac{m}{d})=\frac{pt+m}{d}
\end{equation*}
where $t$ is the smallest positive integer such that $d\mid ta+m$.
\end{lemma}

\begin{proof}
Certainly $\frac{pt+m}{d} \equiv \frac{m}{d} \pmod{p}$. By assumption $0\leq\frac{pt+m}{d} \in \mathbb{Z}$. Note $\frac{pt+m}{d} < p$ if and only if $1\leq m<p(d-t)$. Therefore $\frac{pt+m}{d} < p$ as $m<p$ and $d>t$. If $d=t$ then $d\mid m$ but $m<d$. If $d>t$ then $d=qt+r$ for some $q,r \in \mathbb{Z}$ with $q\geq1$ and  $0\leq r <t$. Then  $d\mid ra+m$ contradicting the choice of $t$.
\end{proof}

\begin{cor}\label{cor_repGenformula}
Let $p$ be a prime such that $p\equiv a \pmod d$ with $a< d < p$. Then
\begin{equation*}
rep(\tfrac{a}{d})=p-\lfloor \tfrac{p-1}{d} \rfloor
\end{equation*}
and
\begin{equation*}
rep(\tfrac{d-a}{d})=\lfloor \tfrac{p-1}{d} \rfloor +1 \;.
\end{equation*}
\end{cor}

\begin{proof}
From Lemma \ref{lem_repGenformula} with $t=d-1$ we see that $rep(\tfrac{a}{d})=p-\tfrac{p-a}{d} =p-\lfloor \tfrac{p-1}{d} \rfloor$. The second result follows from Proposition \ref{prop_repOneminus}.
\end{proof}

\noindent We now use these properties to develop further results concerning the $p$-adic gamma function.
\begin{prop}\label{prop_gammapmd}
For a prime $p$ with $1\leq m < d < p$,
\begin{equation*}
\gfp{\tfrac{m}{d}+j}=
\begin{cases}
(-1)^j \; \biggfp{\frac{m}{d}} \ph{\frac{m}{d}}{j} & \qquad \text{if } 0\leq j\leq p- rep(\tfrac{m}{d}),\\[6pt]
(-1)^j \; \biggfp{\frac{m}{d}} \ph{\frac{m}{d}}{j} {\bigl({\tfrac{m}{d}+p-rep(\tfrac{m}{d})}\bigr)}^{-1} & \qquad \text{if } p- rep(\tfrac{m}{d})+1\leq j\leq p-1.\\
\end{cases}
\end{equation*} 
\end{prop}

\begin{proof}
For $j=0$ the result is trivial. Assume $j>0$. For $0\leq t \leq j-1$,
\begin{equation*}
\tfrac{m}{d} +t \in p \hspace{1pt} \mathbb{Z}_p \Longleftrightarrow rep(\tfrac{m}{d}) +t \in p \hspace{1pt} \mathbb{Z}_p \Longleftrightarrow rep(\tfrac{m}{d}) +t = p \Longleftrightarrow t = p -rep(\tfrac{m}{d}).
\end{equation*}
Using Proposition \ref{prop_pGamma} (1) the result follows.\\
\end{proof}

\begin{lemma}\label{lem_gammapd2}
For a prime $p$ with $1<d<p$ and $\phi(d)\leq 2$,
\begin{equation*}
\gfp{\tfrac{1}{d}+j}\gfp{\tfrac{d-1}{d}+j}=
\begin{cases}
\biggfp{\frac{1}{d}} \biggfp{\frac{d-1}{d}} \ph{\frac{1}{d}}{j} \ph{\frac{d-1}{d}}{j} &  \text{if }0\leq j\leq \lfloor \frac{p-1}{d}\rfloor, \\[6pt]
\biggfp{\frac{1}{d}} \biggfp{\frac{d-1}{d}} \ph{\frac{1}{d}}{j} \ph{\frac{d-1}{d}}{j} \bigl(\frac{d}{p}\bigr) &  \text{if }\lfloor \frac{p-1}{d}\rfloor+1\leq j\leq p-\lfloor \frac{p-1}{d}\rfloor-1,\\[6pt]
\biggfp{\frac{1}{d}} \biggfp{\frac{d-1}{d}} \ph{\frac{1}{d}}{j} \ph{\frac{d-1}{d}}{j} \bigl(\frac{d^2}{(d-1)p^2}\bigr)& \text{if }p-\lfloor \frac{p-1}{d}\rfloor\leq j\leq p-1.\\[6pt]
\end{cases}
\end{equation*} 
\end{lemma}

\begin{proof}
For $d=2$ we need only show 
\begin{equation*}
\gfp{\tfrac{1}{2}+j}^2=
\begin{cases}
\biggfp{\frac{1}{2}}^2 \ph{\frac{1}{2}}{j}^2  & \quad \text{if } 0\leq j\leq  \frac{p-1}{2}, \\[6pt]
\biggfp{\frac{1}{2}}^2 \ph{\frac{1}{2}}{j}^2 \bigl(\frac{2}{p}\bigr)^2 & \quad \text{if } \frac{p+1}{2}\leq j\leq p-1.\\
\end{cases}
\end{equation*} 
This follows from Proposition \ref{prop_gammapmd} and the fact that $rep(\tfrac{1}{2})= \tfrac{p+1}{2}$.
We now consider when $\phi(d)=2$. If $p\equiv a \pmod d$ then $a \in \{1,d-1\}$. Therefore
\begin{equation*}
\gfp{\tfrac{1}{d}+j}\gfp{\tfrac{d-1}{d}+j}=\gfp{\tfrac{a}{d}+j}\gfp{\tfrac{d-a}{d}+j}.
\end{equation*}
By Proposition \ref{prop_gammapmd} and Corollary \ref{cor_repGenformula} we see that
\begin{equation*}
\gfp{\tfrac{a}{d}+j}\gfp{\tfrac{d-a}{d}+j}
=
\begin{cases}
\biggfp{\frac{1}{d}} \biggfp{\frac{d-1}{d}} \ph{\frac{1}{d}}{j} \ph{\frac{d-1}{d}}{j} & \text{if }0\leq j\leq \lfloor \frac{p-1}{d}\rfloor, \\[6pt]
\biggfp{\frac{1}{d}} \biggfp{\frac{d-1}{d}} \ph{\frac{1}{d}}{j} \ph{\frac{d-1}{d}}{j} \bigl(\frac{d}{p}\bigr) &  \text{if }\lfloor \frac{p-1}{d}\rfloor+1\leq j\leq p-\lfloor \frac{p-1}{d}\rfloor-1,\\[6pt]
\biggfp{\frac{1}{d}} \biggfp{\frac{d-1}{d}} \ph{\frac{1}{d}}{j} \ph{\frac{d-1}{d}}{j} \bigl(\frac{d^2}{(d-1)p^2}\bigr)& \text{if } p-\lfloor \frac{p-1}{d}\rfloor\leq j\leq p-1.\\[6pt]
\end{cases}
\end{equation*} 
\end{proof}

\noindent For $i$, $n \in \mathbb{N}$, we define the generalised harmonic sums, ${H}^{(i)}_{n}$, by
\begin{equation*}
{H}^{(i)}_{n}:= \sum^{n}_{j=1} \frac{1}{j^i}
\end{equation*}
and ${H}^{(i)}_{0}:=0$. We can now use the above to develop some congruences for use in Section 4.

\begin{lemma}\label{lem_ProdGammap2}
Let $p\geq 7$ be a prime with $1\leq m < d < p$. Choose $m_1 \in \{m, d-m\}$ such that $rep(\frac{m_1}{d}) = Max\left(rep(\frac{m}{d}) ,rep(1-\frac{m}{d})\right)$ and $m_2=d-m_1$. Then for $j<rep(\frac{m_1}{d})$,
\begin{multline*}
\frac{\gfp{\frac{m}{d}+j}\gfp{1-\frac{m}{d}+j}}{\gfp{\tfrac{m}{d}}\gfp{1-\tfrac{m}{d}}{j!}^2} \equiv (-1)^{j} \bin{rep(\frac{m_1}{d})-1+j}{j} \bin{rep(\frac{m_1}{d})-1}{j}\Bigl[\delta\Bigr]\\
\cdot \left[1-\Bigl(rep(\tfrac{m_1}{d})-\tfrac{m_1}{d}\Bigr)\left(H_{rep(\frac{m_1}{d})-1+j}^{(1)}-H_{rep(\frac{m_2}{d})-1+j}^{(1)}-\delta \right)\right]
\pmod{p^2}
\end{multline*}
where 
\begin{align*}
\delta=
\begin{cases}
1 & \text{if } 0\leq j \leq rep(\frac{m_2}{d})-1,\\
\frac{1}{p} & \text{if }rep(\frac{m_2}{d}) \leq j < rep(\frac{m_1}{d}).
\end{cases}
\end{align*}
\end{lemma}

\begin{proof}
By Proposition \ref{prop_pGammaCong} (2) we get that
\begin{align*}
&\gfp{\tfrac{m}{d}+j}\gfp{1-\tfrac{m}{d}+j}\\[3pt]
&=\gfp{\tfrac{m_1}{d}+j}\gfp{\tfrac{m_2}{d}+j}\\[3pt]
& \equiv
\left[\gfp{rep(\tfrac{m_1}{d})+j}-\left(rep(\tfrac{m_1}{d})-\tfrac{m_1}{d}\right)
\Gamma_p^{\prime}\left(rep(\tfrac{m_1}{d})+j\right) \right]\\[3pt]
& \quad \; \cdot
\left[\gfp{rep(\tfrac{m_2}{d})-p+j}-\left(rep(\tfrac{m_2}{d})-p-\tfrac{m_2}{d}\right)
\Gamma_p^{\prime}\left(rep(\tfrac{m_2}{d})+j\right) \right]\\[3pt]
& \equiv \gfp{rep(\tfrac{m_1}{d})+j} \gfp{1-rep(\tfrac{m_1}{d})+j} -
\left(rep(\tfrac{m_1}{d})-\tfrac{m_1}{d}\right)\\[3pt]
& \quad \; \cdot
\left[\Gamma_p^{\prime}\left(rep(\tfrac{m_1}{d})+j\right) \gfp{rep(\tfrac{m_2}{d})-p+j}
-\Gamma_p^{\prime}\left(rep(\tfrac{m_2}{d})+j\right) \gfp{rep(\tfrac{m_1}{d})+j} \right]
\pmod{p^2} ,
\end{align*}
as $rep(\tfrac{m_2}{d})-p-\tfrac{m_2}{d}=-\left(rep(\tfrac{m_1}{d})-\tfrac{m_1}{d}\right)$ and $rep(\tfrac{m_2}{d})-p=1-rep(\tfrac{m_1}{d})$ by Proposition \ref{prop_repOneminus}.\\[3pt]
Using Proposition \ref{prop_pGamma} (3) we have
\begin{align*}
\Gamma_p^{\prime}&\left(rep(\tfrac{m_1}{d})+j\right) \gfp{rep(\tfrac{m_2}{d})-p+j}
-\Gamma_p^{\prime}\left(rep(\tfrac{m_2}{d})+j\right) \gfp{rep(\tfrac{m_1}{d})+j}\\[3pt]
&\equiv
\gfp{rep(\tfrac{m_1}{d})+j} \gfp{rep(\tfrac{m_2}{d})+j}
\left[G_1\left(rep(\tfrac{m_1}{d})+j\right)  -G_1\left(rep(\tfrac{m_2}{d})+j\right)\right]\\[3pt]
&\equiv
\gfp{rep(\tfrac{m_1}{d})+j} \gfp{1-rep(\tfrac{m_1}{d})+j}
\left[G_1\left(rep(\tfrac{m_1}{d})+j\right)  -G_1\left(rep(\tfrac{m_2}{d})+j\right)\right]
\pmod{p}.
\end{align*}
So
\begin{multline*}
\gfp{\tfrac{m}{d}+j}\gfp{1-\tfrac{m}{d}+j}
\equiv
\gfp{rep(\tfrac{m_1}{d})+j} \gfp{1-rep(\tfrac{m_1}{d})+j}\left[1-\left(rep(\tfrac{m_1}{d})-\tfrac{m_1}{d}\right)
\right. \\[3pt] \left.
\left(G_1\left(rep(\tfrac{m_1}{d})+j\right)  -G_1\left(rep(\tfrac{m_2}{d})+j\right)\right)\right]
\pmod{p^2}.
\end{multline*}
By Proposition \ref{prop_pGammaG} (1),
\begin{equation*}
G_1\left(rep(\tfrac{m_1}{d})+j\right)  -G_1\left(rep(\tfrac{m_2}{d})+j\right)
= H_{rep(\frac{m_1}{d})-1+j}^{(1)}-H_{rep(\frac{m_2}{d})-1+j}^{(1)}-\delta.
\end{equation*}
For $j<rep(\frac{m_1}{d})$, Proposition \ref{prop_pGamma} gives us
\begin{align*}
\gfp{rep(\tfrac{m_1}{d})+j} \gfp{1-rep(\tfrac{m_1}{d})+j}
&= (-1)^{rep(\frac{m_1}{d})-j} \frac{\gfp{rep(\tfrac{m_1}{d})+j}}{\gfp{rep(\tfrac{m_1}{d})-j}}\\[3pt]
&=(-1)^{rep(\frac{m_1}{d})-j} \frac{\left(rep(\tfrac{m_1}{d})-1+j\right)!}{\left(rep(\tfrac{m_1}{d})-1-j\right)!} \;\Big[ \delta \Bigr].
\end{align*}
Noting that
\begin{equation*}
\frac{\left(rep(\tfrac{m_1}{d})-1+j\right)!}{\left(rep(\tfrac{m_1}{d})-1-j\right)! \hspace{2pt} {j!}^2}
= \bin{rep(\frac{m_1}{d})-1+j}{j} \bin{rep(\frac{m_1}{d})-1}{j} \;
\end{equation*}
and
\begin{equation*}
\gfp{\tfrac{m}{d}}\gfp{1-\tfrac{m}{d}}=(-1)^{rep(\frac{m_1}{d})}
\end{equation*}
gives the required result.\\
\end{proof}

\begin{lemma}\label{lem_SumGammap2}
Let $p\geq 7$ be a prime with $1\leq m < d < p$. Choose $m_1 \in \{m, d-m\}$ such that $rep(\frac{m_1}{d}) = Max\left(rep(\frac{m}{d}) ,rep(1-\frac{m}{d})\right)$ and $m_2=d-m_1$. Then for $j<rep(\frac{m_1}{d})$,
\begin{multline*}
G_1\left(\tfrac{m}{d}+j\right)+G_1\left(1-\tfrac{m}{d}+j\right) -2G_1(1+j)
\equiv
H_{rep(\frac{m_1}{d})-1+j}^{(1)}+H_{rep(\frac{m_1}{d})-1-j}^{(1)}-2\hspace{1pt} H_{j}^{(1)}-\alpha\\[3pt]
+\left(rep(\tfrac{m_1}{d})-\tfrac{m_1}{d}\right) 
\left(H_{rep(\frac{m_1}{d})-1+j}^{(2)}-H_{rep(\frac{m_2}{d})-1+j}^{(2)}-\beta \right)
\pmod{p^2}
\end{multline*}
where 
\begin{align*}
\alpha=
\begin{cases}
0 & \text{if }0\leq j \leq rep(\frac{m_2}{d})-1,\\
\frac{1}{p} &\text{if } rep(\frac{m_2}{d}) \leq j < rep(\frac{m_1}{d}).
\end{cases}
& \qquad and \qquad
\beta=
\begin{cases}
0 &\text{if } 0\leq j \leq rep(\frac{m_2}{d})-1,\\
\frac{1}{p^2} &\text{if } rep(\frac{m_2}{d}) \leq j < rep(\frac{m_1}{d}).
\end{cases}
\end{align*}\\[-12pt]
\end{lemma}

\begin{proof}
Using Corollary \ref{cor_pGammaCong1} and Corollary \ref{cor_pGammaCong2} we see that
\begin{multline*}
G_1\left(\tfrac{m_1}{d}+j\right)\equiv G_1\left(rep(\tfrac{m_1}{d})+j\right)\\+\left(rep(\tfrac{m_1}{d})-\tfrac{m_1}{d}\right)
\left(G_1\left(rep(\tfrac{m_1}{d_1})+j\right)^2-G_2\left(rep(\tfrac{m_1}{d})+j\right)\right)
\pmod{p^2}
\end{multline*}
and
\begin{multline*}
G_1\left(\tfrac{m_2}{d}+j\right)
\equiv G_1\left(rep(\tfrac{m_2}{d})-p+j\right)\\
+\left(rep(\tfrac{m_2}{d})-p-\tfrac{m_2}{d}\right)
\left(G_1\left(rep(\tfrac{m_2}{d_1})+j\right)^2-G_2\left(rep(\tfrac{m_2}{d})+j\right)\right)
\pmod{p^2}.\\[-6pt]
\end{multline*}

\noindent We note again that $rep(\tfrac{m_2}{d_1})-p=1-rep(\tfrac{m_1}{d_1})$ and $rep(\tfrac{m_2}{d})-p-\tfrac{m_2}{d}=-(rep(\tfrac{m_1}{d})-\tfrac{m_1}{d})$.
Therefore, using Proposition \ref{prop_pGammaG} (3),
\begin{align*}
&G_1\left(\tfrac{m}{d}+j\right)+G_1\left(1-\tfrac{m}{d}+j\right)
\\[4pt]&\;
=G_1\left(\tfrac{m_1}{d}+j\right)+G_1\left(\tfrac{m_2}{d}+j\right)\\[4pt]
& \; \equiv
G_1\left(rep(\tfrac{m_1}{d})+j\right)+G_1\left(rep(\tfrac{m_1}{d})-j\right)+
\left(rep(\tfrac{m_1}{d})-\tfrac{m_1}{d}\right)
\\[4pt]& \;
\left(G_1\left(rep(\tfrac{m_1}{d_1})+j\right)^2-G_2\left(rep(\tfrac{m_1}{d})+j\right)
-G_1\left(rep(\tfrac{m_2}{d_1})+j\right)^2+G_2\left(rep(\tfrac{m_2}{d})+j\right)\right)
\pmod{p^2}.
\end{align*}

\pagebreak
\noindent By Proposition \ref{prop_pGammaG} (1), (2) we get that
\begin{align*}
G_1\left(rep(\tfrac{m_1}{d})+j\right)-G_1(1+j)
=
H_{rep(\frac{m_1}{d})-1+j}^{(1)}- H_{j}^{(1)}-
\alpha,
\end{align*}
\begin{multline*}
G_1\left(rep(\tfrac{m_1}{d_1})+j\right)^2-G_2\left(rep(\tfrac{m_1}{d})+j\right)
-G_1\left(rep(\tfrac{m_2}{d_1})+j\right)^2+G_2\left(rep(\tfrac{m_2}{d})+j\right)\\
=
H_{rep(\frac{m_1}{d})-1-j}^{(2)}-H_{rep(\frac{m_2}{d})-1-j}^{(2)}-\beta
\end{multline*}
and, for $0\leq j < rep(\tfrac{m_1}{d})$,
\begin{align*}
G_1\left(rep(\tfrac{m_1}{d})-j\right)-G_1(1+j)
=
H_{rep(\frac{m_1}{d})-1-j}^{(1)}- H_{j}^{(1)}.
\end{align*}
The result follows.
\end{proof}



\section{Binomial Coefficient-Harmonic Sum Identities}
We establish two new binomial coefficient-harmonic sum identities using the partial fraction decomposition method of Chu (see \cite{C} or \cite{C1} for example). We first develop two algebraic identities of which the binomial coefficient-harmonic sum identities are limiting cases.

\begin{theorem}\label{Thm_BinHarId1}
Let $x$ be an indeterminate and $m,n$ positive integers with $m\geq n$. Then
\begin{multline}\label{BinHarIdx1}
\frac{x \ph{1-x}{n} \ph{1-x}{m}}{\ph{x}{n+1} \ph{x}{m+1}}= \frac{1}{x} + \\
\sum_{k=1}^{n} \biggbin{m+k}{k} \biggbin{m}{k} \biggbin{n+k}{k} \biggbin{n}{k}
 \left\{ {\frac{-k}{(x+k)^2} + \frac{1+k \left(H_{m+k}^{(1)} +H_{m-k}^{(1)} + H_{n+k}^{(1)} + H_{n-k}^{(1)} -4H_k^{(1)}\right)}{x+k}} \right\} \\
 +\sum_{k=n+1}^{m} \frac{(-1)^{k-n}}{x+k} \biggbin{m+k}{k} \biggbin{m}{k} \biggbin{n+k}{k} \Big/ \biggbin{k-1}{n} .
\end{multline}
\end{theorem}

\begin{proof}
Using partial fraction decomposition we can write
\begin{equation*}
f(x):= \frac{x \ph{1-x}{n} \ph{1-x}{m}}{\ph{x}{n+1} \ph{x}{m+1}}= \frac{A}{x} + \sum_{k=1}^{n}
\biggl\{ \frac{B_k}{(x+k)^2} + \frac{C_k}{x+k} \biggr\} +  \sum_{k=n+1}^{m} \frac{D_k}{x+k}.
\end{equation*}

\noindent We now isolate the coefficients $A, B_k, C_k$ and $D_k$ by taking various limits of $f(x)$ as follows.
\begin{equation*}
A = \lim_{x \to 0} x f(x) 
= \lim_{x \to 0}\frac{\ph{1-x}{n} \ph{1-x}{m}}{\ph{1+x}{n} \ph{1+x}{m}}
=1.
\end{equation*}
For $1 \leq k \leq n$,
\begin{align*}
B_k &= \lim_{x \to-k} (x+k)^2 f(x)
=\lim_{x \to-k} \frac{x \ph{1-x}{n} \ph{1-x}{m}}{\ph{x}{k}^2 \ph{x+k+1}{n-k} \ph{x+k+1}{m-k}}\\[9pt]
&=\frac{-k \ph{k+1}{n} \ph{k+1}{m}}{\ph{-k}{k}^2 \ph{1}{n-k} \ph{1}{m-k}}
=-k  \biggbin{m+k}{k} \biggbin{m}{k} \biggbin{n+k}{k} \biggbin{n}{k},
\end{align*}
and, using L'H\^{o}pital's rule,
\begin{align*}
C_k &=  \lim_{x \to-k} \frac{(x+k)^2 f(x) - B_k}{x+k}\\[6pt]
&=\lim_{x \to-k} \frac{d}{dx}\Biggl[(x+k)^2 f(x)\Biggr]\\[6pt]
&=\lim_{x \to-k}  \frac{d}{dx}\left[\frac{x \ph{1-x}{n} \ph{1-x}{m}}{\ph{x}{k}^2 \ph{x+k+1}{n-k} \ph{x+k+1}{m-k}}\right]\\[6pt]
&=\lim_{x \to-k} \left[\frac{ \ph{1-x}{n} \ph{1-x}{m}}{\ph{x}{k}^2 \ph{x+k+1}{n-k} \ph{x+k+1}{m-k}}\right]
\left[1-x \left(\sum_{s=1}^{n}(-x+s)^{-1}+\sum_{s=1}^{m}(-x+s)^{-1}\right.\right.\\[6pt]
&\qquad \qquad \qquad \qquad \qquad \qquad \qquad \left.\left.+\sum_{s=1}^{n-k}(x+k+s)^{-1}+\sum_{s=1}^{m-k}(x+k+s)^{-1}+2\sum_{s=0}^{k-1}(x+s)^{-1}\right)\right]\\[6pt]
&=\left[\frac{ \ph{1+k}{n} \ph{1+k}{m}}{\ph{-k}{k}^2 \ph{1}{n-k} \ph{1}{m-k}}\right]
\left[1+k \left(\sum_{s=1}^{n}(k+s)^{-1}+\sum_{s=1}^{m}(k+s)^{-1}\right.\right.\\[6pt]
&\qquad \qquad \qquad\qquad \qquad\qquad \qquad\qquad \qquad \qquad 
\left.\left.+\sum_{s=1}^{n-k}(s)^{-1}+\sum_{s=1}^{m-k}(s)^{-1}+2\sum_{s=0}^{k-1}(-k+s)^{-1}\right)\right]\\[6pt]
&= \biggbin{m+k}{k} \biggbin{m}{k} \biggbin{n+k}{k} \biggbin{n}{k} \left[1+k \left(H_{m+k}^{(1)} +H_{m-k}^{(1)} + H_{n+k}^{(1)} + H_{n-k}^{(1)} -4H_k^{(1)}\right)\right].\\[-4pt]
\end{align*}

\noindent Similarly, for $n+1 \leq k \leq m$,
\begin{align*}
D_k &= \lim_{x \to-k} (x+k) f(x)
=\lim_{x \to-k} \frac{x \ph{1-x}{n} \ph{1-x}{m}}{\ph{x}{n+1} \ph{x}{k} \ph{x+k+1}{m-k}}\\[6pt]
&=\frac{-k \ph{k+1}{n} \ph{k+1}{m}}{\ph{-k}{n+1} \ph{-k}{k} \ph{1}{m-k}}
=(-1)^{k-n}  \biggbin{m+k}{k} \biggbin{m}{k} \biggbin{n+k}{k} \Big/ \biggbin{k-1}{n}.
\end{align*}
\end{proof}

\begin{cor}\label{Cor_BinHarId1}
Let $m,n$ be positive integers with $m\geq n$. Then
\begin{multline*}
(-1)^{m+n}  =
\sum_{k=0}^{n} \biggbin{m+k}{k} \biggbin{m}{k} \biggbin{n+k}{k} \biggbin{n}{k}
 \biggl[ 1+k \left(H_{m+k}^{(1)} +H_{m-k}^{(1)} + H_{n+k}^{(1)} + H_{n-k}^{(1)} -4H_k^{(1)}\right) \biggr]\\
\notag +\sum_{k=n+1}^{m} (-1)^{k-n} \biggbin{m+k}{k} \biggbin{m}{k} \biggbin{n+k}{k} \Big/ \biggbin{k-1}{n} .
\end{multline*}
\end{cor}

\begin{proof}
Multiply both sides of (\ref{BinHarIdx1}) by $x$ and take the limit as $x \to \infty$.
\end{proof}

\begin{theorem}\label{Thm_BinHarId2}
Let $x$ be an indeterminate, $p,m, n$ positive integers with $p\geq m\geq n\geq\frac{p}{2}$ and $C_1, C_2 \in \mathbb{Q}$ some constants. Then
\begin{multline}\label{BinHarIdx2}
\frac{x \ph{1-x}{n} \ph{1-x}{m}}{\ph{x}{n+1} \ph{x}{m+1}}
\left[C_1 \sum_{s=p-n}^{n} (-x+s)^{-1} + C_2 \sum_{s=p-m}^{m} (-x+s)^{-1} \right]\\[6pt]
\shoveleft = \frac{C_1\left(H_n^{(1)} - H_{p-n-1}^{(1)}\right)  + C_2 \left(H_m^{(1)} - H_{p-m-1}^{(1)}\right)}{x}\\[6pt]
+ \sum_{k=1}^{n} \biggbin{m+k}{k} \biggbin{m}{k} \biggbin{n+k}{k} \biggbin{n}{k}
 \left\{ \frac{-k\left(C_1\left(H_{k+n}^{(1)} - H_{k+p-n-1}^{(1)}\right)  + C_2 \left(H_{k+m}^{(1)} - H_{k+p-m-1}^{(1)}\right)\right)}{(x+k)^2}
\right. \\[6pt] \left. 
+ \frac{\left(1+k \left(H_{m+k}^{(1)} +H_{m-k}^{(1)} + H_{n+k}^{(1)} + H_{n-k}^{(1)} -4H_k^{(1)}\right)\right)
\left(C_1\left(H_{k+n}^{(1)} - H_{k+p-n-1}^{(1)}\right) \right)}{x+k} 
\right.\\[6pt] \left.
+ \frac{\left(1+k \left(H_{m+k}^{(1)} +H_{m-k}^{(1)} + H_{n+k}^{(1)} + H_{n-k}^{(1)} -4H_k^{(1)}\right)\right)
\left( C_2 \left(H_{k+m}^{(1)} - H_{k+p-m-1}^{(1)}\right)\right)}{x+k} 
\right.\\[6pt] \left.
\shoveright \qquad \qquad \qquad \qquad \qquad \qquad
- \frac{k\left(C_1\left(H_{k+n}^{(2)} - H_{k+p-n-1}^{(2)}\right)  + C_2 \left(H_{k+m}^{(2)} - H_{k+p-m-1}^{(2)}\right)\right)}{x+k}\right\} \\[6pt]
\shoveleft +\sum_{k=n+1}^{m} \frac{(-1)^{k-n}}{x+k} \biggbin{m+k}{k} \biggbin{m}{k} \biggbin{n+k}{k} \Big/ \biggbin{k-1}{n} 
\left(C_1\left(H_{k+n}^{(1)} - H_{k+p-n-1}^{(1)}\right)
\right.\\ \left.
  + C_2 \left(H_{k+m}^{(1)} - H_{k+p-m-1}^{(1)}\right)\right).
\end{multline}
\end{theorem}

\begin{proof}
Using partial fraction decomposition we can write
\begin{align*}
f(x):&=\frac{x \ph{1-x}{n} \ph{1-x}{m}}{\ph{x}{n+1} \ph{x}{m+1}}
\left[C_1 \sum_{s=p-n}^{n} (-x+s)^{-1} + C_2 \sum_{s=p-m}^{m} (-x+s)^{-1} \right]\\[6pt]
&=\frac{A}{x} + \sum_{k=1}^{n} \biggl\{ \frac{B_k}{(x+k)^2} + \frac{C_k}{x+k} \biggr\} +  \sum_{k=n+1}^{m} \frac{D_k}{x+k}.
\end{align*}
Similar to the proof of Theorem \ref{Thm_BinHarId1} we isolate the coefficients $A, B_k, C_k$ and $D_k$ by taking various limits of $f(x)$ as follows.
\begin{align*}
A = \lim_{x \to 0} x f(x) 
&= \lim_{x \to 0}\frac{\ph{1-x}{n} \ph{1-x}{m}}{\ph{1+x}{n} \ph{1+x}{m}}
\left[C_1 \sum_{s=p-n}^{n} (-x+s)^{-1} + C_2 \sum_{s=p-m}^{m} (-x+s)^{-1} \right]\\[6pt]
&= C_1 \lim_{x \to 0} \sum_{s=p-n}^{n} \frac{\ph{1-x}{n} \ph{1-x}{m}}{\ph{1+x}{n} \ph{1+x}{m} (s-x)}
+ C_2 \lim_{x \to 0} \sum_{s=p-m}^{m} \frac{\ph{1-x}{n} \ph{1-x}{m}}{\ph{1+x}{n} \ph{1+x}{m} (s-x)}\\[6pt]
&=C_1\left(H_n^{(1)} - H_{p-n-1}^{(1)}\right)  + C_2 \left(H_m^{(1)} - H_{p-m-1}^{(1)}\right).
\end{align*}

\noindent For $1 \leq k \leq n$,
\begin{align*}
B_k &= \lim_{x \to-k} (x+k)^2 f(x)\\
&=\lim_{x \to-k} \frac{x \ph{1-x}{n} \ph{1-x}{m}}{\ph{x}{k}^2 \ph{x+k+1}{n-k} \ph{x+k+1}{m-k}}
\left[C_1 \sum_{s=p-n}^{n} (-x+s)^{-1} + C_2 \sum_{s=p-m}^{m} (-x+s)^{-1} \right]\\[6pt]
&=\frac{-k \ph{k+1}{n} \ph{k+1}{m}}{\ph{-k}{k}^2 \ph{1}{n-k} \ph{1}{m-k}}
\left[C_1 \sum_{s=p-n}^{n} (k+s)^{-1} + C_2 \sum_{s=p-m}^{m} (k+s)^{-1} \right]\\[6pt]
&=-k  \biggbin{m+k}{k} \biggbin{m}{k} \biggbin{n+k}{k} \biggbin{n}{k}
\Biggl[C_1\left(H_{k+n}^{(1)} - H_{k+p-n-1}^{(1)}\right)  + C_2 \left(H_{k+m}^{(1)} - H_{k+p-m-1}^{(1)}\right) \Biggr],
\end{align*}
and
\begin{align*}
C_k &=\lim_{x \to-k} \frac{d}{dx}\Biggl[(x+k)^2 f(x)\Biggr]\\[6pt]
&=\lim_{x \to-k}  \frac{d}{dx}\left[\frac{x \ph{1-x}{n} \ph{1-x}{m}}{\ph{x}{k}^2 \ph{x+k+1}{n-k} \ph{x+k+1}{m-k}}\right]
\left[C_1 \sum_{s=p-n}^{n} (-x+s)^{-1} + C_2 \sum_{s=p-m}^{m} (-x+s)^{-1} \right]\\[6pt]
&=\lim_{x \to-k} \left[\frac{x \ph{1-x}{n} \ph{1-x}{m}}{\ph{x}{k}^2 \ph{x+k+1}{n-k} \ph{x+k+1}{m-k}}\right]
\left[C_1 \sum_{s=p-n}^{n} (-x+s)^{-2} + C_2 \sum_{s=p-m}^{m} (-x+s)^{-2} \right]\\[6pt]
&\quad \;+\left[C_1 \sum_{s=p-n}^{n} (-x+s)^{-1} + C_2 \sum_{s=p-m}^{m} (-x+s)^{-1} \right]
\left[\frac{ \ph{1-x}{n} \ph{1-x}{m}}{\ph{x}{k}^2 \ph{x+k+1}{n-k} \ph{x+k+1}{m-k}}\right]\\[6pt]
&\quad \; \left[1-x \left(\sum_{s=1}^{n}(-x+s)^{-1}+\sum_{s=1}^{m}(-x+s)^{-1}
+\sum_{s=1}^{n-k}(x+k+s)^{-1}+\sum_{s=1}^{m-k}(x+k+s)^{-1}
\right. \right. \\ &\qquad \qquad \qquad \qquad \qquad \qquad \qquad \qquad \qquad \qquad \qquad \qquad \qquad \qquad \qquad  \quad \left. \left.
+2\sum_{s=0}^{k-1}(x+s)^{-1}\right)\right]\\[6pt]
&=\left[\frac{ \ph{1+k}{n} \ph{1+k}{m}}{\ph{-k}{k}^2 \ph{1}{n-k} \ph{1}{m-k}}\right]
\left[-k \left(C_1 \sum_{s=p-n}^{n} (k+s)^{-2} + C_2 \sum_{s=p-m}^{m} (k+s)^{-2} \right)+\right.\\[6pt]
& \quad \; \left. \left(C_1 \sum_{s=p-n}^{n} (k+s)^{-1} + C_2 \sum_{s=p-m}^{m} (k+s)^{-1} \right)
\left(1+k \left(\sum_{s=1}^{n}(k+s)^{-1}+\sum_{s=1}^{m}(k+s)^{-1} \right.\right.\right.\\[6pt]
&\qquad \qquad \qquad \qquad \qquad \qquad \qquad \qquad \qquad  \qquad\left.\left.\left.
+\sum_{s=1}^{n-k}(s)^{-1} +\sum_{s=1}^{m-k}(s)^{-1}+2\sum_{s=0}^{k-1}(-k+s)^{-1}\right)\right)\right]\\[6pt]
&= \biggbin{m+k}{k} \biggbin{m}{k} \biggbin{n+k}{k} \biggbin{n}{k} 
\biggl[-k\left(C_1\left(H_{k+n}^{(2)} - H_{k+p-n-1}^{(2)}\right)  + C_2 \left(H_{k+m}^{(2)} - H_{k+p-m-1}^{(2)}\right)\right) \\[6pt]
&\quad\;+\left(1+k \left(H_{m+k}^{(1)} +H_{m-k}^{(1)} + H_{n+k}^{(1)} + H_{n-k}^{(1)} -4H_k^{(1)}\right)\right)
\\[6pt] &\quad \; \qquad \qquad \qquad \qquad  \qquad \qquad \qquad \qquad 
\left(C_1(H_{k+n}^{(1)} - H_{k+p-n-1}^{(1)})  + C_2 (H_{k+m}^{(1)} - H_{k+p-m-1}^{(1)})\right)
\biggr].
\end{align*}

\noindent For $n+1 \leq k \leq m$,
\begin{align*}
D_k &= \lim_{x \to-k} (x+k) f(x)\\
&=\lim_{x \to-k} \frac{x \ph{1-x}{n} \ph{1-x}{m}}{\ph{x}{n+1} \ph{x}{k} \ph{x+k+1}{m-k}}
\left[C_1 \sum_{s=p-n}^{n} (-x+s)^{-1} + C_2 \sum_{s=p-m}^{m} (-x+s)^{-1} \right]\\[6pt]
&=\frac{-k \ph{k+1}{n} \ph{k+1}{m}}{\ph{-k}{n+1} \ph{-k}{k} \ph{1}{m-k}}
\left[C_1 \sum_{s=p-n}^{n} (k+s)^{-1} + C_2 \sum_{s=p-m}^{m} (k+s)^{-1} \right]\\[6pt]
&=(-1)^{k-n}  \biggbin{m+k}{k} \biggbin{m}{k} \biggbin{n+k}{k} \Big/ \biggbin{k-1}{n}
\left(C_1\left(H_{k+n}^{(1)} - H_{k+p-n-1}^{(1)}\right) 
\right. \\[6pt] & \left. \qquad \qquad \qquad \qquad  \qquad \qquad \qquad \qquad \qquad \qquad \qquad \qquad \qquad
 + C_2 \left(H_{k+m}^{(1)} - H_{k+p-m-1}^{(1)}\right)\right).
\end{align*}

\end{proof}

\begin{cor}\label{Cor_BinHarId2}
Let $p,m, n$ be positive integers with $p\geq m\geq n\geq\frac{p}{2}$ and $C_1, C_2 \in \mathbb{Q}$ some constants. Then\begin{align*}
& 0 =  \sum_{k=0}^{n} \biggbin{m+k}{k} \biggbin{m}{k} \biggbin{n+k}{k} \biggbin{n}{k} 
\small \biggl[ \left(1+k \left(H_{m+k}^{(1)} +H_{m-k}^{(1)} + H_{n+k}^{(1)} + H_{n-k}^{(1)} -4H_k^{(1)}\right)\right)
\\[5pt]
&\quad\; \left(C_1\left(H_{k+n}^{(1)} - H_{k+p-n-1}^{(1)}\right)  + C_2 \left(H_{k+m}^{(1)} - H_{k+p-m-1}^{(1)}\right)\right) -k\left(C_1\left(H_{k+n}^{(2)} - H_{k+p-n-1}^{(2)}\right)
\right.\\[5pt]&\notag \qquad \qquad \qquad \qquad \qquad \qquad \qquad \qquad \qquad \qquad \qquad \qquad \qquad \left.
 + C_2 \left(H_{k+m}^{(2)} - H_{k+p-m-1}^{(2)}\right)\right) \biggr]
 \\[5pt]
&\quad\;+ \sum_{k=n+1}^{m} (-1)^{k-n} \biggbin{m+k}{k} \biggbin{m}{k} \biggbin{n+k}{k} \Big/ \biggbin{k-1}{n}
\left(C_1\left(H_{k+n}^{(1)} - H_{k+p-n-1}^{(1)}\right) 
\right. \\[5pt] & \left. \qquad \qquad \qquad \qquad  \qquad \qquad \qquad \qquad \qquad \qquad \qquad \qquad \qquad \quad
+ C_2 \left(H_{k+m}^{(1)} - H_{k+p-m-1}^{(1)}\right)\right) .
\end{align*}
\end{cor}

\begin{proof}
Multiply both sides of (\ref{BinHarIdx2}) by $x$ and take the limit as $x \to \infty$.
\end{proof}


\section{Framework of Congruences}
In \cite{G}, Greene introduced the notion of general hypergeometric series over finite fields or \emph{Gaussian hypergeometric series}. We introduce two definitions from \cite{G}. The first definition is a finite field analogue of the binomial coefficient. For characters $A$ and $B$ of $\mathbb{F}_{p}$, define $\bin{A}{B}$ by
\begin{equation*}\label{FF_Binomial}
\binom{A}{B} := \frac{B(-1)}{p} J(A, \bar{B})= \frac{B(-1)}{p} \sum_{x \in \mathbb{F}_{p}} A(x) \bar{B}(1-x)
\end{equation*}
where $\bar{B}(x)$ is the complex conjugate of $B(x)$. The second definition is a finite field analogue of ordinary hypergeometric series. For characters $A_0,A_1,\dotsc, A_n$ and $B_1, \dotsc, B_n$ on $\mathbb{F}_{p}$ and 
$x \in \mathbb{F}_{p}$, define the Gaussian hypergeometric series by
\begin{equation*}
{_{n+1}F_n} {\left( \begin{array}{cccc} A_0, & A_1, & \dotsc, & A_n \\
\phantom{A_0} & B_1, & \dotsc, & B_n \end{array}
\Big| \; x \right)}_{p}
:= \frac{p}{p-1} \sum_{\chi} \binom{A_0 \chi}{\chi} \binom{A_1 \chi}{B_1 \chi}
\dotsm \binom{A_n \chi}{B_n \chi} \chi(x)
\end{equation*}
where the summation is over all characters $\chi$ on $\mathbb{F}_{p}$.

These series are analogous to classical hypergeometric series and have played an important role in proving many of the supercongruence conjectures already established. The main approach taken has been to form a congruence between the ordinary and Gaussian hypergeometric series and then relate the latter to the other side of the particular supercongruence. One of the main results in this process has been Theorem 1 in \cite{M}.
The wording of this theorem would suggest that the result is valid for any choice of character $\rho_i$ of order $d_i$. However, the proof would indicate that a refinement to the statement of the theorem is necessary which specifies that $\rho_i$ is the character of order $d_i$ on $\mathbb{F}_p$ given by $\bar{\omega}^{\frac{p-1}{d_i}}$. A similar refinement is also required to its corollaries and to Theorem 1 in \cite{M2}, of which Theorem 1 in \cite{M} is a generalisation. The proofs of the supercongruences which rely on these results are still valid due to the fact that $\phi(d)=2$ in each of these cases.

However, many results using this approach are restricted to primes of a certain type (e.g. $p\equiv 1 \pmod d$ in some of the $D=2$ cases described in Section 1). We would like to develop some generalisation which does not have such restrictions. Using the definitions above, the relationship between Jacobi and Gauss sums, and the Gross-Koblitz formula, we can express certain Gaussian hypergeometric series in terms of the $p$-adic gamma function, as follows.

\begin{align*}
&(-1)^n \: p^n \:  {_{n+1}F_n}  {\left( \begin{array}{cccc} \rho_1^{m_1}, & \rho_2^{m_2}, & \dotsc, & \rho_{n+1}^{m_{n+1}} \\
\phantom{\rho_1^{m_1}} & \varepsilon, & \dotsc, & \varepsilon \end{array}
\Big| \; 1 \right)}_{p}\\
&= \frac{-1}{p-1} \sum_{k=0}^{n+1} (-p)^k \sum_{j=m_kr_k+1}^{m_{k+1}r_{k+1}} \omega^{j(n+1)}(-1) 
{\biggfp{\tfrac{j}{p-1}}}^{n+1}
\prod_{\substack{i=1\\i>k}}^{n+1} \frac{\biggfp{\frac{m_i}{d_i}-\frac{j}{p-1}}}{\biggfp{\frac{m_i}{d_i}}}
\prod_{\substack{i=1\\i\leq k}}^{n+1} \frac{\biggfp{\frac{d_i+m_i}{d_i}-\frac{j}{p-1}}}{\biggfp{\frac{m_i}{d_i}}} \; ,
\end{align*}
where $\rho_i$ is the character of order $d_i$ on $\mathbb{F}_p$ given by $\bar{\omega}^{\frac{p-1}{d_i}}$, $0<\frac{m_1}{d_1} \leq \frac{m_2}{d_2} \leq \dots \leq \frac{m_{n+1}}{d_{n+1}}<1$, $r_i:=\frac{p-1}{d_i}$, $m_0:=-1$, $m_{n+2}:=p-2$ and $d_0=d_{n+2}:=p-1$. For $x \in \mathbb{Q}$ we let $\lfloor x \rfloor$ denote the greatest integer less than or equal to $x$. We then define the following generalisation.
\begin{defi}\label{def_GFn}
For $p$ an odd prime and $0<\frac{m_1}{d_1} \leq \frac{m_2}{d_2} \leq \dots \leq \frac{m_{n+1}}{d_{n+1}}<1$,  we define the $p$-adic hypergeometric series, ${_{n+1}G}$, by
\begin{align*}
&{_{n+1}G} \left( \tfrac{m_1}{d_1}, \tfrac{m_2}{d_2}, \dotsc, \tfrac{m_{n+1}}{d_{n+1}} \right)_p\\
&:= \frac{-1}{p-1} \sum_{k=0}^{n+1} (-p)^k \sum_{j=\lfloor m_k r_k \rfloor +1}^{\lfloor m_{k+1}r_{k+1}\rfloor} \omega^{j(n+1)}(-1) 
{\biggfp{\tfrac{j}{p-1}}}^{n+1}
\prod_{\substack{i=1\\i>k}}^{n+1} \frac{\biggfp{\frac{m_i}{d_i}-\frac{j}{p-1}}}{\biggfp{\frac{m_i}{d_i}}}
\prod_{\substack{i=1\\i\leq k}}^{n+1} \frac{\biggfp{\frac{d_i+m_i}{d_i}-\frac{j}{p-1}}}{\biggfp{\frac{m_i}{d_i}}} \; ,
\end{align*}
where $r_i:=\frac{p-1}{d_i}$, $m_0:=-1$, $m_{n+2}:=p-2$, $d_0=d_{n+2}:=p-1$.
\end{defi}
\noindent In practice we can write the arguments of the $G$ function in any order on the understanding that its meaning is based on the function with the arguments listed in ascending order.

It easy to see from the definition of the $G$ function that we have the following relationship with the Gaussian hypergeometric series.
\begin{prop}\label{prop_GtoGHS}
If $p$ is a prime with $p \equiv 1 \pmod {d_i}$ then
\begin{align*}
{_{n+1}G} \left( \tfrac{m_1}{d_1}, \tfrac{m_2}{d_2}, \dotsc, \tfrac{m_{n+1}}{d_{n+1}} \right)_p
=(-1)^n \: p^n \:  {_{n+1}F_n}  {\left( \begin{array}{cccc} \rho_1^{m_1}, & \rho_2^{m_2}, & \dotsc, & \rho_{n+1}^{m_{n+1}} \\
\phantom{\rho_1^{m_1}} & \varepsilon, & \dotsc, & \varepsilon \end{array}
\Big| \; 1 \right)}_{p} \;,
\end{align*}
where $\rho_i$ is the character of order $d_i$ on $\mathbb{F}_p$ given by $\bar{\omega}^{\frac{p-1}{d_i}}$.
\end{prop}

\noindent We now state some congruences between the $G$ function and truncated ordinary hypergeometric series which we will prove later in this section.

\begin{theorem}\label{thm_2G}
For a prime $p$ with $1 < d < p$ and $\phi(d) \leq 2$,
\begin{align*}
{_{2}G} \left( \tfrac{1}{d}, 1-\tfrac{1}{d} \right)_p
&\equiv
{{_{2}F_1} \left[ \begin{array}{cc} \frac{1}{d}, & 1-\frac{1}{d}\vspace{.05in}\\
\phantom{\frac{1}{d}} & 1 \end{array}
\Big| \; 1 \right]}_{p-1}
\pmod {p^2} .
\end{align*}
\end{theorem}

\begin{theorem}\label{thm_3G}
For a prime $p$ with $1 < d < p$ and $\phi(d) \leq 2$,
\begin{align*}
{_{3}G} \left( \tfrac{1}{2}, \tfrac{1}{d} , 1-\tfrac{1}{d} \right)_p
&\equiv
{{_{3}F_2} \left[ \begin{array}{ccc} \frac{1}{2}, & \frac{1}{d}, & 1-\frac{1}{d}\vspace{.05in}\\
\phantom{\frac{1}{d}} & 1, &1 \end{array}
\Big| \; 1 \right]}_{p-1}
\pmod {p^2} .
\end{align*}
\end{theorem}

\begin{theorem}\label{thm_4G1}
For a prime $p$ with $1 < d_i < p$ and $\phi(d_i) \leq 2$,
\begin{align*}
{_{4}G} \left(\tfrac{1}{d_1} , 1-\tfrac{1}{d_1}, \tfrac{1}{d_2} , 1-\tfrac{1}{d_2}\right)_p
&\equiv
{{_{4}F_3} \left[ \begin{array}{cccc} \frac{1}{d_1}, & 1-\frac{1}{d_1}, & \frac{1}{d_2}, & 1-\frac{1}{d_2}\vspace{.05in}\\
\phantom{\frac{1}{d_1}} & 1, & 1, & 1 \end{array}
\Big| \; 1 \right]}_{p-1}
+s(p)\hspace{1pt}p
\pmod {p^3} ,
\end{align*}
where $s(p):=\biggfp{\frac{1}{d_1}}\biggfp{\frac{d_1-1}{d_1}}\biggfp{\frac{1}{d_2}}\biggfp{\frac{d_2-1}{d_2}}=(-1)^{\lfloor \frac{p-1}{d_1} \rfloor+\lfloor \frac{p-1}{d_2} \rfloor}$.
\end{theorem}

\begin{theorem}\label{thm_4G2}
For a prime $p$ with $d< p$, $\phi(d) =4$ and $gcd(r,d)=1$,
\begin{align*}
{_{4}G} \left(\tfrac{1}{d} , \tfrac{r}{d}, 1-\tfrac{r}{d} , 1-\tfrac{1}{d}\right)_p
&\equiv
{{_{4}F_3} \left[ \begin{array}{cccc} \frac{1}{d}, & \frac{r}{d}, & 1-\frac{r}{d}, & 1-\frac{1}{d}\vspace{.05in}\\
\phantom{\frac{1}{d_1}} & 1, & 1, & 1 \end{array}
\Big| \; 1 \right]}_{p-1}
+s(p)\hspace{1pt} p
\pmod {p^3},
\end{align*}
where $s(p) := \gfp{\tfrac{1}{d}} \gfp{\tfrac{r}{d}}\gfp{\tfrac{d-r}{d}}\gfp{\tfrac{d-1}{d}}$.
\end{theorem}

\noindent These congruences cover exactly the 22 truncated hypergeometric series outlined in Section 1. Using Proposition \ref{prop_GtoGHS} it is easy to see the following corollaries.
\begin{cor}\label{cor_2G}
For a prime $p \equiv 1 \pmod {d}$ and $\phi(d) \leq 2$ with $d>1$,
\begin{align*}
-p \:  {_{2}F_1}  {\left( \begin{array}{cc} \rho, & \bar{\rho} \\
\phantom{\rho} & \varepsilon \end{array}
\Big| \; 1 \right)}_{p}
\equiv
{{_{2}F_1} \left[ \begin{array}{cc} \frac{1}{d}, & 1-\frac{1}{d}\vspace{.05in}\\
\phantom{\frac{1}{d}} & 1 \end{array}
\Big| \; 1 \right]}_{p-1}
\pmod {p^2} ,
\end{align*}
where $\rho$ is a character of order $d$ on $\mathbb{F}_p$.
\end{cor}

\begin{cor}\label{cor_3G}
For a prime $p \equiv 1 \pmod {d}$ and $\phi(d) \leq 2$ with $d>1$,
\begin{align*}
p^2 \:  {_{3}F_2}  {\left( \begin{array}{ccc} \psi, & \rho, & \bar{\rho} \\
\phantom{\psi} & \varepsilon, & \varepsilon \end{array}
\Big| \; 1 \right)}_{p}
\equiv
{{_{3}F_2} \left[ \begin{array}{ccc} \frac{1}{2}, & \frac{1}{d}, & 1-\frac{1}{d}\vspace{.05in}\\
\phantom{\frac{1}{d}} & 1, &1 \end{array}
\Big| \; 1 \right]}_{p-1}
\pmod {p^2} ,
\end{align*}
where $\psi$ is the character of order 2 and $\rho$ is a character of order $d$ on $\mathbb{F}_p$.
\end{cor}

\begin{cor}\label{cor_4G1}
For a prime $p \equiv 1 \pmod {d_i}$ and $\phi(d) \leq 2$ with $d>1$,
\begin{align*}
-p^3 \:  {_{4}F_3}  {\left( \begin{array}{cccc} \rho_1, & \bar{\rho_1}, & \rho_2, & \bar{\rho_2} \\
\phantom{\rho_1} & \varepsilon, & \varepsilon, & \varepsilon \end{array}
\Big| \; 1 \right)}_{p}
\equiv
{{_{4}F_3} \left[ \begin{array}{cccc} \frac{1}{d_1}, & 1-\frac{1}{d_1}, & \frac{1}{d_2}, & 1-\frac{1}{d_2}\vspace{.05in}\\
\phantom{\frac{1}{d_1}} & 1, & 1, & 1 \end{array}
\Big| \; 1 \right]}_{p-1}
+s(p)\hspace{1pt}p
\pmod {p^3} ,
\end{align*}
where $s(p):=\biggfp{\frac{1}{d_1}}\biggfp{\frac{d_1-1}{d_1}}\biggfp{\frac{1}{d_2}}\biggfp{\frac{d_2-1}{d_2}}=(-1)^{\lfloor \frac{p-1}{d_1} \rfloor+\lfloor \frac{p-1}{d_2} \rfloor}$ and $\rho_i$ is a character of order $d_i$ on $\mathbb{F}_p$.
\end{cor}

\begin{cor}\label{cor_4G2}
For a prime $p \equiv 1 \pmod {d}$ and $\phi(d) = 4$ and $gcd(r,d)=1$,
\begin{align*}
-p^3 \:  {_{4}F_3}  {\left( \begin{array}{cccc} \rho, & \bar{\rho}, & \rho^r, & \bar{\rho}^r \\
\phantom{\rho_1} & \varepsilon, & \varepsilon, & \varepsilon \end{array}
\Big| \; 1 \right)}_{p}
\equiv
{{_{4}F_3} \left[ \begin{array}{cccc} \frac{1}{d}, & \frac{r}{d}, & 1-\frac{r}{d}, & 1-\frac{1}{d}\vspace{.05in}\\
\phantom{\frac{1}{d_1}} & 1, & 1, & 1 \end{array}
\Big| \; 1 \right]}_{p-1}
+s(p)\hspace{1pt} p
\pmod {p^3},
\end{align*}
where $s(p) := \gfp{\tfrac{1}{d}} \gfp{\tfrac{r}{d}}\gfp{\tfrac{d-r}{d}}\gfp{\tfrac{d-1}{d}}$ and $\rho$ is a character of order $d$ on $\mathbb{F}_p$.
\end{cor}

\noindent Note that Corollaries \ref{cor_2G} and \ref{cor_3G} coincide with Theorem 1 in  \cite{M2} and Corollary 2 in \cite{M} (after the above refinement is made) when $\phi(d)=2$.\\

We now prove Theorems \ref{thm_2G} to \ref{thm_4G2}.
\begin{proof}[Proof of Theorem \ref{thm_2G}]
One easily checks the result for $p<7$. Let $p\geq7$ be a prime.
Reducing Definition \ref{def_GFn} modulo $p^2$ and noting that $\omega(-1) \equiv  -1 \pmod{p^2}$ by (\ref{prop_TechChar}) we get
\begin{align*}
&{_{2}G} \left( \tfrac{1}{d}, 1-\tfrac{1}{d} \right)_p
\equiv \frac{-1}{p-1} \left[
\sum_{j=0}^{\lfloor \frac{p-1}{d} \rfloor} 
\frac{{\biggfp{\tfrac{j}{p-1}}}^{2} \biggfp{\frac{1}{d}-\frac{j}{p-1}}\biggfp{\frac{d-1}{d}-\frac{j}{p-1}}}
{\biggfp{\frac{1}{d}}\biggfp{\frac{d-1}{d}}} \right. \\
&\left. \qquad \qquad\qquad \qquad \qquad \qquad \qquad
-p \sum_{j=\lfloor \frac{p-1}{d} \rfloor+1}^{\lfloor (d-1)\frac{p-1}{d} \rfloor} 
\frac{{\biggfp{\tfrac{j}{p-1}}}^{2} \biggfp{\frac{d+1}{d}-\frac{j}{p-1}}\biggfp{\frac{d-1}{d}-\frac{j}{p-1}}}
{\biggfp{\frac{1}{d}}\biggfp{\frac{d-1}{d}}}
\right] \pmod{p^2}.
\end{align*}
We note that $\frac{-1}{p-1}\equiv 1+p \pmod{p^2}$ and each of the $p$-adic gamma function products above are in $\mathbb{Z}_p^*$. Using Propositions \ref{prop_pGamma} and \ref{prop_pGammaCong} to expand the terms involved, we have
\begin{align*}
{_{2}G} & \left( \tfrac{1}{d}, 1-\tfrac{1}{d} \right)_p\\
&\equiv \sum_{j=0}^{\lfloor \frac{p-1}{d} \rfloor} 
\frac{{\biggfp{-j-jp}}^{2} \biggfp{\frac{1}{d}+j+jp}\biggfp{\frac{d-1}{d}+j+jp}}
{\biggfp{\frac{1}{d}}\biggfp{\frac{d-1}{d}}} \\
&+p \left[
\sum_{j=0}^{\lfloor \frac{p-1}{d} \rfloor} 
\frac{{\biggfp{-j}}^{2} \biggfp{\frac{1}{d}+j}\biggfp{\frac{d-1}{d}+j}}
{\biggfp{\frac{1}{d}}\biggfp{\frac{d-1}{d}}}
- \sum_{j=\lfloor \frac{p-1}{d} \rfloor+1}^{\lfloor (d-1)\frac{p-1}{d} \rfloor} 
\frac{{\biggfp{-j}}^{2} \biggfp{\frac{d+1}{d}+j}\biggfp{\frac{d-1}{d}+j}}
{\biggfp{\frac{1}{d}}\biggfp{\frac{d-1}{d}}}
\right]\\
&\equiv \sum_{j=0}^{\lfloor \frac{p-1}{d} \rfloor} 
\frac{\biggfp{\frac{1}{d}+j+jp}\biggfp{\frac{d-1}{d}+j+jp}}
{\biggfp{\frac{1}{d}}\biggfp{\frac{d-1}{d}}{\biggfp{1+j+jp}}^{2}} \\
&+p \left[
\sum_{j=0}^{\lfloor \frac{p-1}{d} \rfloor} 
\frac{\biggfp{\frac{1}{d}+j}\biggfp{\frac{d-1}{d}+j}}
{\biggfp{\frac{1}{d}}\biggfp{\frac{d-1}{d}}{\biggfp{1+j}}^{2}}
- \sum_{j=\lfloor \frac{p-1}{d} \rfloor+1}^{\lfloor (d-1)\frac{p-1}{d} \rfloor} 
\frac{\biggfp{\frac{d+1}{d}+j}\biggfp{\frac{d-1}{d}+j}}
{\biggfp{\frac{1}{d}}\biggfp{\frac{d-1}{d}}{\biggfp{1+j}}^{2}}
\right]\\
&\equiv \sum_{j=0}^{\lfloor \frac{p-1}{d} \rfloor} 
\frac{\biggfp{\frac{1}{d}+j}\biggfp{\frac{d-1}{d}+j}
\left[1+jp\left(G_1(\frac{1}{d}+j)+G_1(\frac{d-1}{d}+j)\right)\right]}
{\biggfp{\frac{1}{d}}\biggfp{\frac{d-1}{d}}{\biggfp{1+j}}^{2}
\left[1+2\hspace{1pt}jp\hspace{1pt}G_1(1+j)\right]} \\
&+p \left[
\sum_{j=0}^{\lfloor \frac{p-1}{d} \rfloor} 
\frac{\biggfp{\frac{1}{d}+j}\biggfp{\frac{d-1}{d}+j}}
{\biggfp{\frac{1}{d}}\biggfp{\frac{d-1}{d}}{\biggfp{1+j}}^{2}}
- \sum_{j=\lfloor \frac{p-1}{d} \rfloor+1}^{\lfloor (d-1)\frac{p-1}{d} \rfloor} 
\frac{\biggfp{\frac{d+1}{d}+j}\biggfp{\frac{d-1}{d}+j}}
{\biggfp{\frac{1}{d}}\biggfp{\frac{d-1}{d}}{\biggfp{1+j}}^{2}}
\right] \pmod {p^2}.\\
\end{align*}
By multiplying the left-hand side above and below by ${1-2\hspace{1pt}jp \hspace{1pt} G_1(1+j)}$ we see that
\begin{align*}
\frac{1+jp\left(G_1(\frac{1}{d}+j)+G_1(\frac{d-1}{d}+j)\right)}{1+2\hspace{1pt}jp\hspace{1pt}G_1(1+j)}
\equiv 1+jp  \left(G_1(\tfrac{1}{d}+j) + G_1(\tfrac{d-1}{d}+j) - 2\hspace{1pt}G_1(1+j)\right) \pmod{p^2}.
\end{align*}
We define $$A(j):=G_1\left(\tfrac{1}{d}+j\right) + G_1\left(\tfrac{d-1}{d}+j\right) - 2\hspace{1pt} G_1\bigl(1+j\bigr).$$

\noindent Consider $\gfp{\frac{d+1}{d}+j}=\gfp{1+\frac{1}{d}+j}$ for $\lfloor \frac{p-1}{d} \rfloor+1 \leq j \leq  {\lfloor (d-1)\frac{p-1}{d} \rfloor}$. We first note that
\begin{equation*}
\tfrac{1}{d}+j \in p\mathbb{Z}_p \Longleftrightarrow rep\left(\tfrac{1}{d}\right)+j \in p\mathbb{Z}_p \Longleftrightarrow rep\left(\tfrac{1}{d}\right)+j =p \Longleftrightarrow j=p-rep\left(\tfrac{1}{d}\right)\;.
\end{equation*}
If $p\equiv 1 \pmod d$ then using Corollary \ref{cor_repGenformula} we get that 
\begin{equation*}
p-rep\left(\tfrac{1}{d}\right)=\lfloor \tfrac{p-1}{d} \rfloor <  \lfloor \tfrac{p-1}{d} \rfloor +1\;.
\end{equation*}
Similarly, if $p\equiv d-1 \pmod d$,
\begin{equation*}
p-rep\left(\tfrac{1}{d}\right)=p-\lfloor \tfrac{p-1}{d} \rfloor -1= p- \lceil \tfrac{p-1}{d} \rceil >  p- \lceil \tfrac{p-1}{d} \rceil-1 = {\lfloor (d-1)\tfrac{p-1}{d} \rfloor} \;.
\end{equation*}
(For $x \in \mathbb{Q}$, we let $\lceil x \rceil$ denote the least integer greater than or equal to $x$.)
Then it follows from Proposition \ref{prop_pGamma} (1) that 
\begin{equation}\label{for_GammapPlus1}
\gfp{\tfrac{d+1}{d}+j}=-\left(\tfrac{1}{d}+j\right)\gfp{\tfrac{1}{d}+j}
\end{equation}
for $\lfloor \frac{p-1}{d} \rfloor+1 \leq j \leq  {\lfloor (d-1)\frac{p-1}{d} \rfloor}$. Also 
\begin{equation}\label{for_GammapPlus}
\gfp{1+j}=(-1)^{1+j} j!
\end{equation}
 for $j<p$.
Therefore, 
\begin{multline}\label{for_2G}
{_{2}G} \left( \tfrac{1}{d}, 1-\tfrac{1}{d} \right)_p
\equiv \sum_{j=0}^{\lfloor \frac{p-1}{d} \rfloor} 
\frac{\biggfp{\frac{1}{d}+j}\biggfp{\frac{d-1}{d}+j}
\left[1+jp\hspace{1pt} A(j)\right]}
{\biggfp{\frac{1}{d}}\biggfp{\frac{d-1}{d}}{j!}^{2}} \\
+p \left[
\sum_{j=0}^{\lfloor \frac{p-1}{d} \rfloor} 
\frac{\biggfp{\frac{1}{d}+j}\biggfp{\frac{d-1}{d}+j}}
{\biggfp{\frac{1}{d}}\biggfp{\frac{d-1}{d}}{j!}^{2}}
+ \sum_{j=\lfloor \frac{p-1}{d} \rfloor+1}^{\lfloor (d-1)\frac{p-1}{d} \rfloor} 
\frac{\biggfp{\frac{1}{d}+j}\biggfp{\frac{d-1}{d}+j}\bigl(\frac{1}{d}+j\bigr)}
{\biggfp{\frac{1}{d}}\biggfp{\frac{d-1}{d}}{j!}^{2}}
\right]
\pmod {p^2} .
\end{multline}
By definition,
\begin{equation*}
{{_{2}F_1} \left[ \begin{array}{cc} \frac{1}{d}, & 1-\frac{1}{d}\vspace{.05in}\\
\phantom{\frac{1}{d}} & 1 \end{array}
\Big| \; 1 \right]}_{p-1}
=\sum_{j=0}^{p-1} \frac{\ph{\frac{1}{d}}{j} \ph{\frac{d-1}{d}}{j}}{{j!}^{2}} 
\equiv \sum_{j=0}^{p-\lfloor \tfrac{p-1}{d} \rfloor - 1} \frac{\ph{\frac{1}{d}}{j} \ph{\frac{d-1}{d}}{j}} {{j!}^{2}} \pmod {p^2},
\end{equation*}
as we can see from Lemma \ref{lem_gammapd2} that $\ph{\frac{1}{d}}{j} \ph{\frac{d-1}{d}}{j} \in p^2\mathbb{Z}_p$ for $p-\lfloor \tfrac{p-1}{d} \rfloor \leq j \leq p-1$.
Using Lemma \ref{lem_gammapd2} again then gives us
\begin{multline}\label{for_2F1}
{{_{2}F_1} \left[ \begin{array}{cc} \frac{1}{d}, & 1-\frac{1}{d}\vspace{.05in}\\
\phantom{\frac{1}{d}} & 1 \end{array}
\Big| \; 1 \right]}_{p-1}
\equiv
\sum_{j=0}^{\lfloor \frac{p-1}{d} \rfloor} \frac{\biggfp{\frac{1}{d}+j}\biggfp{\frac{d-1}{d}+j}}
{\biggfp{\frac{1}{d}}\biggfp{\frac{d-1}{d}}{j!}^{2}}\\
+p\sum_{j=\lfloor \frac{p-1}{d} \rfloor +1}^{p-\lfloor \tfrac{p-1}{d} \rfloor - 1} \frac{\biggfp{\frac{1}{d}+j}\biggfp{\frac{d-1}{d}+j} \bigl(\frac{1}{d}\bigr)}
{\biggfp{\frac{1}{d}}\biggfp{\frac{d-1}{d}}{j!}^{2}}
\pmod {p^2}.
\end{multline}

\noindent Combining (\ref{for_2G}) and (\ref{for_2F1}) it suffices to show
\begin{multline*}
\sum_{j=0}^{\lfloor \frac{p-1}{d} \rfloor} 
\frac{\biggfp{\frac{1}{d}+j}\biggfp{\frac{d-1}{d}+j}\left[1+j\hspace{1pt} A(j)\right]}
{\biggfp{\frac{1}{d}}\biggfp{\frac{d-1}{d}}{j!}^{2}}
+ \sum_{j=\lfloor \frac{p-1}{d} \rfloor+1}^{\lfloor (d-1)\frac{p-1}{d} \rfloor} 
\frac{\biggfp{\frac{1}{d}+j}\biggfp{\frac{d-1}{d}+j}\bigl(j\bigr)}
{\biggfp{\frac{1}{d}}\biggfp{\frac{d-1}{d}}{j!}^{2}}
\\
- \sum_{j=\lfloor (d-1)\frac{p-1}{d} \rfloor+1}^{p-\lfloor \tfrac{p-1}{d} \rfloor - 1} 
\frac{\biggfp{\frac{1}{d}+j}\biggfp{\frac{d-1}{d}+j}\bigl(\frac{1}{d}\bigr)}
{\biggfp{\frac{1}{d}}\biggfp{\frac{d-1}{d}}{j!}^{2}}
\equiv 0 \pmod {p}.
\end{multline*}

\noindent We now examine $A(j)$, $\biggfp{\frac{1}{d}+j}\biggfp{\frac{d-1}{d}+j}$ and $\biggfp{\frac{1}{d}}\biggfp{\frac{d-1}{d}}$. We first note that $\{rep\left(\tfrac{1}{d}\right),rep\left(\tfrac{d-1}{d}\right)\}=\{p-\lfloor \tfrac{p-1}{d} \rfloor, \lfloor \tfrac{p-1}{d} \rfloor +1\}$. This can seen from Corollary \ref{cor_repGenformula} and the fact that if $p\equiv a \pmod {d}$ then $a \in \{1,d-1\}$ as $\phi(d)\leq2$.
Using Corollary \ref{cor_pGammaCong1} (3) and Proposition \ref{prop_pGammaG} (1) we see that
\begin{align*}
A(j)&=G_1\left(\tfrac{1}{d}+j\right) + G_1\left(\tfrac{d-1}{d}+j\right) - 2\hspace{1pt} G_1\bigl(1+j\bigr)\\
&\equiv G_1\left(rep\left(\tfrac{1}{d}\right)+j\right) + G_1\left(rep\left(\tfrac{d-1}{d}\right)+j\right) - 2\hspace{1pt} G_1\bigl(1+j\bigr)\\
&\equiv G_1\left(p-\lfloor \tfrac{p-1}{d} \rfloor+j\right) + G_1\left(\lfloor \tfrac{p-1}{d} \rfloor +1+j\right) - 2\hspace{1pt} G_1\bigl(1+j\bigr)\\
&\equiv H_{p-\lfloor \frac{p-1}{d} \rfloor-1+j}^{(1)} +  H_{\lfloor \frac{p-1}{d} \rfloor +j}^{(1)}-2 \hspace{1pt} H_{j}^{(1)}-
\begin{cases}
0 & \text{if } 0\leq j\leq \lfloor \frac{p-1}{d}\rfloor, \\[6pt]
\frac{1}{p} & \text{if } \lfloor \frac{p-1}{d}\rfloor+1\leq j\leq p-\lfloor \frac{p-1}{d}\rfloor-1,\\[6pt]
\frac{2}{p} & \text{if } p-\lfloor \frac{p-1}{d}\rfloor\leq j\leq p-1,\\[6pt]
\end{cases}
\pmod{p}.
\end{align*}
Using Proposition \ref{prop_pGamma} we get that
\begin{align*}
&\biggfp{\tfrac{1}{d}+j}\biggfp{\tfrac{d-1}{d}+j}\\
&\qquad \equiv \biggfp{rep\left(\tfrac{1}{d}\right)+j}\biggfp{rep\left(\tfrac{d-1}{d}\right)+j}\\
&\qquad\equiv \biggfp{p-\lfloor \tfrac{p-1}{d} \rfloor+j}\biggfp{\lfloor \tfrac{p-1}{d} \rfloor +1+j}\\
&\qquad\equiv \left(p-\lfloor \tfrac{p-1}{d} \rfloor-1+j\right)! \left(\lfloor \tfrac{p-1}{d} \rfloor +j\right)! (-1)^{p+1+2j}
\begin{cases}
1 & \text{if } 0\leq j\leq \lfloor \frac{p-1}{d}\rfloor ,\\[6pt]
\frac{1}{p} & \text{if }  \lfloor \frac{p-1}{d}\rfloor+1\leq j\leq p-\lfloor \frac{p-1}{d}\rfloor-1,\\[6pt]
\frac{1}{p^2} & \text{if }  p-\lfloor \frac{p-1}{d}\rfloor\leq j\leq p-1,\\[6pt]
\end{cases}\\
&\qquad \equiv \left(p-\lfloor \tfrac{p-1}{d} \rfloor-1+j\right)! \left(\lfloor \tfrac{p-1}{d} \rfloor +j\right)!
\begin{cases}
1 &\text{if } 0\leq j\leq \lfloor \frac{p-1}{d}\rfloor, \\[6pt]
\frac{1}{p} &\text{if }  \lfloor \frac{p-1}{d}\rfloor+1\leq j\leq p-\lfloor \frac{p-1}{d}\rfloor-1,\\[6pt]
\frac{1}{p^2} &\text{if }  p-\lfloor \frac{p-1}{d}\rfloor\leq j\leq p-1,\\[6pt]
\end{cases}
\pmod{p},
\end{align*}
and
\begin{align*}
\biggfp{\tfrac{1}{d}}\biggfp{\tfrac{d-1}{d}}=\pm 1\;.
\end{align*}

\noindent Therefore it suffices to prove
\begin{multline}\label{for_Resid2G}
\sum_{j=0}^{\lfloor \frac{p-1}{d} \rfloor} 
\ph{j+1}{p-\lfloor \frac{p-1}{d} \rfloor-1}\ph{j+1}{\lfloor \frac{p-1}{d} \rfloor}
\left[1+j\hspace{1pt} \left( H_{p-\lfloor \frac{p-1}{d} \rfloor-1+j}^{(1)} +  H_{\lfloor \frac{p-1}{d} \rfloor +j}^{(1)}-2 \hspace{1pt} H_{j}^{(1)}\right)\right]\\
+ \sum_{j=\lfloor \frac{p-1}{d} \rfloor+1}^{\lfloor (d-1)\frac{p-1}{d} \rfloor} 
\ph{j+1}{p-\lfloor \frac{p-1}{d} \rfloor-1}\ph{j+1}{\lfloor \frac{p-1}{d} \rfloor}
\Bigl(\tfrac{j}{p}\Bigr)
\\
- \sum_{j=\lfloor (d-1)\frac{p-1}{d} \rfloor+1}^{p-\lfloor \tfrac{p-1}{d} \rfloor - 1} 
\ph{j+1}{p-\lfloor \frac{p-1}{d} \rfloor-1}\ph{j+1}{\lfloor \frac{p-1}{d} \rfloor}
\Bigl(\tfrac{1}{dp}\Bigr)
\equiv 0 \pmod {p}.
\end{multline}

\noindent Now
$\ph{j+1}{p-\lfloor \frac{p-1}{d} \rfloor-1}\ph{j+1}{\lfloor \frac{p-1}{d} \rfloor} \in p\mathbb{Z}_p$ for $\lfloor \frac{p-1}{d}\rfloor+1\leq j\leq p-\lfloor \frac{p-1}{d}\rfloor-1$ so
\begin{multline*}
\ph{j+1}{p-\lfloor \frac{p-1}{d} \rfloor-1}\ph{j+1}{\lfloor \frac{p-1}{d} \rfloor}
\left[1+j\hspace{1pt} \left( H_{p-\lfloor \frac{p-1}{d} \rfloor-1+j}^{(1)} +  H_{\lfloor \frac{p-1}{d} \rfloor +j}^{(1)}-2 \hspace{1pt} H_{j}^{(1)}\right)\right]\\[6pt]
\equiv
\ph{j+1}{p-\lfloor \frac{p-1}{d} \rfloor-1}\ph{j+1}{\lfloor \frac{p-1}{d} \rfloor} \hspace{1pt} \Bigl(\tfrac{j}{p}\Bigr) \pmod{p}
\end{multline*}
for $\lfloor \frac{p-1}{d}\rfloor+1\leq j\leq p-\lfloor \frac{p-1}{d}\rfloor-1$.
\noindent Thus (\ref{for_Resid2G}) becomes
\begin{multline*}
\sum_{j=0}^{p-\lfloor \frac{p-1}{d} \rfloor - 1} 
\ph{j+1}{p-\lfloor \frac{p-1}{d} \rfloor-1}\ph{j+1}{\lfloor \frac{p-1}{d} \rfloor}
\left[1+j\hspace{1pt} \left( H_{p-\lfloor \frac{p-1}{d} \rfloor-1+j}^{(1)} +  H_{\lfloor \frac{p-1}{d} \rfloor +j}^{(1)}-2 \hspace{1pt} H_{j}^{(1)}\right)\right]
\\ 
- \sum_{j=\lfloor (d-1)\frac{p-1}{d} \rfloor+1}^{p-\lfloor \frac{p-1}{d} \rfloor - 1} 
\ph{j+1}{p-\lfloor \frac{p-1}{d} \rfloor-1}\ph{j+1}{\lfloor \frac{p-1}{d} \rfloor}
\bigl(\tfrac{1}{dp}+\tfrac{j}{p}\bigr)
\equiv 0 \pmod {p}.
\end{multline*}

\noindent If $p\equiv1 \pmod{d}$ then $\lfloor (d-1)\frac{p-1}{d} \rfloor+1>p-\lfloor \tfrac{p-1}{d} \rfloor - 1$ and the second sum is vacuous. If $p\equiv d-1 \pmod{d}$ then $j=p-\frac{p+1}{d}$ and
\begin{equation*}
\tfrac{1}{dp}+\tfrac{j}{p}=\tfrac{d-1}{d} \in \mathbb{Z}_p^{*}.
\end{equation*}
Therefore, in this case, the second sum $\equiv 0 \pmod{p}$ as $\ph{j+1}{p-\lfloor \tfrac{p-1}{d} \rfloor-1}\ph{j+1}{\lfloor \tfrac{p-1}{d} \rfloor} \in p\mathbb{Z}_p$ for $\lfloor \frac{p-1}{d}\rfloor+1\leq j\leq p-\lfloor \frac{p-1}{d}\rfloor-1$.
\noindent Finally, to complete the proof, we need to show
\begin{align*}
&
\sum_{j=0}^{p-\lfloor \tfrac{p-1}{d} \rfloor - 1} 
\ph{j+1}{p-\lfloor \frac{p-1}{d} \rfloor-1}\ph{j+1}{\lfloor \frac{p-1}{d} \rfloor}
\left[1+j\hspace{1pt} \left( H_{p-\lfloor \frac{p-1}{d} \rfloor-1+j}^{(1)} +  H_{\lfloor \frac{p-1}{d} \rfloor +j}^{(1)}-2 \hspace{1pt} H_{j}^{(1)}\right)\right]
\equiv 0 \pmod {p}.
\end{align*}
\\

\noindent Note we can extend the upper limit of this sum to $j=p-1$ as $\ph{j+1}{p-\lfloor \frac{p-1}{d} \rfloor-1}\ph{j+1}{\lfloor \frac{p-1}{d} \rfloor} \in p^2\mathbb{Z}_p$ for $p-\lfloor \frac{p-1}{d}\rfloor \leq j\leq p-1.$
\noindent Define
\begin{equation*}
P(j):= \frac{d}{dj} \left[ j \hspace{1pt} \ph{j+1}{p-\lfloor \frac{p-1}{d} \rfloor-1}\ph{j+1}{\lfloor \frac{p-1}{d} \rfloor} \right] = \sum_{k=0}^{p-1} a_k j^k \;.
\end{equation*}
Then
\begin{align*}
\sum_{j=0}^{p- 1} 
\ph{j+1}{p-\lfloor \frac{p-1}{d} \rfloor-1}\ph{j+1}{\lfloor \frac{p-1}{d} \rfloor}
\left[1+j\hspace{1pt} \left( H_{p-\lfloor \frac{p-1}{d} \rfloor-1+j}^{(1)} +  H_{\lfloor \frac{p-1}{d} \rfloor +j}^{(1)}-2 \hspace{1pt} H_{j}^{(1)}\right)\right]
= \sum_{j=0}^{p-1} P(j).
\end{align*}
For a positive integer $k$, we have
\begin{equation}\label{exp_sums}
\sum^{p-1}_{j=1} j^k\equiv
\begin{cases}
-1 \pmod {p}& \text{if $(p-1) \vert k$} \; ,\\
\phantom{-}0 \pmod {p}& \text{otherwise} \; .
\end{cases}
\end{equation}
Therefore,
\begin{equation*}
 \sum_{j=0}^{p-1} P(j)=\sum_{j=0}^{p-1} \sum_{k=0}^{p-1} a_k j^k =  p \hspace{1pt}a_0 + \sum_{j=1}^{p-1} \sum_{k=1}^{p-1} a_k j^k \equiv   \sum_{k=1}^{p-1} a_k \sum_{j=1}^{p-1}  j^k
\equiv -a_{p-1} \pmod{p}.
\end{equation*}
By definition of $P(j)$ we get
\begin{equation*}
\ph{j+1}{p-\lfloor \frac{p-1}{d} \rfloor-1}\ph{j+1}{\lfloor \frac{p-1}{d}\rfloor}  = \sum_{k=0}^{p-1} \frac{a_k}{k+1} j^k \;.
\end{equation*}
$P(j)$ is monic so $a_{p-1}=p\equiv 0 \pmod{p}$.
Thus
\begin{align*}
\sum_{j=0}^{p- 1} 
\ph{j+1}{p-\lfloor \frac{p-1}{d} \rfloor-1}\ph{j+1}{\lfloor \frac{p-1}{d} \rfloor}
\left[1+j\hspace{1pt} \left( H_{p-\lfloor \frac{p-1}{d} \rfloor-1+j}^{(1)} +  H_{\lfloor \frac{p-1}{d} \rfloor +j}^{(1)}-2 \hspace{1pt} H_{j}^{(1)}\right)\right]
\equiv 0 \pmod{p}
\end{align*}
as required.
\end{proof}

\begin{proof}[Proof of Theorem \ref{thm_3G}]
The proof is similar to that of Theorem \ref{thm_2G} so we omit many of the details. One easily checks the result for $p<7$. Let $p\geq7$ be a prime. We reduce Definition \ref{def_GFn} modulo $p^2$ and use Propositions \ref{prop_pGamma} and \ref{prop_pGammaCong} to expand the terms involved, taking note of (\ref{for_GammapPlus1}) and (\ref{for_GammapPlus}), to get
\begin{multline}\label{for_3G}
{_{3}G} \left( \tfrac{1}{2}, \tfrac{1}{d} , 1-\tfrac{1}{d} \right)_p
\equiv \sum_{j=0}^{\lfloor \frac{p-1}{d} \rfloor} 
\frac{(-1)^j \hspace{1pt} \biggfp{\frac{1}{2}+j}\biggfp{\frac{1}{d}+j}\biggfp{\frac{d-1}{d}+j}
\left[1+jp\hspace{1pt} A(j)\right]}
{\biggfp{\frac{1}{2}}\biggfp{\frac{1}{d}}\biggfp{\frac{d-1}{d}}{j!}^{3}} \\
+p \left[
\sum_{j=0}^{\lfloor \frac{p-1}{d} \rfloor} 
\frac{(-1)^j \hspace{1pt} \biggfp{\frac{1}{2}+j}\biggfp{\frac{1}{d}+j}\biggfp{\frac{d-1}{d}+j}}
{\biggfp{\frac{1}{2}}\biggfp{\frac{1}{d}}\biggfp{\frac{d-1}{d}}{j!}^{3}} 
\right. \\  \left.
+ \sum_{j=\lfloor \frac{p-1}{d} \rfloor+1}^{\frac{p-1}{2}}  
\frac{(-1)^j \hspace{1pt}\biggfp{\frac{1}{2}+j}\biggfp{\frac{1}{d}+j}\biggfp{\frac{d-1}{d}+j}\bigl(\frac{1}{d}+j\bigr)}
{\biggfp{\frac{1}{2}}\biggfp{\frac{1}{d}}\biggfp{\frac{d-1}{d}}{j!}^{3}}
\right] \pmod {p^2},
\end{multline}
where $$A(j):=G_1\left(\tfrac{1}{2}+j\right)+G_1\left(\tfrac{1}{d}+j\right) + G_1\left(\tfrac{d-1}{d}+j\right) - 3\hspace{1pt} G_1\bigl(1+j\bigr).$$
Applying Lemma \ref{lem_gammapd2} gives us
\begin{multline}\label{for_3F2}
{{_{3}F_2} \left[ \begin{array}{ccc} \frac{1}{2}, & \frac{1}{d}, & 1-\frac{1}{d}\vspace{.05in}\\
\phantom{\frac{1}{d}} & 1, &1 \end{array}
\Big| \; 1 \right]}_{p-1}
\equiv
\sum_{j=0}^{\lfloor \frac{p-1}{d} \rfloor} 
\frac{(-1)^j \hspace{1pt} \biggfp{\frac{1}{2}+j}\biggfp{\frac{1}{d}+j}\biggfp{\frac{d-1}{d}+j}}
{\biggfp{\frac{1}{2}}\biggfp{\frac{1}{d}}\biggfp{\frac{d-1}{d}}{j!}^{3}} 
\\
+p\sum_{j=\lfloor \frac{p-1}{d} \rfloor +1}^{\frac{p-1}{2}} 
\frac{(-1)^j \hspace{1pt} \biggfp{\frac{1}{2}+j}\biggfp{\frac{1}{d}+j}\biggfp{\frac{d-1}{d}+j}\bigl(\frac{1}{d}\bigr)}
{\biggfp{\frac{1}{2}}\biggfp{\frac{1}{d}}\biggfp{\frac{d-1}{d}}{j!}^{3}} 
\pmod {p^2}.
\end{multline}

\noindent Combining (\ref{for_3G}) and (\ref{for_3F2}) it suffices to show
\begin{multline*}
\sum_{j=0}^{\lfloor \frac{p-1}{d} \rfloor} 
\frac{(-1)^j \hspace{1pt}\biggfp{\frac{1}{2}+j}\biggfp{\frac{1}{d}+j}\biggfp{\frac{d-1}{d}+j}\left[1+j\hspace{1pt} A(j)\right]}
{\biggfp{\frac{1}{2}}\biggfp{\frac{1}{d}}\biggfp{\frac{d-1}{d}}{j!}^{3}}
\\
+ \sum_{j=\lfloor \frac{p-1}{d} \rfloor+1}^{\frac{p-1}{2}} 
\frac{(-1)^j \hspace{1pt}\biggfp{\frac{1}{2}+j}\biggfp{\frac{1}{d}+j}\biggfp{\frac{d-1}{d}+j}\bigl(j\bigr)}
{\biggfp{\frac{1}{2}}\biggfp{\frac{1}{d}}\biggfp{\frac{d-1}{d}}{j!}^{3}}
 \equiv 0 \pmod {p}.
\end{multline*}

\noindent We now examine $A(j)$, $\biggfp{\frac{1}{2}+j}\biggfp{\frac{1}{d}+j}\biggfp{\frac{d-1}{d}+j}$ and $\biggfp{\frac{1}{2}}\biggfp{\frac{1}{d}}\biggfp{\frac{d-1}{d}}$. 
Using Corollary \ref{cor_pGammaCong1} (3) and Proposition \ref{prop_pGammaG} (1) we see that
\begin{align*}
A(j)
&\equiv H_{\frac{p-1}{2}+j}^{(1)} + H_{p-\lfloor \frac{p-1}{d} \rfloor-1+j}^{(1)} +  H_{\lfloor \frac{p-1}{d} \rfloor +j}^{(1)}-3 \hspace{1pt} H_{j}^{(1)}-
\begin{cases}
0 & \text{if } 0\leq j\leq \lfloor \frac{p-1}{d}\rfloor, \\[6pt]
\frac{1}{p} & \text{if }  \lfloor \frac{p-1}{d}\rfloor+1\leq j\leq \frac{p-1}{2},\\[6pt]
\frac{2}{p} & \text{if }  \frac{p+1}{2}\leq j\leq p-\lfloor \frac{p-1}{d}\rfloor-1,\\[6pt]
\frac{3}{p} & \text{if }  p-\lfloor \frac{p-1}{d}\rfloor\leq j\leq p-1,\\[6pt]
\end{cases}
\pmod{p}.
\end{align*}
Using Proposition \ref{prop_pGamma} we get that
\begin{multline*}
\biggfp{\tfrac{1}{2}+j}\biggfp{\tfrac{1}{d}+j}\biggfp{\tfrac{d-1}{d}+j}
\equiv \left(\tfrac{p-1}{2}+j\right)! \left(p-\lfloor \tfrac{p-1}{d} \rfloor-1+j\right)! \left(\lfloor \tfrac{p-1}{d} \rfloor +j\right)! \\ \cdot (-1)^{\frac{p+1}{2}+j}
\begin{cases}
1 &\text{if } 0\leq j\leq \lfloor \frac{p-1}{d}\rfloor, \\[6pt]
\frac{1}{p} & \text{if }  \lfloor \frac{p-1}{d}\rfloor+1\leq j\leq \frac{p-1}{2},\\[6pt]
\frac{1}{p^2} & \text{if }  \frac{p+1}{2}\leq j\leq p-\lfloor \frac{p-1}{d}-1\rfloor,\\[6pt]
\frac{1}{p^3} & \text{if }  p-\lfloor \frac{p-1}{d}\rfloor\leq j\leq p-1,\\[6pt]
\end{cases}
\pmod{p},
\end{multline*}
and
\begin{align*}
\biggfp{\tfrac{1}{2}}\biggfp{\tfrac{1}{d}}\biggfp{\tfrac{d-1}{d}}=\pm \biggfp{\tfrac{1}{2}} \equiv \pm \biggfp{\tfrac{p+1}{2}} \equiv \pm (-1)^{\frac{p+1}{2}} \left(\tfrac{p-1}{2}\right)! \pmod{p}.
\end{align*}

\noindent Therefore, as gcd$\left(p,\left(\tfrac{p-1}{2}\right)!\right)=1$, it suffices to prove
\begin{multline*}
\sum_{j=0}^{\tfrac{p-1}{2}} 
\ph{j+1}{\frac{p-1}{2}} \ph{j+1}{p-\lfloor \frac{p-1}{d} \rfloor-1}\ph{j+1}{\lfloor \frac{p-1}{d} \rfloor}\\
\left[1+j\left(H_{\frac{p-1}{2}+j}^{(1)} + H_{p-\lfloor \frac{p-1}{d} \rfloor-1+j}^{(1)} +  H_{\lfloor \frac{p-1}{d} \rfloor +j}^{(1)}
-3 \hspace{1pt} H_{j}^{(1)}\right)\right]
\equiv 0 \pmod {p}.
\end{multline*}

\noindent We extend the upper limit of this sum to $j=p-1$ as $\ph{j+1}{\frac{p-1}{2}} \ph{j+1}{p-\lfloor \frac{p-1}{d} \rfloor-1} \ph{j+1}{\lfloor \frac{p-1}{d} \rfloor}\\ \in p^2\mathbb{Z}_p$ for $\frac{p-1}{2} \leq j\leq p-1.$ We define
\begin{equation*}
P(j):= \frac{d}{dj} \left[ j \hspace{1pt} \ph{j+1}{\frac{p-1}{2}} \ph{j+1}{p-\lfloor \frac{p-1}{d} \rfloor-1}\ph{j+1}{\lfloor \frac{p-1}{d} \rfloor} \right] = \sum_{k=0}^{\frac{3(p-1)}{2}} a_k j^k \;
\end{equation*}
and show
\begin{equation*}
\sum_{j=0}^{p- 1} P(j) \equiv 0 \pmod p
\end{equation*}
in similar fashion to the proof of Theorem (\ref{thm_2G}). Noticing that
\begin{multline*}
P(j)=
\ph{j+1}{\frac{p-1}{2}} \ph{j+1}{p-\lfloor \frac{p-1}{d} \rfloor-1}\ph{j+1}{\lfloor \frac{p-1}{d} \rfloor}\\[6pt]
\left[1+j\left(H_{\frac{p-1}{2}+j}^{(1)} + H_{p-\lfloor \frac{p-1}{d} \rfloor-1+j}^{(1)} +  H_{\lfloor \frac{p-1}{d} \rfloor +j}^{(1)}-3 \hspace{1pt} H_{j}^{(1)}\right)\right]
\end{multline*}
completes the proof.
\end{proof}

\begin{proof}[Proof of Theorem \ref{thm_4G1}]
One easily checks the result for $p<7$. Let $p\geq7$ be a prime.
Assume without loss of generality that $\tfrac{1}{d_1}\leq\tfrac{1}{d_2}$. Let $p\equiv a_i \pmod{d_i}$. Then $a_i\in \{1,d_i-1\}$ as $\phi(d_i)\leq2$. Therefore, by Corollary \ref{cor_repGenformula}, $\{rep\bigl(\tfrac{1}{d_i}\bigr),rep\bigl(\tfrac{d_i-1}{d_i}\bigr)\}=\{p-\lfloor \tfrac{p-1}{d_i} \rfloor, \lfloor \tfrac{p-1}{d_i} \rfloor +1\}$, where the exact correspondence between the elements of each set depends on the choice of $p$.
We reduce Definition \ref{def_GFn} modulo $p^3$ and use Proposition \ref{prop_pGamma} to expand the terms involved, noting that $\frac{1}{1-p}\equiv 1+p+p^2 \pmod{p^3}$, to get
\begin{multline*}
 {_{4}G} \left(\tfrac{1}{d_1} , 1-\tfrac{1}{d_1}, \tfrac{1}{d_2} , 1-\tfrac{1}{d_2}\right)_p\\
\shoveleft \equiv
\sum_{j=0}^{\lfloor \frac{p-1}{d_1} \rfloor} 
\frac{\prod_{i=1}^{2} \biggfp{\frac{1}{d_i}+j+jp+jp^2}\biggfp{\frac{d_i-1}{d_i}+j+jp+jp^2}}
{\prod_{i=1}^{2} \biggfp{\frac{1}{d_i}}\biggfp{\frac{d_i-1}{d_i}}
{\biggfp{1+j+jp+jp^2}}^{4}} 
\\ 
\phantom{\equiv\:} \shoveleft  +p \left[
\sum_{j=0}^{\lfloor \frac{p-1}{d_1} \rfloor} 
\frac{\prod_{i=1}^{2} \biggfp{\frac{1}{d_i}+j+jp}\biggfp{\frac{d_i-1}{d_1}+j+jp}}
{\prod_{i=1}^{2} \biggfp{\frac{1}{d_i}}\biggfp{\frac{d_i-1}{d_i}}
{\biggfp{1+j+jp}}^{4}} 
\right. \\  \left.
+ \sum_{j=\lfloor \frac{p-1}{d_1} \rfloor+1}^{\lfloor \frac{p-1}{d_2} \rfloor} 
\frac{\prod_{i=1}^{2} \biggfp{\frac{1}{d_i}+j+jp}\biggfp{\frac{d_i-1}{d_i}+j+jp}
 \bigl(\frac{1}{d_1}+j+jp\bigr)}
{\prod_{i=1}^{2} \biggfp{\frac{1}{d_i}}\biggfp{\frac{d_i-1}{d_i}}
{\biggfp{1+j+jp}}^{4}} 
\right] \\ 
\phantom{\equiv\:} \shoveleft + p^2 \left[
\sum_{j=0}^{\lfloor \frac{p-1}{d_1} \rfloor} 
\frac{\prod_{i=1}^{2} \biggfp{\frac{1}{d_i}+j}\biggfp{\frac{d_i-1}{d_i}+j}}
{\prod_{i=1}^{2} \biggfp{\frac{1}{d_i}}\biggfp{\frac{d_i-1}{d_i}}
{\biggfp{1+j}}^{4}} 
\right. \\ \left.
+\sum_{j=\lfloor \frac{p-1}{d_1} \rfloor+1}^{\lfloor \frac{p-1}{d_2} \rfloor} 
\frac{\prod_{i=1}^{2} \biggfp{\frac{1}{d_i}+j}\biggfp{\frac{d_i-1}{d_i}+j}
\bigl(\frac{1}{d_1}+j\bigr)}
{\prod_{i=1}^{2} \biggfp{\frac{1}{d_i}}\biggfp{\frac{d_i-1}{d_i}}
{\biggfp{1+j}}^{4}} 
\right. \\ \left. 
+\sum_{j=\lfloor \frac{p-1}{d_2} \rfloor+1}^{\lfloor (d_2-1)\frac{p-1}{d_2} \rfloor} 
\frac{\prod_{i=1}^{2}\biggfp{\frac{1}{d_i}+j}\biggfp{\frac{d_i-1}{d_i}+j}
\bigl(\frac{1}{d_1}+j\bigr)\bigl(\frac{1}{d_2}+j\bigr)}
{\prod_{i=1}^{2} \biggfp{\frac{1}{d_i}}\biggfp{\frac{d_i-1}{d_i}}
{\biggfp{1+j}}^{4}} 
\right] 
\pmod{p^3}.
\end{multline*}

\noindent Choose ${m_1} \in \Bigl\{\frac{1}{d_1}, \frac{d_1-1}{d_1}\Bigr\}$ such that $rep\bigl(m_1\bigr) = Max\left(rep\bigl(\frac{1}{d_1}\bigr) ,rep\bigl(\frac{d_1-1}{d_1}\bigr)\right)$ and ${m_2} \in \Bigl\{\frac{1}{d_2}, \frac{d_2-1}{d_2}\Bigr\}$ such that $rep\bigl(m_2\bigr) = Max\left(rep\bigl(\frac{1}{d_2}\bigr) ,rep\bigl(\frac{d_2-1}{d_2}\bigr)\right)$. Let ${m_4}=1-m_1$ and ${m_3}=1-m_2$. Then $rep\bigl(m_4\bigr)\leq rep\bigl(m_3\bigr) \leq rep\bigl(m_2\bigr) \leq rep\bigl(m_1\bigr)$.
By Proposition \ref{prop_pGammaCong} we see that
\begin{align*}
\prod_{i=1}^{2} & \biggfp{\tfrac{1}{d_i}+j+jp+jp^2}\biggfp{\tfrac{d_i-1}{d_i}+j+jp+jp^2}
\equiv
\prod_{k=1}^{4} \biggfp{m_k+j+jp+jp^2}\\
&\qquad \qquad \equiv
\prod_{k=1}^{4} \biggfp{m_k+j}
\left[1+(jp+jp^2)G_1\left(m_k+j\right)+\tfrac{j^2p^2}{2}G_2\left(m_k+j\right)\right]  
\pmod{p^3},
\end{align*}
and
\begin{align*}
\biggfp{1+j+jp+jp^2}^{4}&
&\equiv \biggfp{1+j}^4
\left[1+(jp+jp^2)G_1\left(1+j\right)
+\tfrac{j^2p^2}{2}G_2\left(1+j\right)\right]^4
\pmod{p^3}.
\end{align*}

\noindent Multiplying the numerator and denominator by $${1-4(jp+jp^2)G_1\left(1+j\right)-2j^2p^2\left(G_2\left(1+j\right)-5G_1\left(1+j\right)^2\right)}$$ we get that
\begin{equation*}
\frac{\displaystyle \prod_{k=1}^{4} \left[1+(jp+jp^2)G_1\left(m_k+j\right)+\tfrac{j^2p^2}{2}G_2\left(m_k+j\right)\right] }
{\left[1+(jp+jp^2)G_1\left(1+j\right)
+\frac{j^2p^2}{2}G_2\left(1+j\right)\right]^4}
\equiv
1+(jp+jp^2)A(j)+j^2p^2B(j)
\pmod{p^3},
\end{equation*}
where
$$A(j):=\sum_{k=1}^{4} \Bigl(G_1\left(m_k+j\right)-G_1\left(1+j\right)\Bigr)$$
and
$$B(j):=\frac{1}{2}\left[A(j)^2 - \sum_{k=1}^{4} \Bigl(G_1\left(m_k+j\right)^2-G_2\left(m_k+j\right)
-G_1\left(1+j\right)^2 +G_2\left(1+j\right)\Bigr)\right].$$
We note that both $A(j)$ and $B(j)$ $\in \mathbb{Z}_p$ by Proposition \ref{prop_pGammaCong}. Applying the above and (\ref{for_GammapPlus}) we get
\begin{multline}\label{for_4G1}
{_{4}G} \left(\tfrac{1}{d_1} , 1-\tfrac{1}{d_1}, \tfrac{1}{d_2} , 1-\tfrac{1}{d_2}\right)_p
\equiv
\sum_{j=0}^{\lfloor \frac{p-1}{d_1} \rfloor} 
\frac{\prod_{k=1}^{4} \biggfp{m_k+j}}
{\prod_{k=1}^{4} \biggfp{m_k} j!^4}
\\ 
 +p \left[
\sum_{j=0}^{\lfloor \frac{p-1}{d_1} \rfloor} 
\frac{\prod_{k=1}^{4} \biggfp{m_k+j}}
{\prod_{k=1}^{4} \biggfp{m_k} j!^4}
\Bigl[1+jA(j)\Bigr]
+ \sum_{j=\lfloor \frac{p-1}{d_1} \rfloor+1}^{\lfloor \frac{p-1}{d_2} \rfloor} 
\frac{\prod_{k=1}^{4} \biggfp{m_k+j} }
{\prod_{k=1}^{4} \biggfp{m_k} j!^4}
\Bigl[\tfrac{1}{d_1}+j\Bigr]
\right] \\ 
\shoveleft +p^2 \left[
\sum_{j=0}^{\lfloor \frac{p-1}{d_1} \rfloor} 
\frac{\prod_{k=1}^{4} \biggfp{m_k+j}}
{\prod_{k=1}^{4} \biggfp{m_k} j!^4}
\Bigl[1+2jA(j)+j^2B(j)\Bigr]
\right. \\ \left. \qquad
+\sum_{j=\lfloor \frac{p-1}{d_1} \rfloor+1}^{\lfloor \frac{p-1}{d_2} \rfloor} 
\frac{\prod_{k=1}^{4} \biggfp{m_k+j}}
{\prod_{k=1}^{4} \biggfp{m_k} j!^4}
\Bigl[\Bigl(\tfrac{1}{d_1}+j\Bigr)\Bigl(1+jA(j)\Bigr)+j\Bigr]
\right. \\ \left.  \qquad
+\sum_{j=\lfloor \frac{p-1}{d_2} \rfloor+1}^{\lfloor (d_2-1)\frac{p-1}{d_2} \rfloor} 
\frac{\prod_{k=1}^{4} \biggfp{m_k+j}}
{\prod_{k=1}^{4} \biggfp{m_k} j!^4}
\Bigl[\Bigl(\tfrac{1}{d_1}+j\Bigr)\Bigl(\tfrac{1}{d_2}+j\Bigr)\Bigr]
\right] 
\pmod{p^3}.
\end{multline}

\noindent Lemma \ref{lem_gammapd2} gives us
\begin{multline}\label{for_4F31}
{{_{4}F_3} \left[ \begin{array}{cccc} \frac{1}{d_1}, & 1-\frac{1}{d_1}, & \frac{1}{d_2}, & 1-\frac{1}{d_2}\vspace{.05in}\\
\phantom{\frac{1}{d_1}} & 1, & 1, & 1 \end{array}
\Big| \; 1 \right]}_{p-1}
\equiv
\sum_{j=0}^{\lfloor \frac{p-1}{d_1} \rfloor}
\prod_{k=1}^{4} \frac{\biggfp{m_k+j}}{\biggfp{m_k} j!^4}\\
+p\sum_{j=\lfloor \frac{p-1}{d_1} \rfloor +1}^{\lfloor \frac{p-1}{d_2} \rfloor } 
\prod_{k=1}^{4} \frac{\biggfp{m_k+j}}{\biggfp{m_k} j!^4}
\Bigl(\tfrac{1}{d_1}\Bigr)
+p^2\sum_{j=\lfloor \frac{p-1}{d_2} \rfloor +1}^{p-\lfloor \frac{p-1}{d_2} \rfloor - 1} 
\prod_{k=1}^{4} \frac{\biggfp{m_k+j}}{\biggfp{m_k} j!^4}
\Bigl(\tfrac{1}{d_1 d_2}\Bigr)
\pmod {p^3}.
\end{multline}

\noindent Combining (\ref{for_4G1}) and (\ref{for_4F31}) it suffices to show
\begin{multline}\label{Resid1_4G1}
\sum_{j=0}^{\lfloor \frac{p-1}{d_1} \rfloor} 
\frac{\prod_{k=1}^{4} \biggfp{m_k+j}}
{\prod_{k=1}^{4} \biggfp{m_k} j!^4}
\Bigl[1+jA(j)\Bigr]
+ \sum_{j=\lfloor \frac{p-1}{d_1} \rfloor+1}^{\lfloor \frac{p-1}{d_2} \rfloor} 
\frac{\prod_{k=1}^{4} \biggfp{m_k+j} }
{\prod_{k=1}^{4} \biggfp{m_k} j!^4}
\Bigl[j\Bigr]
 \\ 
\shoveleft +p \left[
\sum_{j=0}^{\lfloor \frac{p-1}{d_1} \rfloor} 
\frac{\prod_{k=1}^{4} \biggfp{m_k+j}}
{\prod_{k=1}^{4} \biggfp{m_k} j!^4}
\Bigl[1+2jA(j)+j^2B(j)\Bigr]
\right. \\ \left.
+\sum_{j=\lfloor \frac{p-1}{d_1} \rfloor+1}^{\lfloor \frac{p-1}{d_2} \rfloor} 
\frac{\prod_{k=1}^{4} \biggfp{m_k+j}}
{\prod_{k=1}^{4} \biggfp{m_k} j!^4}
\Bigl[\Bigl(\tfrac{1}{d_1}+j\Bigr)\Bigl(1+jA(j)\Bigr)+j\Bigr]
\right. \\ \left.  \qquad
+\sum_{j=\lfloor \frac{p-1}{d_2} \rfloor+1}^{\lfloor (d_2-1)\frac{p-1}{d_2} \rfloor} 
\frac{\prod_{k=1}^{4} \biggfp{m_k+j}}
{\prod_{k=1}^{4} \biggfp{m_k} j!^4}
\Bigl[j^2+j \Bigl(\tfrac{1}{d_1}+\tfrac{1}{d_2}\Bigr)\Bigr]
\right. \\ \left. 
-\sum_{j=\lfloor (d_2-1)\frac{p-1}{d_2} \rfloor+1}^{p-\lfloor \tfrac{p-1}{d_2} \rfloor - 1}
\frac{\prod_{k=1}^{4} \biggfp{m_k+j}}
{\prod_{k=1}^{4} \biggfp{m_k} j!^4}
\Bigl[\Bigl(\tfrac{1}{d_1}\Bigr)\Bigl(\tfrac{1}{d_2}\Bigr)\Bigr]
\right] 
\equiv s(p) 
\pmod{p^2}.
\end{multline}
If $p\equiv 1\pmod{d_2}$ then the last sum above is vacuous. If $p\not\equiv 1 \pmod{d_2}$ then the limits of summation are equal and the sum is over one value of $j=\lfloor (d_2-1)\frac{p-1}{d_2} \rfloor+1=p-\lfloor \tfrac{p-1}{d_2} \rfloor - 1$.

We now examine $A(j)$, $B(j)$, $\prod_{k=1}^{4} \biggfp{m_k+j}$ and $\prod_{k=1}^{4} \biggfp{m_k}$ modulo $p$. 
Using Corollary \ref{cor_pGammaCong1} (3) and Proposition \ref{prop_pGammaG} (1) we see that
\begin{align*}
A(j)
&\equiv \sum_{i=1}^{2} \Bigl(H_{p-\lfloor \frac{p-1}{d_i} \rfloor-1+j}^{(1)} +  H_{\lfloor \frac{p-1}{d_i} \rfloor +j}^{(1)}-2 \hspace{1pt} H_{j}^{(1)}\Bigr)-
\begin{cases}
0 & \text{if } 0\leq j\leq \lfloor \frac{p-1}{d_1}\rfloor, \\[6pt]
\frac{1}{p} & \text{if }  \lfloor \frac{p-1}{d_1}\rfloor+1\leq j\leq \lfloor \frac{p-1}{d_2}\rfloor,\\[6pt]
\frac{2}{p} & \text{if } \lfloor \frac{p-1}{d_2}\rfloor+1\leq j\leq p-\lfloor \frac{p-1}{d_2}\rfloor-1,\\[6pt]
\frac{3}{p} & \text{if }  p-\lfloor \frac{p-1}{d_2}\rfloor\leq j\leq p-\lfloor \frac{p-1}{d_1}\rfloor-1,\\[6pt]
\frac{4}{p} & \text{if }  p-\lfloor \frac{p-1}{d_1}\rfloor\leq j\leq p-1,
\end{cases}
\pmod{p}.
\end{align*}

\noindent Similarly, using Corollary \ref{cor_pGammaCong1} (3), (4) and Proposition \ref{prop_pGammaG} (2), we have that
\begin{align*}
&\sum_{k=1}^{4}  \Bigl(G_1\left(m_k+j\right)^2-G_2\left(m_k+j\right)
-G_1\left(1+j\right)^2 +G_2\left(1+j\right)\Bigr)\\
&\equiv \sum_{i=1}^{2} \Bigl(H_{p-\lfloor \frac{p-1}{d_i} \rfloor-1+j}^{(2)} +  H_{\lfloor \frac{p-1}{d_i} \rfloor +j}^{(2)}-2 \hspace{1pt} H_{j}^{(2)}\Bigr)+
\begin{cases}
0 & \text{if } 0\leq j\leq \lfloor \frac{p-1}{d_1}\rfloor, \\[6pt]
\frac{1}{p^2} &\text{if }  \lfloor \frac{p-1}{d_1}\rfloor+1\leq j\leq \lfloor \frac{p-1}{d_2}\rfloor,\\[6pt]
\frac{2}{p^2} & \text{if }  \lfloor \frac{p-1}{d_2}\rfloor+1\leq j\leq p-\lfloor \frac{p-1}{d_2}\rfloor-1,\\[6pt]
\frac{3}{p^2} & \text{if }  p-\lfloor \frac{p-1}{d_2}\rfloor\leq j\leq p-\lfloor \frac{p-1}{d_1}\rfloor-1,\\[6pt]
\frac{4}{p^2} & \text{if }  p-\lfloor \frac{p-1}{d_1}\rfloor\leq j\leq p-1,\\[6pt]
\end{cases}
\pmod{p}.
\end{align*}

\noindent Using Proposition \ref{prop_pGamma} we get
\begin{align*}
\prod_{k=1}^{4} & \biggfp{m_k+j}\\
&\equiv \prod_{k=1}^{4} \left(rep(m_k)+j-1\right)! (-1)^{rep(m_k)+j}
\begin{cases}
1 &\text{if } 0\leq j \leq rep(m_{5-k})-1,\\
\frac{1}{p} & \text{if } rep(m_{5-k})\leq j \leq p-1,
\end{cases}\\[6pt]
&\equiv \prod_{i=1}^{2}\Bigl( \left(p-\lfloor \tfrac{p-1}{d_i} \rfloor-1+j\right)! \left(\lfloor \tfrac{p-1}{d_i} \rfloor +j\right)! \Bigr)
\begin{cases}
1 & \text{if } 0\leq j\leq \lfloor \frac{p-1}{d_1}\rfloor, \\[6pt]
\frac{1}{p} & \text{if }  \lfloor \frac{p-1}{d_1}\rfloor+1\leq j\leq \lfloor \frac{p-1}{d_2}\rfloor,\\[6pt]
\frac{1}{p^2} & \text{if } \lfloor \frac{p-1}{d_2}\rfloor+1\leq j\leq p-\lfloor \frac{p-1}{d_2}\rfloor-1,\\[6pt]
\frac{1}{p^3} & \text{if }  p-\lfloor \frac{p-1}{d_2}\rfloor\leq j\leq p-\lfloor \frac{p-1}{d_1}\rfloor-1,\\[6pt]
\frac{1}{p^4} & \text{if }  p-\lfloor \frac{p-1}{d_1}\rfloor\leq j\leq p-1,\\[6pt]
\end{cases}
\pmod{p}
\end{align*}
and
\begin{align*}
\prod_{k=1}^{4} &\biggfp{m_k} = \prod_{i=1}^{2} \biggfp{\tfrac{1}{d_i}}\biggfp{\tfrac{d_i-1}{d_i}}
=\pm1\;.
\end{align*}

\noindent We now consider
\begin{multline*}
X(j):=\sum_{j=0}^{\lfloor \frac{p-1}{d_1} \rfloor} 
\frac{\prod_{k=1}^{4} \biggfp{m_k+j}}
{\prod_{k=1}^{4} \biggfp{m_k} j!^4}
\Bigl[1+2jA(j)+j^2B(j)\Bigr]
\\
+\sum_{j=\lfloor \frac{p-1}{d_1} \rfloor+1}^{\lfloor \frac{p-1}{d_2} \rfloor} 
\frac{\prod_{k=1}^{4} \biggfp{m_k+j}}
{\prod_{k=1}^{4} \biggfp{m_k} j!^4}
\Bigl[2j+j^2A(j)\Bigr]
+\sum_{j=\lfloor \frac{p-1}{d_2} \rfloor+1}^{p-\lfloor \tfrac{p-1}{d_2} \rfloor - 1} 
\frac{\prod_{k=1}^{4} \biggfp{m_k+j}}
{\prod_{k=1}^{4} \biggfp{m_k} j!^4}
\Bigl[j^2\Bigr]
\pmod{p}.
\end{multline*}

\noindent We will show that $X(j)\equiv 0 \pmod{p}$. Substituting for $A(j)$, $B(j)$, $\prod_{k=1}^{4} \biggfp{m_k+j}$ and $\prod_{k=1}^{4} \biggfp{m_k}$ modulo $p$ yields
\begin{multline*}
\pm X(j)
\equiv
\sum_{j=0}^{\lfloor \frac{p-1}{d_1} \rfloor} 
\Biggl[\prod_{i=1}^{2} \ph{j+1}{p-\lfloor \frac{p-1}{d_i} \rfloor-1}\ph{j+1}{\lfloor \frac{p-1}{d_i} \rfloor} \Biggr]
\Biggl[1+2j\sum_{i=1}^{2} \Bigl(H_{p-\lfloor \frac{p-1}{d_i} \rfloor-1+j}^{(1)} +  H_{\lfloor \frac{p-1}{d_i} \rfloor +j}^{(1)}-2 \hspace{1pt} H_{j}^{(1)}\Bigr)
\\
+\frac{j^2}{2} \Biggl[\left(\sum_{i=1}^{2} \Bigl(H_{p-\lfloor \frac{p-1}{d_i} \rfloor-1+j}^{(1)} +  H_{\lfloor \frac{p-1}{d_i} \rfloor +j}^{(1)}-2 \hspace{1pt} H_{j}^{(1)}\Bigr)\right)^2
-\sum_{i=1}^{2} \Bigl(H_{p-\lfloor \frac{p-1}{d_i} \rfloor-1+j}^{(2)} +  H_{\lfloor \frac{p-1}{d_i} \rfloor +j}^{(2)}-2 \hspace{1pt} H_{j}^{(2)}\Bigr)
\Biggr]
\Biggr]
\\
+\sum_{j=\lfloor \frac{p-1}{d_1} \rfloor+1}^{\lfloor \frac{p-1}{d_2} \rfloor}
\Biggl[\prod_{i=1}^{2} \ph{j+1}{p-\lfloor \frac{p-1}{d_i} \rfloor-1}\ph{j+1}{\lfloor \frac{p-1}{d_i} \rfloor} \Biggr]
\Biggl[ \frac{1}{p} \Biggr]
\Biggl[2j+j^2 \sum_{i=1}^{2} \Bigl(H_{p-\lfloor \frac{p-1}{d_i} \rfloor-1+j}^{(1)} +  H_{\lfloor \frac{p-1}{d_i} \rfloor +j}^{(1)}-2 \hspace{1pt} H_{j}^{(1)}\Bigr)
\\
\shoveright {-\frac{j^2}{p} \Biggr]
+\sum_{j=\lfloor \frac{p-1}{d_2} \rfloor+1}^{p-\lfloor \tfrac{p-1}{d_2} \rfloor - 1} 
\Biggl[\prod_{i=1}^{2} \ph{j+1}{p-\lfloor \frac{p-1}{d_i} \rfloor-1}\ph{j+1}{\lfloor \frac{p-1}{d_i} \rfloor} \Biggr]
\Biggl[ \frac{1}{p^2} \Biggr]
\Biggl[ j^2 \Biggr]} \\
\equiv
\sum_{j=0}^{p-\lfloor \frac{p-1}{d_2} \rfloor - 1}  
\Biggl[\prod_{i=1}^{2} \ph{j+1}{p-\lfloor \frac{p-1}{d_i} \rfloor-1}\ph{j+1}{\lfloor \frac{p-1}{d_i} \rfloor} \Biggr]
\Biggl[1+2j\sum_{i=1}^{2} \Bigl(H_{p-\lfloor \frac{p-1}{d_i} \rfloor-1+j}^{(1)} +  H_{\lfloor \frac{p-1}{d_i} \rfloor +j}^{(1)}-2 \hspace{1pt} H_{j}^{(1)}\Bigr)
\\
+\frac{j^2}{2} \Biggl[\left(\sum_{i=1}^{2} \Bigl(H_{p-\lfloor \frac{p-1}{d_i} \rfloor-1+j}^{(1)} +  H_{\lfloor \frac{p-1}{d_i} \rfloor +j}^{(1)}-2 \hspace{1pt} H_{j}^{(1)}\Bigr)\right)^2
-\sum_{i=1}^{2} \Bigl(H_{p-\lfloor \frac{p-1}{d_i} \rfloor-1+j}^{(2)} +  H_{\lfloor \frac{p-1}{d_i} \rfloor +j}^{(2)}-2 \hspace{1pt} H_{j}^{(2)}\Bigr)
\Biggr]
\Biggr]\\[6pt]
\pmod{p}.
\end{multline*}

\noindent Note we can extend the upper limit of this sum to $j=p-1$ as $\prod_{i=1}^{2} \ph{j+1}{p-\lfloor \frac{p-1}{d_i} \rfloor-1}\ph{j+1}{\lfloor \frac{p-1}{d_i} \rfloor}\\
 \in p^3\mathbb{Z}_p$ for $p-\lfloor \frac{p-1}{d_2}\rfloor \leq j\leq p-1$. Define
\begin{equation*}
P(j):= \frac{d}{dj} \left[ j \hspace{1pt} \prod_{i=1}^{2} \ph{j+1}{p-\lfloor \frac{p-1}{d_i} \rfloor-1}\ph{j+1}{\lfloor \frac{p-1}{d_i} \rfloor} \right] = \sum_{k=0}^{2(p-1)} a_k j^k \;.
\end{equation*}
and
\begin{equation*}
Q(j):= \frac{j}{2} \frac{d^2}{dj^2} \left[ j \hspace{1pt} \prod_{i=1}^{2} \ph{j+1}{p-\lfloor \frac{p-1}{d_i} \rfloor-1}\ph{j+1}{\lfloor \frac{p-1}{d_i} \rfloor} \right] = \sum_{k=0}^{2(p-1)} b_k j^k \;.
\end{equation*}
Then
\begin{multline*}
P(j)=
\Biggl[\prod_{i=1}^{2} \ph{j+1}{p-\lfloor \frac{p-1}{d_i} \rfloor-1}\ph{j+1}{\lfloor \frac{p-1}{d_i} \rfloor} \Biggr]
\Biggl[1+j\sum_{i=1}^{2} \Bigl(H_{p-\lfloor \frac{p-1}{d_i} \rfloor-1+j}^{(1)} +  H_{\lfloor \frac{p-1}{d_i} \rfloor +j}^{(1)}-2 \hspace{1pt} H_{j}^{(1)}\Bigr)\Biggr]
\end{multline*}
and
\begin{multline*}
Q(j)=
\Biggl[\prod_{i=1}^{2} \ph{j+1}{p-\lfloor \frac{p-1}{d_i} \rfloor-1}\ph{j+1}{\lfloor \frac{p-1}{d_i} \rfloor} \Biggr]
\Biggl[j\sum_{i=1}^{2} \Bigl(H_{p-\lfloor \frac{p-1}{d_i} \rfloor-1+j}^{(1)} +  H_{\lfloor \frac{p-1}{d_i} \rfloor +j}^{(1)}-2 \hspace{1pt} H_{j}^{(1)}\Bigr)
\\
+\frac{j^2}{2} \Biggl[\left(\sum_{i=1}^{2} \Bigl(H_{p-\lfloor \frac{p-1}{d_i} \rfloor-1+j}^{(1)} +  H_{\lfloor \frac{p-1}{d_i} \rfloor +j}^{(1)}-2 \hspace{1pt} H_{j}^{(1)}\Bigr)\right)^2
-\sum_{i=1}^{2} \Bigl(H_{p-\lfloor \frac{p-1}{d_i} \rfloor-1+j}^{(2)} +  H_{\lfloor \frac{p-1}{d_i} \rfloor +j}^{(2)}-2 \hspace{1pt} H_{j}^{(2)}\Bigr)
\Biggr]
\Biggr].\\
\end{multline*}

\noindent Applying (\ref{exp_sums}) yields
\begin{equation*}
 \sum_{j=0}^{p-1} P(j)=\sum_{j=0}^{p-1} \sum_{k=0}^{2(p-1)} a_k j^k =  p \hspace{1pt}a_0 + \sum_{j=1}^{p-1} \sum_{k=1}^{2(p-1)} a_k j^k \equiv   \sum_{k=1}^{2(p-1)} a_k \sum_{j=1}^{p-1}  j^k
\equiv -a_{p-1} -a_{2p-2}\pmod{p}.
\end{equation*}
By definition of $P(j)$ we get that
\begin{equation*}
\prod_{i=1}^{2} \ph{j+1}{p-\lfloor \frac{p-1}{d_i} \rfloor-1}\ph{j+1}{\lfloor \frac{p-1}{d_i}\rfloor}  = \sum_{k=0}^{2(p-1)} \frac{a_k}{k+1} j^k \;.
\end{equation*}
$P(j)$ is a monic polynomial with integer coefficients so $a_{2p-2}=2p-1$ and $p\mid a_{p-1}$. Therefore
$a_{2p-2}\equiv -1\pmod{p}$, $a_{p-1}\equiv 0 \pmod{p}$ and
\begin{equation*}
\sum_{j=0}^{p-1} P(j) \equiv 1 \pmod{p}.
\end{equation*}

\noindent Similarly
\begin{equation*}
 \sum_{j=0}^{p-1} Q(j)=\sum_{j=0}^{p-1} \sum_{k=0}^{2(p-1)} b_k j^k =  p \hspace{1pt}a_0 + \sum_{j=1}^{p-1} \sum_{k=1}^{2(p-1)} b_k j^k \equiv   \sum_{k=1}^{2(p-1)} b_k \sum_{j=1}^{p-1}  j^k
\equiv -b_{p-1} -b_{2p-2}\pmod{p}.
\end{equation*}
By definition of $Q(j)$ we get that
\begin{equation*}
\prod_{i=1}^{2} \ph{j+1}{p-\lfloor \frac{p-1}{d_i} \rfloor-1}\ph{j+1}{\lfloor \frac{p-1}{d_i}\rfloor}  = 2\sum_{k=0}^{2(p-1)} \frac{b_k}{(k)(k+1)} j^k \;.
\end{equation*}
$Q(j)$ is a monic polynomial with integer coefficients so $2 a_{2p-2}=(2p-2)(2p-1)$ and $p\mid a_{p-1}$. Therefore
$a_{2p-2}\equiv 1\pmod{p}$, $a_{p-1}\equiv 0 \pmod{p}$ and
\begin{equation*}
\sum_{j=0}^{p-1} Q(j) \equiv -1 \pmod{p}.
\end{equation*}
So
\begin{equation}\label{for_X0}
X(j)=\pm \left(\sum_{j=1}^{p-1} P(j)+Q(j) \right)\equiv 0 \pmod{p}.
\end{equation}

\noindent Accounting for (\ref{for_X0}) in (\ref{Resid1_4G1}) means we need only show
\begin{multline*}
\sum_{j=0}^{\lfloor \frac{p-1}{d_1} \rfloor} 
\frac{\prod_{k=1}^{4} \biggfp{m_k+j}}
{\prod_{k=1}^{4} \biggfp{m_k} j!^4}
\Bigl[1+jA(j)\Bigr]
+ \sum_{j=\lfloor \frac{p-1}{d_1} \rfloor+1}^{\lfloor \frac{p-1}{d_2} \rfloor} 
\frac{\prod_{k=1}^{4} \biggfp{m_k+j} }
{\prod_{k=1}^{4} \biggfp{m_k} j!^4}
\Bigl[j\Bigr]
 \\ \quad
\shoveleft +p \left[
\sum_{j=\lfloor \frac{p-1}{d_1} \rfloor+1}^{\lfloor \frac{p-1}{d_2} \rfloor} 
\frac{\prod_{k=1}^{4} \biggfp{m_k+j}}
{\prod_{k=1}^{4} \biggfp{m_k} j!^4}
\Bigl[\Bigl(\tfrac{1}{d_1}\Bigr)\Bigl(1+jA(j)\Bigr)\Bigr]
\right. \\ \left. 
+\sum_{j=\lfloor \frac{p-1}{d_2} \rfloor+1}^{\lfloor (d_2-1)\frac{p-1}{d_2} \rfloor} 
\frac{\prod_{k=1}^{4} \biggfp{m_k+j}}
{\prod_{k=1}^{4} \biggfp{m_k} j!^4}
\Bigl[j \Bigl(\tfrac{1}{d_1}+\tfrac{1}{d_2}\Bigr)\Bigr]
\right. \\ \left. 
-\sum_{j=\lfloor (d_2-1)\frac{p-1}{d_2} \rfloor+1}^{p-\lfloor \frac{p-1}{d_2} \rfloor - 1}
\frac{\prod_{k=1}^{4} \biggfp{m_k+j}}
{\prod_{k=1}^{4} \biggfp{m_k} j!^4}
\Bigl[\Bigl(\tfrac{1}{d_1}\Bigr)\Bigl(\tfrac{1}{d_2}\Bigr)+j^2\Bigr]
\right] 
\equiv s(p) 
\pmod{p^2}.
\end{multline*}

\noindent We now convert these remaining terms to an expression involving binomial coefficients and harmonic sums and then use the results of Section 3 to simplify them. First we define 
\begin{align*}
Bin \hspace{1pt}(j) &:=  \bin{rep(m_1)-1+j}{j} \bin{rep(m_1)-1}{j}  \bin{rep(m_2)-1+j}{j} \bin{rep(m_2)-1}{j},\\[15pt]
\mathcal{H}(j)&:= H_{rep(m_1)-1+j}^{(1)}+H_{rep(m_1)-1-j}^{(1)}+H_{rep(m_2)-1+j}^{(1)}+H_{rep(m_2)-1-j}^{(1)}-4H_j^{(1)},\\[15pt]
\mathcal{A}(j)&:= \Bigl(rep(m_1)-m_1\Bigr)\left(H_{rep(m_1)-1+j}^{(1)}-H_{rep(m_{4})-1+j}^{(1)}\right)\\[6pt]
&  \qquad \qquad \qquad \qquad\qquad \qquad
+ \Bigl(rep(m_2)-m_2\Bigr)\left(H_{rep(m_2)-1+j}^{(1)}-H_{rep(m_{3})-1+j}^{(1)}\right),\\
\intertext{and}
\mathcal{B}(j)&:= \Bigl(rep(m_1)-m_1\Bigr)\left(H_{rep(m_1)-1+j}^{(2)}-H_{rep(m_{4})-1+j}^{(2)}\right)\\[6pt]
&  \qquad \qquad \qquad \qquad\qquad \qquad
+ \Bigl(rep(m_2)-m_2\Bigr)\left(H_{rep(m_2)-1+j}^{(2)}-H_{rep(m_{3})-1+j}^{(2)}\right)\;.\\[3pt]
\end{align*}

\noindent By Lemma \ref{lem_ProdGammap2} we see that for $j<rep(\frac{m_2}{d})=p-\lfloor \tfrac{p-1}{d_2} \rfloor$,
\begin{align*}
\frac{\prod_{k=1}^{4} \biggfp{m_k+j}}{\prod_{k=1}^{4}\biggfp{m_k} j!^4}
&\equiv 
\Biggl[ \prod_{i=1}^{2} \bin{rep(m_i)-1+j}{j} \bin{rep(m_i)-1}{j}\Biggr]\Biggl[\delta^{\prime}\Biggr]\\
&\quad \cdot \left[1-\sum_{i=1}^{2}\Bigl(rep(m_i)-m_i\Bigr)\left(H_{rep(m_i)-1+j}^{(1)}-H_{rep(m_{5-i})-1+j}^{(1)}-\delta_i\right)\right]\\
&\equiv Bin \hspace{1pt}(j) \Biggl[\delta^{\prime}\Biggr] \left[1-\mathcal{A}(j) + \sum_{i=1}^{2}\Bigl(rep(m_i)-m_i\Bigr) \delta_i\right]
\pmod{p^2},
\end{align*}
where 
\begin{align*}
\delta^{\prime}=
\begin{cases}
1 & \text{if } 0\leq j \leq rep(m_4)-1,\\
\frac{1}{p} & \text{if } rep(m_4) \leq j \leq rep(m_3)-1,\\
\frac{1}{p^2} & \text{if } rep(m_3) \leq j < rep(m_2),
\end{cases}
& \qquad and \qquad
\delta_i=
\begin{cases}
0 & \text{if } 0\leq j \leq rep(m_{5-i})-1,\\
\frac{1}{p} & \text{if } rep(m_{5-i}) \leq j < rep(m_i).
\end{cases}
\end{align*}

\noindent Again for $j<rep(\frac{m_2}{d})=p-\lfloor \tfrac{p-1}{d_2} \rfloor$ Lemma \ref{lem_SumGammap2} gives us 
\begin{align*}
A(j):&=\sum_{k=1}^{4} \Bigl(G_1\left(m_k+j\right)-G_1\left(1+j\right)\Bigr)\\
&\equiv \sum_{i=1}^{2} 
\left( H_{rep(m_i)-1+j}^{(1)}+H_{rep(m_i)-1-j}^{(1)}-2\hspace{1pt} H_{j}^{(1)}\right)-\alpha^{\prime}\\
&\qquad\qquad+\sum_{i=1}^{2} \left(rep(m_i)-m_i\right) 
\left(H_{rep(m_i)-1+j}^{(2)}-H_{rep(m_{5-i})-1+j}^{(2)}-\beta_{i}\right)\\
&\equiv \mathcal{H}(j) -\alpha^{\prime} + \mathcal{B}(j) - \sum_{i=1}^{2} \left(rep(m_i)-m_i\right) \beta_{i}
\pmod{p^2}
\end{align*}
where 
\begin{align*}
\alpha^{\prime}=
\begin{cases}
0 & \text{if } 0\leq j \leq rep(m_4)-1,\\
\frac{1}{p} & \text{if } rep(m_4) \leq j \leq rep(m_3)-1,\\
\frac{2}{p} & \text{if } rep(m_3) \leq j < rep(m_2),\;
\end{cases}
& \qquad and \qquad
\beta_{i}=
\begin{cases}
0 & \text{if } 0\leq j \leq rep(m_{5-i})-1,\\
\frac{1}{p^2} & \text{if } rep(m_{5-i}) \leq j < rep(m_i).
\end{cases}
\end{align*}

\noindent Reducing the above results modulo $p$ we see that for $j<rep(\frac{m_2}{d})=p-\lfloor \tfrac{p-1}{d_2} \rfloor$,
\begin{align*}
\frac{\prod_{k=1}^{4} \biggfp{m_k+j}}{\prod_{k=1}^{4}\biggfp{m_k} j!^4}
&\equiv Bin \hspace{1pt}(j) \left[\delta^{\prime}\right]
\pmod{p}
\end{align*}
and
\begin{align*}
A(j) \equiv \mathcal{H}(j) -\alpha^{\prime} 
\pmod{p}\;.
\end{align*}

\noindent Thus it suffices to show
\begin{multline*}
\sum_{j=0}^{\lfloor \frac{p-1}{d_1} \rfloor} 
Bin \hspace{1pt}(j) \Bigl[1-\mathcal{A}(j) \Bigr]
\Bigl[1+j \mathcal{H}(j)  + j \mathcal{B}(j) \Bigr]
+ \sum_{j=\lfloor \frac{p-1}{d_1} \rfloor+1}^{\lfloor \frac{p-1}{d_2} \rfloor} 
Bin \hspace{1pt}(j) \Bigl[\tfrac{j}{p}\Bigr] \left[1-\mathcal{A}(j) + \tfrac{rep(m_1)-m_1}{p}\right]
 \\  \qquad
+p \left[
\sum_{j=\lfloor \frac{p-1}{d_1} \rfloor+1}^{\lfloor \frac{p-1}{d_2} \rfloor} 
Bin \hspace{1pt}(j) \Bigl[\tfrac{1}{p}\Bigr]
\Bigl[\Bigl(\tfrac{1}{d_1}\Bigr)\Bigl(1+j\mathcal{H}(j)-\tfrac{j}{p}\Bigr)\Bigr]
+\sum_{j=\lfloor \frac{p-1}{d_2} \rfloor+1}^{\lfloor (d_2-1)\frac{p-1}{d_2} \rfloor} 
Bin \hspace{1pt}(j) \Bigl[\tfrac{1}{p^2}\Bigr]
\Bigl[j \Bigl(\tfrac{1}{d_1}+\tfrac{1}{d_2}\Bigr)\Bigr]
\right. \\ \left.  \qquad
-\sum_{j=\lfloor (d_2-1)\frac{p-1}{d_2} \rfloor+1}^{p-\lfloor \frac{p-1}{d_2} \rfloor - 1}
Bin \hspace{1pt}(j) \Bigl[\tfrac{1}{p^2}\Bigr]
\Bigl[\Bigl(\tfrac{1}{d_1}\Bigr)\Bigl(\tfrac{1}{d_2}\Bigr)+j^2\Bigr]
\right] 
\equiv s(p) 
\pmod{p^2}.
\end{multline*}

\noindent We now consider
\begin{equation*}
\sum_{j=0}^{p-\lfloor \frac{p-1}{d_2} \rfloor - 1}
Bin \hspace{1pt}(j) \Bigl[1+j \mathcal{H}(j) \Bigr]
+\sum_{j=0}^{p-\lfloor \frac{p-1}{d_2} \rfloor - 1}
Bin \hspace{1pt}(j)
\Bigl[ j \mathcal{B}(j) -\mathcal{A}(j) - j \mathcal{A}(j) \mathcal{H}(j) \Bigr]
\pmod{p^2}.\\
\end{equation*}

\noindent For $0\leq j \leq rep(m_4)-1=\lfloor \frac{p-1}{d_1} \rfloor$ we see that
\begin{multline*}
Bin \hspace{1pt}(j) \Bigl[1-\mathcal{A}(j) \Bigr]
\Bigl[1+j \mathcal{H}(j)  + j \mathcal{B}(j) \Bigr]\\[6pt]
\equiv
Bin \hspace{1pt}(j) \Bigl[1+j \mathcal{H}(j)+ j \mathcal{B}(j) -\mathcal{A}(j) - j \mathcal{A}(j) \mathcal{H}(j) \Bigr]
\pmod{p^2},
\end{multline*}
as $ \mathcal{A}(j)  \mathcal{B}(j) \in p^2 \mathbb{Z}_p$ for such $j$.
\noindent For $\lfloor \frac{p-1}{d_1} \rfloor+1=rep(m_4) \leq j \leq rep(m_3)-1=\lfloor \frac{p-1}{d_2} \rfloor$ we have
\begin{align*}
Bin \hspace{1pt}(j)& \Bigl[1+j \mathcal{H}(j)+ j \mathcal{B}(j) -\mathcal{A}(j) - j \mathcal{A}(j) \mathcal{H}(j) \Bigr] \\[9pt]
&\equiv
Bin \hspace{1pt}(j) \Bigl[1+j \mathcal{H}(j) +j \tfrac{rep(m_1)-m_1}{p^2} - \tfrac{rep(m_1)-m_1}{p}-
j\left(\tfrac{1}{p} \left(\mathcal{A}(j)- \tfrac{rep(m_1)-m_1}{p}\right)
\right.\\[9pt] &\left. \qquad \qquad \qquad \qquad \qquad \qquad \qquad \qquad  \qquad \qquad
+\tfrac{rep(m_1)-m_1}{p} \left(\mathcal{H}(j)-\tfrac{1}{p}\right) +\tfrac{rep(m_1)-m_1}{p^2}\right)\Bigr]\\[9pt]
&\equiv
Bin \hspace{1pt}(j) \Bigl[\Bigl(1+j \mathcal{H}(j)\Bigr)\left(1- \tfrac{rep(m_1)-m_1}{p}\right) +2j \tfrac{rep(m_1)-m_1}{p^2} -
\tfrac{j}{p} \mathcal{A}(j)\Bigr]\\[9pt]
&\equiv
Bin \hspace{1pt}(j) \Bigl[\Bigl(1+j \mathcal{H}(j)\Bigr)\left(\tfrac{1}{d_1}\right) +\tfrac{2j}{p} \left(1-\tfrac{1}{d_1}\right) -
\tfrac{j}{p} \mathcal{A}(j)\Bigr]\\[9pt]
&\equiv
Bin \hspace{1pt}(j) \Bigl[\Bigl(\tfrac{1}{d_1}\Bigr)\Bigl(1+j\mathcal{H}(j)-\tfrac{j}{p}\Bigr)\Bigr]+
Bin \hspace{1pt}(j) \Bigl[\Bigl(\tfrac{j}{p}\Bigr)\Bigl(2-\tfrac{1}{d_1}-\mathcal{A}(j)\Bigr)\Bigr]\\[9pt]
&\equiv
Bin \hspace{1pt}(j) \Bigl[\Bigl(\tfrac{1}{d_1}\Bigr)\Bigl(1+j\mathcal{H}(j)-\tfrac{j}{p}\Bigr)\Bigr]+
Bin \hspace{1pt}(j) \Bigl[\Bigl(\tfrac{j}{p}\Bigr)\Bigl(1-\mathcal{A}(j)+\tfrac{rep(m_1)-m_1}{p}\Bigr)\Bigr]
\pmod{p^2}.
\end{align*}

\noindent Similarly, for $\lfloor \frac{p-1}{d_2} \rfloor+1=rep(m_3) \leq j \leq rep(m_2)-1=p-\lfloor \tfrac{p-1}{d_2} \rfloor - 1$, 
\begin{align*}
Bin \hspace{1pt}(j) &\Bigl[1+j \mathcal{H}(j)+ j \mathcal{B}(j) -\mathcal{A}(j) - j \mathcal{A}(j) \mathcal{H}(j) \Bigr] \\[9pt]
&\equiv
Bin \hspace{1pt}(j) \Bigl[0+\tfrac{2j}{p}+j\tfrac{rep(m_1)-m_1+rep(m_2)-m_2}{p^2}-0-j\left(\tfrac{2\left(rep(m_1)-m_1\right)}{p^2}+\tfrac{2\left(rep(m_2)-m_2\right)}{p^2}\right)\Bigr]\\[9pt]
&\equiv
Bin \hspace{1pt}(j) \Bigl[\tfrac{2j}{p}-\tfrac{j}{p^2}\left(rep(m_1)-m_1+rep(m_2)-m_2\right)\Bigr]\\[9pt]
&\equiv
Bin \hspace{1pt}(j) \Bigl[\tfrac{2j}{p}-\tfrac{j}{p}\left(2-\left(\tfrac{1}{d_1}+\tfrac{1}{d_2}\right)\right)\Bigr]\\[9pt]
&\equiv
Bin \hspace{1pt}(j) \Bigl[\tfrac{1}{p}\Bigr]
\Bigl[j \Bigl(\tfrac{1}{d_1}+\tfrac{1}{d_2}\Bigr)\Bigr]
\pmod{p^2}.\\
\end{align*}

\noindent Therefore it now suffices to show
\begin{multline*}
\sum_{j=0}^{p-\lfloor \frac{p-1}{d_2} \rfloor - 1}
Bin \hspace{1pt}(j) \Bigl[1+j \mathcal{H}(j) \Bigr]
+\sum_{j=0}^{p-\lfloor \frac{p-1}{d_2} \rfloor - 1}
Bin \hspace{1pt}(j)
\Bigl[ j \mathcal{B}(j) -\mathcal{A}(j) - j \mathcal{A}(j) \mathcal{H}(j) \Bigr]\\
-\sum_{j=\lfloor (d_2-1)\frac{p-1}{d_2} \rfloor+1}^{p-\lfloor \frac{p-1}{d_2} \rfloor - 1}
Bin \hspace{1pt}(j) \Bigl[\tfrac{1}{p}\Bigr]
\Bigl[j^2+\Bigl(\tfrac{1}{d_1}\Bigr)\Bigl(\tfrac{1}{d_2}\Bigr)+j \Bigl(\tfrac{1}{d_1}+\tfrac{1}{d_2}\Bigr)\Bigr]
\equiv s(p) 
\pmod{p^2}.
\end{multline*}

\noindent Recall that if $p\equiv 1\pmod{d_2}$ then the last sum above is vacuous. If $p\not\equiv 1 \pmod{d_2}$ then the limits of summation are equal and the sum is over one value of $j=\lfloor (d_2-1)\frac{p-1}{d_2} \rfloor+1=p-\lfloor \tfrac{p-1}{d_2} \rfloor - 1$. In this case $p\equiv d_2-1 \pmod{d_2}$ and $j=p- \frac{p+1}{d_2}=p\left(\tfrac{d_2-1}{d_2}\right)-\tfrac{1}{d_2}$. For this $j$, 
\begin{align*}
j^2+\Bigl(\tfrac{1}{d_1}\Bigr)\Bigl(\tfrac{1}{d_2}\Bigr)+j \Bigl(\tfrac{1}{d_1}+\tfrac{1}{d_2}\Bigr)
=
\Bigl(j+\tfrac{1}{d_1}\Bigr)\Bigl(j+\tfrac{1}{d_2}\Bigr)
=
\Bigl(p\left(\tfrac{d_2-1}{d_2}\right)-\tfrac{1}{d_2}+\tfrac{1}{d_1}\Bigr)\Bigl(p\left(\tfrac{d_2-1}{d_2}\right)\Bigr) \in p \mathbb{Z}_p\;.
\end{align*}
Therefore,
\begin{equation*}
\sum_{j=\lfloor (d_2-1)\frac{p-1}{d_2} \rfloor+1}^{p-\lfloor \frac{p-1}{d_2} \rfloor - 1}
Bin \hspace{1pt}(j) \Bigl[\tfrac{1}{p}\Bigr]
\Bigl[j^2+\Bigl(\tfrac{1}{d_1}\Bigr)\Bigl(\tfrac{1}{d_2}\Bigr)+j \Bigl(\tfrac{1}{d_1}+\tfrac{1}{d_2}\Bigr)\Bigr]
\equiv 0
\pmod{p^2},
\end{equation*}
as $Bin \hspace{1pt}(j) \in p^2\mathbb{Z}_p$ for $\lfloor \frac{p-1}{d_2} \rfloor+1=rep(m_3) \leq j \leq rep(m_2)-1=p-\lfloor \tfrac{p-1}{d_2} \rfloor - 1$.\\

\noindent We now use the results of Section 3 to resolve the two remaining sums. Taking $m=rep(m_1)-1$ and $n=rep(m_2)-1$ in Corollary \ref {Cor_BinHarId1} we get that
\begin{multline*}
(-1)^{rep(m_1)+rep(m_2)-2}  =\sum_{j=0}^{rep(m_2) - 1}
Bin \hspace{1pt}(j) \Bigl[1+j \mathcal{H}(j) \Bigr] +\\
\sum_{j=rep(m_2)}^{rep(m_1)-1} (-1)^{j-rep(m_2)+1} \biggbin{rep(m_1)-1+j}{j} \biggbin{rep(m_1)-1}{j} \biggbin{rep(m_2)-1+j}{j} \Big/ \biggbin{j-1}{rep(m_2)-1} .
\end{multline*}
We see from Lemma \ref{lem_ProdGammap2} that
$$ \biggbin{rep(m_1)-1+j}{j} \biggbin{rep(m_1)-1}{j}  \in p\mathbb{Z}_p$$
for $rep(m_4) \leq j \leq rep(m_1)-1$. Also
$$  \biggbin{rep(m_2)-1+j}{j} \in p\mathbb{Z}_p$$
for $p-rep(m_2)+1=rep(m_3) \leq j \leq p-1$, and
$$\biggbin{j-1}{rep(m_2)-1} \in \mathbb{Z}_p^{*}$$
for $rep(m_2)\leq j\leq p$.
Therefore
\begin{multline*}
\sum_{j=rep(m_2)}^{rep(m_1)-1} (-1)^{j-rep(m_2)+1} \biggbin{rep(m_1)-1+j}{j} \biggbin{rep(m_1)-1}{j} \biggbin{rep(m_2)-1+j}{j} \Big/ \biggbin{j-1}{rep(m_2)-1} \\
\equiv 0
\pmod{p^2}
\end{multline*}
and
\begin{multline*}
\sum_{j=0}^{p-\lfloor \frac{p-1}{d_2} \rfloor - 1}
Bin \hspace{1pt}(j) \Bigl[1+j \mathcal{H}(j) \Bigr]\equiv (-1)^{rep(m_1)+rep(m_2)-2}  
\equiv(-1)^{\lfloor \frac{p-1}{d_1} \rfloor +\lfloor \frac{p-1}{d_2} \rfloor}
\equiv s(p) 
\pmod{p^2}.
\end{multline*}

\noindent Similarly, taking $m=rep(m_1)-1$, $n=rep(m_2)-1$, $C_1=rep(m_1)-m_1$ and $C_2=rep(m_2)-m_2$ in Corollary \ref {Cor_BinHarId2} we get that
\begin{multline*}
\sum_{j=0}^{rep(m_2) - 1}Bin \hspace{1pt}(j)
\Bigl[ j \mathcal{B}(j) -\mathcal{A}(j) - j \mathcal{A}(j) \mathcal{H}(j) \Bigr]\\
=
\sum_{j=rep(m_2)}^{rep(m_1)-1} (-1)^{j-rep(m_2)+1} \Bigg[ \biggbin{rep(m_1)-1+j}{j} \biggbin{rep(m_1)-1}{j} \biggbin{rep(m_2)-1+j}{j} \Big/ \biggbin{j-1}{rep(m_2)-1}\Bigg]\\
\cdot \left(\left(rep(m_1)-m_1\right)(H_{j+rep(m_1)-1}^{(1)} - H_{j+p-rep(m_1)}^{(1)})  + \left(rep(m_2)-m_2\right) (H_{j+rep(m_2)}^{(1)} - H_{j+p-rep(m_2)}^{(1)})\right) .\\
\end{multline*}

\noindent We've seen that
$$ \biggbin{rep(m_1)-1+j}{j} \biggbin{rep(m_1)-1}{j}  \biggbin{rep(m_2)-1+j}{j}  \Big/ \biggbin{j-1}{rep(m_2)-1} \in p^2\mathbb{Z}_p$$
for $rep(m_2) \leq j \leq rep(m_1)-1$. We note also that $\left(rep(m_1)-m_1\right) \in p\mathbb{Z}_p$. Therefore
\begin{multline*}
\sum_{j=rep(m_2)}^{rep(m_1)-1} (-1)^{j-rep(m_2)+1} \Bigg[\biggbin{rep(m_1)-1+j}{j} \biggbin{rep(m_1)-1}{j} \biggbin{rep(m_2)-1+j}{j} \Big/ \biggbin{j-1}{rep(m_2)-1}\Bigg]\\[9pt]
\cdot \left(\left(rep(m_1)-m_1\right)(H_{j+rep(m_1)-1}^{(1)} - H_{j+p-rep(m_1)}^{(1)})  + \left(rep(m_2)-m_2\right) (H_{j+rep(m_2)}^{(1)} - H_{j+p-rep(m_2)}^{(1)})\right) \\[9pt]
\equiv 0 \pmod{p^2}
\end{multline*}
and
\begin{align*}
\sum_{j=0}^{p-\lfloor \frac{p-1}{d_2} \rfloor - 1}
Bin \hspace{1pt}(j)
\Bigl[ j \mathcal{B}(j) -\mathcal{A}(j) - j \mathcal{A}(j) \mathcal{H}(j) \Bigr]
\equiv 0 \pmod{p^2}
\end{align*}
as required.
\end{proof}

\begin{proof}[Proof of Theorem \ref{thm_4G2}] The proof proceeds along similar lines to that of Theorem \ref{thm_4G1}, however, $\phi(d)=4$ introduces some added complexity.

Let $p$ be a prime which satisfies the conditions of the theorem. Note that $p\geq 7$.
Assume without loss of generality that $r<d/2$. Let $p\equiv a \pmod{d}$. Then $a \in \{1,r,d-r,d-1\}$. We note that $\{1, r, d-r, d-1\}$ forms a group under multiplication modulo $d$ and $r^2\equiv\pm1 \pmod{d}$. Then $\{1, r, d-r, d-1\} \equiv \{a, ar, d-ar, d-a\}$ modulo $d$. If we let $s:=ar-d \lfloor \frac{ar}{d} \rfloor\equiv ar \pmod{d}$ then $\{1, r, d-r, d-1\} = \{a, s, d-s, d-a\}$.

From Corollary \ref{cor_repGenformula} we know that $rep(\frac{a}{d}) = p- \lfloor \frac{p-1}{d} \rfloor$ and $rep(\frac{d-a}{d}) = \lfloor \frac{p-1}{d} \rfloor+1$.
By Lemma \ref{lem_repGenformula}, $rep(\tfrac{m}{d})=\frac{pt+m}{d}$ where $t$ is the smallest integer such that $d\mid ta+m$ and that $t<d$. If $t=d-r$ then $d\mid ta+s$ and we get that $rep(\frac{s}{d})=p-r \lfloor \tfrac{p-1}{d} \rfloor - \lfloor \tfrac{ar}{d} \rfloor$. Then, by Proposition \ref{prop_repOneminus}, $rep(\frac{d-s}{d})= r \lfloor \tfrac{p-1}{d} \rfloor +1 +\lfloor \tfrac{ar}{d} \rfloor$.

Therefore $\Bigl\{rep\left(\tfrac{1}{d}\right),rep\left(\tfrac{r}{d}\right),rep\left(\tfrac{d-r}{d}\right),rep\left(\tfrac{d-1}{d}\right)\Bigr\}=\Bigl\{\lfloor \tfrac{p-1}{d} \rfloor +1,r \lfloor \tfrac{p-1}{d} \rfloor +1 +\lfloor \tfrac{ar}{d} \rfloor, p-r \lfloor \tfrac{p-1}{d} \rfloor - \lfloor \tfrac{ar}{d} \rfloor ,p-\lfloor \tfrac{p-1}{d} \rfloor\Bigr\}$, where the exact correspondence between the elements of each set depends on the choice of $p$. If we let $m_1:=\frac{a}{d}$, $m_2:=\frac{s}{d}$, $m_3:=\frac{d-s}{d}$ and $m_4:=\frac{d-a}{d}$ then $rep\bigl(m_4\bigr)< rep\bigl(m_3\bigr) < rep\bigl(m_2\bigr) < rep\bigl(m_1\bigr)$.\\

We reduce Definition \ref{def_GFn} modulo $p^3$ and use Proposition \ref{prop_pGamma} to expand the terms involved, noting that $\frac{1}{1-p}\equiv 1+p+p^2 \pmod{p^3}$, to get
\begin{multline*}
{_{4}G} \left(\tfrac{1}{d} , \tfrac{r}{d}, 1-\tfrac{r}{d} , 1-\tfrac{1}{d}\right)_p\\
\equiv
\sum_{j=0}^{\lfloor \frac{p-1}{d} \rfloor} 
\frac{\biggfp{\frac{1}{d}+j+jp+jp^2}\biggfp{\frac{r}{d}+j+jp+jp^2}
\biggfp{\frac{d-r}{d}+j+jp+jp^2}\biggfp{\frac{d-1}{d}+j+jp+jp^2}}
{\biggfp{\frac{1}{d}}\biggfp{\frac{r}{d}}\biggfp{\frac{d-r}{d}}\biggfp{\frac{d-1}{d}}
{\biggfp{1+j+jp+jp^2}}^{4}} 
\\  
\shoveleft +p \left[
\sum_{j=0}^{\lfloor \frac{p-1}{d} \rfloor} 
\frac{\biggfp{\frac{1}{d}+j+jp}\biggfp{\frac{r}{d}+j+jp}
\biggfp{\frac{d-r}{d}+j+jp}\biggfp{\frac{d-1}{d}+j+jp}}
{\biggfp{\frac{1}{d}}\biggfp{\frac{r}{d}}\biggfp{\frac{d-r}{d}}\biggfp{\frac{d-1}{d}}
{\biggfp{1+j+jp}}^{4}} 
\right. \\ \left. \qquad
- \sum_{j=\lfloor \frac{p-1}{d} \rfloor+1}^{\lfloor \frac{r(p-1)}{d} \rfloor}
\frac{\biggfp{\frac{d+1}{d}+j+jp}\biggfp{\frac{r}{d}+j+jp}
\biggfp{\frac{d-r}{d}+j+jp}\biggfp{\frac{d-1}{d}+j+jp}}
{\biggfp{\frac{1}{d}}\biggfp{\frac{r}{d}}\biggfp{\frac{d-r}{d}}\biggfp{\frac{d-1}{d}}
{\biggfp{1+j+jp}}^{4}}
\right] \\ 
\shoveleft +p^2 \left[
\sum_{j=0}^{\lfloor \frac{p-1}{d} \rfloor} 
\frac{\biggfp{\frac{1}{d}+j}\biggfp{\frac{r}{d}+j}
\biggfp{\frac{d-r}{d}+j}\biggfp{\frac{d-1}{d}+j}}
{\biggfp{\frac{1}{d}}\biggfp{\frac{r}{d}}\biggfp{\frac{d-r}{d}}\biggfp{\frac{d-1}{d}}
{\biggfp{1+j}}^{4}} 
\right. \\ \left.
 - \sum_{j=\lfloor \frac{p-1}{d} \rfloor+1}^{\lfloor \frac{r(p-1)}{d} \rfloor}
\frac{\biggfp{\frac{d+1}{d}+j}\biggfp{\frac{r}{d}+j}
\biggfp{\frac{d-r}{d}+j}\biggfp{\frac{d-1}{d}+j}}
{\biggfp{\frac{1}{d}}\biggfp{\frac{r}{d}}\biggfp{\frac{d-r}{d}}\biggfp{\frac{d-1}{d}}
{\biggfp{1+j}}^{4}}
\right. \\ \left.  \qquad
+\sum_{j=\lfloor \frac{r(p-1)}{d} \rfloor+1}^{\lfloor (d-r)\frac{p-1}{d} \rfloor} 
\frac{\biggfp{\frac{d+1}{d}+j}\biggfp{\frac{d+r}{d}+j}
\biggfp{\frac{d-r}{d}+j}\biggfp{\frac{d-1}{d}+j}}
{\biggfp{\frac{1}{d}}\biggfp{\frac{r}{d}}\biggfp{\frac{d-r}{d}}\biggfp{\frac{d-1}{d}}
{\biggfp{1+j}}^{4}}
\right] 
\pmod{p^3}.
\end{multline*}

\noindent We now consider $\gfp{\frac{d+1}{d}+j+jp}$ and $\gfp{\frac{d+1}{d}+j}$. We first note that
\begin{equation*}
\tfrac{1}{d}+j +jp \in p\mathbb{Z}_p \Longleftrightarrow \tfrac{1}{d}+j \in p\mathbb{Z}_p \Longleftrightarrow rep\left(\tfrac{1}{d}\right)+j \in p\mathbb{Z}_p \Longleftrightarrow rep\left(\tfrac{1}{d}\right)+j =p \Longleftrightarrow j=p-rep\left(\tfrac{1}{d}\right)\;.
\end{equation*}
If $p\equiv 1 \pmod {d}$ then using Corollary \ref{cor_repGenformula} we get that 
\begin{equation*}
p-rep\left(\tfrac{1}{d}\right)=\lfloor \tfrac{p-1}{d} \rfloor <  \lfloor \tfrac{p-1}{d} \rfloor +1\;.
\end{equation*}
Similarly, if $p\equiv d-1 \pmod {d}$ then
\begin{equation*}
p-rep\left(\tfrac{1}{d}\right)=p-\lfloor \tfrac{p-1}{d} \rfloor -1> p- \lceil \tfrac{p-1}{d} \rceil -1 >  p-  \lceil {\tfrac{r(p-1)}{d}} \rceil-1 = {\lfloor (d-r)\tfrac{p-1}{d} \rfloor} \;.
\end{equation*}
If $p\equiv r \pmod{d}$ and $r^2\equiv 1 \pmod{d}$, or $p\equiv d-r \pmod{d}$ and $r^2\equiv -1 \pmod{d}$ then $rep\left(\tfrac{1}{d}\right)= p - r \lfloor \tfrac{p-1}{d} \rfloor - \lfloor \tfrac{ar}{d} \rfloor$ and
\begin{equation*}
p-rep\left(\tfrac{1}{d}\right)= r \lfloor \tfrac{p-1}{d} \rfloor + \lfloor \tfrac{ar}{d} \rfloor = r \lfloor \tfrac{p-1}{d} \rfloor + \lfloor \tfrac{r(a-1)}{d} \rfloor +1= \lfloor  \tfrac{r(p-1)}{d} \rfloor +1
\end{equation*}
in either case. If $p\equiv r \pmod{d}$ and $r^2\equiv -1 \pmod{d}$, or $p\equiv d-r \pmod{d}$ and $r^2\equiv 1 \pmod{d}$ then $rep\left(\tfrac{1}{d}\right)= r \lfloor \tfrac{p-1}{d} \rfloor +1 + \lfloor \tfrac{ar}{d} \rfloor$ and
\begin{equation*}
p-rep\left(\tfrac{1}{d}\right)= p- r \lfloor \tfrac{p-1}{d} \rfloor -1- \lfloor \tfrac{ar}{d} \rfloor > p-r \lfloor \tfrac{p-1}{d} \rfloor -1 - \lceil {\tfrac{r(a-1)}{d}} \rceil = \lfloor (d-r)\tfrac{p-1}{d} \rfloor
\end{equation*}
in either case. So we see that the only time that $\tfrac{1}{d}+j \in p\mathbb{Z}_p$ for $\lfloor \tfrac{p-1}{d} \rfloor +1 \leq j \leq \lfloor (d-r)\tfrac{p-1}{d} \rfloor$ is when $p\equiv r \pmod{d}$ and $r^2\equiv 1 \pmod{d}$ or $p\equiv d-r \pmod{d}$ and $r^2\equiv -1 \pmod{d}$ and in these cases $j=rep(m_3)-1=r \lfloor \tfrac{p-1}{d} \rfloor + \lfloor \tfrac{ar}{d}\rfloor=\lfloor  \tfrac{r(p-1)}{d} \rfloor +1$. In all other cases $rep(m_3)=\lfloor  \tfrac{r(p-1)}{d} \rfloor +1$.

\noindent Therefore, using Proposition \ref{prop_pGamma} (1) we get that for $\lfloor \tfrac{p-1}{d} \rfloor +1 \leq j \leq \lfloor (d-r)\tfrac{p-1}{d} \rfloor$,
$$\gfp{\tfrac{d+1}{d}+j}=
\begin{cases}
-\gfp{\tfrac{1}{d}+j} & \text{if } j=\lfloor \tfrac{r(p-1)}{d} \rfloor +1,\; p\equiv r \pmod{d}, \;r^2\equiv 1 \pmod{d},\\
-\gfp{\tfrac{1}{d}+j} & \text{if }j=\lfloor \tfrac{r(p-1)}{d} \rfloor +1, \;p\equiv d-r \pmod{d}, \; r^2\equiv -1 \pmod{d},\\
-\left(\tfrac{1}{d}+j\right)\gfp{\tfrac{1}{d}+j} & \textup{otherwise,}
\end{cases}
$$
and that 
$$\gfp{\tfrac{d+1}{d}+j+jp}=-\left(\tfrac{1}{d}+j+jp\right)\gfp{\tfrac{1}{d}+j+jp}$$
for $\lfloor \tfrac{p-1}{d} \rfloor +1 \leq j \leq \lfloor \tfrac{r(p-1)}{d} \rfloor$.
We next consider  $\gfp{\frac{d+r}{d}+j}$. We first note that
\begin{equation*}
\tfrac{r}{d}+j \in p\mathbb{Z}_p \Longleftrightarrow rep\left(\tfrac{r}{d}\right)+j \in p\mathbb{Z}_p \Longleftrightarrow rep\left(\tfrac{r}{d}\right)+j =p \Longleftrightarrow j=p-rep\left(\tfrac{r}{d}\right)\;.
\end{equation*}
If $p\equiv r \pmod {d}$ then using Corollary \ref{cor_repGenformula} we get that 
\begin{equation*}
p-rep\left(\tfrac{r}{d}\right)=\lfloor \tfrac{p-1}{d} \rfloor <  \lfloor \tfrac{r(p-1)}{d} \rfloor +1\;.
\end{equation*}
Similarly, if $p\equiv d-r \pmod {d}$ then
\begin{equation*}
p-rep\left(\tfrac{r}{d}\right)=p-\lfloor \tfrac{p-1}{d} \rfloor -1> p- \lceil \tfrac{p-1}{d} \rceil -1 >  p-  \lceil \tfrac{r(p-1)}{d} \rceil-1 = {\lfloor (d-r)\tfrac{p-1}{d} \rfloor} \;.
\end{equation*}
If $p\equiv 1 \pmod{d}$, then $rep\left(\tfrac{r}{d}\right)= p - r \lfloor \tfrac{p-1}{d} \rfloor - \lfloor \tfrac{ar}{d} \rfloor$ and
\begin{equation*}
p-rep\left(\tfrac{r}{d}\right)= r \lfloor \tfrac{p-1}{d} \rfloor + \lfloor \tfrac{ar}{d} \rfloor = r \lfloor \tfrac{p-1}{d} \rfloor + \lfloor \tfrac{r}{d} \rfloor =r \lfloor \tfrac{p-1}{d} \rfloor = < \lfloor  \tfrac{r(p-1)}{d} \rfloor +1\;.
\end{equation*}
If $p\equiv d-1 \pmod{d}$, then $rep\left(\tfrac{r}{d}\right)= r \lfloor \tfrac{p-1}{d} \rfloor +1 + \lfloor \tfrac{ar}{d} \rfloor$ and
\begin{equation*}
p-rep\left(\tfrac{r}{d}\right)= p- r \lfloor \tfrac{p-1}{d} \rfloor -1- \lfloor \tfrac{ar}{d} \rfloor > p-r \lfloor \tfrac{p-1}{d} \rfloor -1 - \lceil \tfrac{r(a-1)}{d} \rceil = \lfloor (d-r)\tfrac{p-1}{d} \rfloor
\end{equation*}

\noindent Therefore, using Proposition \ref{prop_pGamma} (1), we get
$$\gfp{\tfrac{d+r}{d}+j}=-\left(\tfrac{r}{d}+j\right)\gfp{\tfrac{r}{d}+j}$$\\
for $\lfloor \frac{r(p-1)}{d} \rfloor+1 \leq j \leq  {\lfloor (d-r)\frac{p-1}{d} \rfloor}$. Applying these results and substituting $\{m_k\}$ for $\{\frac{1}{d}, \frac{r}{d}, \frac{d-r}{d}, \frac{d-1}{d}\}$ yields 
\begin{multline*}
{_{4}G} \left(\tfrac{1}{d} , \tfrac{r}{d}, 1-\tfrac{r}{d} , 1-\tfrac{1}{d}\right)_p
\equiv
\sum_{j=0}^{\lfloor \frac{p-1}{d} \rfloor} 
\frac{\prod_{k=1}^{4} \biggfp{m_k+j+jp+jp^2}}
{\prod_{k=1}^{4} \biggfp{m_k}
{\biggfp{1+j+jp+jp^2}}^{4}} 
\\  
\shoveleft +p \left[
\sum_{j=0}^{\lfloor \frac{p-1}{d} \rfloor} 
\frac{\prod_{k=1}^{4} \biggfp{m_k+j+jp}}
{\prod_{k=1}^{4} \biggfp{m_k}
{\biggfp{1+j+jp}}^{4}} 
\shoveright + \sum_{j=\lfloor \frac{p-1}{d} \rfloor+1}^{\lfloor \frac{r(p-1)}{d} \rfloor}
\frac{\prod_{k=1}^{4} \biggfp{m_k+j+jp}  \bigl(\frac{1}{d}+j+jp\bigr)}
{\prod_{k=1}^{4} \biggfp{m_k}
{\biggfp{1+j+jp}}^{4}} 
\right] \\ 
\shoveleft +p^2 \left[
\sum_{j=0}^{\lfloor \frac{p-1}{d} \rfloor} 
\frac{\prod_{k=1}^{4} \biggfp{m_k+j}}
{\prod_{k=1}^{4} \biggfp{m_k}
{\biggfp{1+j}}^{4}} 
+ \sum_{j=\lfloor \frac{p-1}{d} \rfloor+1}^{\lfloor \frac{r(p-1)}{d} \rfloor}
\frac{\prod_{k=1}^{4} \biggfp{m_k+j} \bigl(\frac{1}{d}+j\bigr)}
{\prod_{k=1}^{4} \biggfp{m_k}
{\biggfp{1+j}}^{4}} 
\right. \\ \left. 
\shoveleft +\sum_{j=\lfloor \frac{r(p-1)}{d} \rfloor+1}^{rep(m_3)-1} 
\frac{\prod_{k=1}^{4} \biggfp{m_k+j} \bigl(\frac{r}{d}+j\bigr)}
{\prod_{k=1}^{4} \biggfp{m_k}
{\biggfp{1+j}}^{4}} 
\right. \\ \left.  \qquad
+\sum_{j=rep(m_3)}^{\lfloor (d-r)\frac{p-1}{d} \rfloor} 
\frac{\prod_{k=1}^{4} \biggfp{m_k+j} \bigl(\frac{1}{d}+j\bigr)\bigl(\frac{r}{d}+j\bigr)}
{\prod_{k=1}^{4} \biggfp{m_k}
{\biggfp{1+j}}^{4}} 
\right] 
\pmod{p^3},
\end{multline*}
where the second last sum is vacuous unless $p\equiv r \pmod{d}$ and $r^2\equiv 1 \pmod{d}$ or $p\equiv d-r \pmod{d}$ and $r^2\equiv -1 \pmod{d}$. Arguing as we did in the proof of Theorem \ref{thm_4G1} gives us
\begin{align*}
\frac{\prod_{k=1}^{4} \biggfp{m_k+j+jp+jp^2}}
{\biggfp{1+j+jp+jp^2}^{4}}
\equiv
\frac{\prod_{k=1}^{4} \biggfp{m_k+j}}{\biggfp{1+j}^{4}}\Bigl[1+(jp+jp^2)A(j)+j^2p^2B(j)\Bigr]
\pmod{p^3}
\end{align*}
and
\begin{align*}
\frac{\prod_{k=1}^{4} \biggfp{m_k+j+jp}}
{\biggfp{1+j+jp}^{4}}
\equiv
\frac{\prod_{k=1}^{4} \biggfp{m_k+j}}{\biggfp{1+j}^{4}}\Bigl[1+jpA(j)\Bigr]
\pmod{p^2},
\end{align*}
where
$$A(j):=\sum_{k=1}^{4} \Bigl(G_1\left(m_k+j\right)-G_1\left(1+j\right)\Bigr)$$
and
$$B(j):=\frac{1}{2}\left[A(j)^2 - \sum_{k=1}^{4} \Bigl(G_1\left(m_k+j\right)^2-G_2\left(m_k+j\right)
-G_1\left(1+j\right)^2 +G_2\left(1+j\right)\Bigr)\right].$$
We note that both $A(j)$ and $B(j)$ $\in \mathbb{Z}_p$ by Proposition (\ref{prop_pGammaCong}). Applying the above and (\ref{for_GammapPlus}) we get
\begin{multline}\label{for_4G2}
{_{4}G} \left(\tfrac{1}{d} , \tfrac{r}{d}, 1-\tfrac{r}{d} , 1-\tfrac{1}{d}\right)_p
\equiv
\sum_{j=0}^{\lfloor \frac{p-1}{d} \rfloor} 
\frac{\prod_{k=1}^{4} \biggfp{m_k+j}}
{\prod_{k=1}^{4} \biggfp{m_k} j!^4}
\\ 
\shoveleft +p \left[
\sum_{j=0}^{\lfloor \frac{p-1}{d} \rfloor} 
\frac{\prod_{k=1}^{4} \biggfp{m_k+j}}
{\prod_{k=1}^{4} \biggfp{m_k} j!^4}
\Bigl[1+jA(j)\Bigr]
+ \sum_{j=\lfloor \frac{p-1}{d} \rfloor+1}^{\lfloor \frac{r(p-1)}{d} \rfloor} 
\frac{\prod_{k=1}^{4} \biggfp{m_k+j} }
{\prod_{k=1}^{4} \biggfp{m_k} j!^4}
\Bigl[\tfrac{1}{d}+j\Bigr]
\right] \\ 
\shoveleft +p^2 \left[
\sum_{j=0}^{\lfloor \frac{p-1}{d} \rfloor} 
\frac{\prod_{k=1}^{4} \biggfp{m_k+j}}
{\prod_{k=1}^{4} \biggfp{m_k} j!^4}
\Bigl[1+2jA(j)+j^2B(j)\Bigr]
\right. \\ \left. 
+\sum_{j=\lfloor \frac{p-1}{d} \rfloor+1}^{\lfloor \frac{r(p-1)}{d} \rfloor} 
\frac{\prod_{k=1}^{4} \biggfp{m_k+j}}
{\prod_{k=1}^{4} \biggfp{m_k} j!^4}
\Bigl[\Bigl(\tfrac{1}{d}+j\Bigr)\Bigl(1+jA(j)\Bigr)+j\Bigr]
\right. \\ \left. 
+\sum_{j=\lfloor \frac{r(p-1)}{d} \rfloor+1}^{rep(m_3)-1} 
\frac{\prod_{k=1}^{4} \biggfp{m_k+j}}
{\prod_{k=1}^{4} \biggfp{m_k} j!^4}
\Bigl[\tfrac{r}{d}+j\Bigr]
\right. \\ \left. 
+\sum_{j=rep(m_3)}^{\lfloor (d-r)\frac{p-1}{d} \rfloor} 
\frac{\prod_{k=1}^{4} \biggfp{m_k+j}}
{\prod_{k=1}^{4} \biggfp{m_k} j!^4}
\Bigl[\Bigl(\tfrac{1}{d}+j\Bigr)\Bigl(\tfrac{r}{d}+j\Bigr)\Bigr]
\right] 
\pmod{p^3}.
\end{multline}
\noindent Proposition \ref{prop_gammapmd} gives us
\begin{multline}\label{for_4F32}
{{_{4}F_3} \left[ \begin{array}{cccc} \frac{1}{d}, & \frac{r}{d}, & 1-\frac{r}{d}, & 1-\frac{1}{d}\vspace{.05in}\\
\phantom{\frac{1}{d}} & 1, & 1, & 1 \end{array}
\Big| \; 1 \right]}_{p-1}\\[6pt]
\shoveleft \equiv
\sum_{j=0}^{rep(m_4)-1}
\prod_{k=1}^{4} \frac{\biggfp{m_k+j}}{\biggfp{m_k} j!^4}
+\sum_{j=rep(m_4)}^{rep(m_3)-1} 
\prod_{k=1}^{4} \frac{\biggfp{m_k+j}}{\biggfp{m_k} j!^4}
\bigl(m_1+p-rep(m_1)\bigr)\\[6pt]
\shoveright {+\sum_{j=rep(m_3)}^{rep(m_2) - 1} 
\prod_{k=1}^{4} \frac{\biggfp{m_k+j}}{\biggfp{m_k} j!^4}
\bigl(m_1+p-rep(m_1)) \bigl(m_2+p-rep(m_2))}\\
\shoveleft \equiv
\sum_{j=0}^{rep(m_4)-1}
\prod_{k=1}^{4} \frac{\biggfp{m_k+j}}{\biggfp{m_k} j!^4}
+p\sum_{j=rep(m_4)}^{rep(m_3)-1} 
\prod_{k=1}^{4} \frac{\biggfp{m_k+j}}{\biggfp{m_k} j!^4}
\Bigl(\tfrac{1}{d}\Bigr)\\[6pt]
+p^2\sum_{j=rep(m_3)}^{rep(m_2) - 1} 
\prod_{k=1}^{4} \frac{\biggfp{m_k+j}}{\biggfp{m_k} j!^4}
\Bigl(\tfrac{1}{d}\Bigr) \Bigl(\tfrac{r}{d}\Bigr)
\pmod {p^3}.\\
\end{multline}

\noindent We now note that\\[4pt]
(1) $rep(m_4)-1=\lfloor \frac{p-1}{d} \rfloor$;\\[3pt]
(2) $rep(m_3)-1=\lfloor \frac{r(p-1)}{d} \rfloor$ except when $p\equiv r \pmod{d}$ and $r^2 \equiv 1 \pmod{d}$ or $p\equiv d-r \pmod{d}$ and $r^2 \equiv -1 \pmod{d}$. In such cases $rep(m_3)-1=\lfloor \frac{r(p-1)}{d} \rfloor +1$; and\\[3pt]
(3) $rep(m_2)-1=\lfloor (d-r)\frac{(p-1)}{d} \rfloor$ except when $p\equiv d-1 \pmod{d}$, $p\equiv r \pmod{d}$ and $r^2 \equiv -1 \pmod{d}$ or $p\equiv d-r \pmod{d}$ and $r^2 \equiv 1 \pmod{d}$. In such cases $rep(m_2)-1= \lfloor (d-r)\frac{(p-1)}{d} \rfloor +1$.\\[3pt]

\noindent Therefore, combining (\ref{for_4G2}) and (\ref{for_4F32}), it suffices to prove
\begin{multline}\label{Resid1_4G2}
\sum_{j=0}^{\lfloor \frac{p-1}{d} \rfloor} 
\frac{\prod_{k=1}^{4} \biggfp{m_k+j}}
{\prod_{k=1}^{4} \biggfp{m_k} j!^4}
\Bigl[1+jA(j)\Bigr]
+ \sum_{j=\lfloor \frac{p-1}{d} \rfloor+1}^{\lfloor \frac{r(p-1)}{d} \rfloor} 
\frac{\prod_{k=1}^{4} \biggfp{m_k+j} }
{\prod_{k=1}^{4} \biggfp{m_k} j!^4}
\Bigl[j\Bigr]
\\
\shoveleft {- \sum_{j=\lfloor \frac{r(p-1)}{d} \rfloor+1}^{rep(m_3)-1} 
\prod_{k=1}^{4} \frac{\biggfp{m_k+j}}{\biggfp{m_k} j!^4}
\Bigl[\tfrac{1}{d}\Bigr]}
 +p \left[
\sum_{j=0}^{\lfloor \frac{p-1}{d} \rfloor} 
\frac{\prod_{k=1}^{4} \biggfp{m_k+j}}
{\prod_{k=1}^{4} \biggfp{m_k} j!^4}
\Bigl[1+2jA(j)+j^2B(j)\Bigr]
\right. \\ \left.
\shoveleft{+\sum_{j=\lfloor \frac{p-1}{d} \rfloor+1}^{\lfloor \frac{r(p-1)}{d} \rfloor} 
\frac{\prod_{k=1}^{4} \biggfp{m_k+j}}
{\prod_{k=1}^{4} \biggfp{m_k} j!^4}
\Bigl[\Bigl(\tfrac{1}{d}+j\Bigr)\Bigl(1+jA(j)\Bigr)+j\Bigr]}
\right. \\ \left. 
\shoveleft{+\sum_{j=\lfloor \frac{r(p-1)}{d} \rfloor+1}^{rep(m_3)-1} 
\frac{\prod_{k=1}^{4} \biggfp{m_k+j}}
{\prod_{k=1}^{4} \biggfp{m_k} j!^4}
\Bigl[\tfrac{r}{d}+j\Bigr]
+\sum_{j=rep(m_3)}^{\lfloor (d-r)\frac{p-1}{d} \rfloor} 
\frac{\prod_{k=1}^{4} \biggfp{m_k+j}}
{\prod_{k=1}^{4} \biggfp{m_k} j!^4}
\Bigl[j^2+j\Bigl(\tfrac{1}{d}+\tfrac{r}{d}\Bigr)\Bigr]}
\right. \\ \left.  \qquad
- \sum_{j=\lfloor (d-r)\frac{p-1}{d} \rfloor+1}^{rep(m_2) - 1} 
\prod_{k=1}^{4} \frac{\biggfp{m_k+j}}{\biggfp{m_k} j!^4}
\Bigl[\Bigl(\tfrac{1}{d}\Bigr) \Bigl(\tfrac{r}{d}\Bigr)\Bigr]
\right] 
\equiv s(p)
\pmod{p^2}.
\end{multline}
We now examine $A(j)$, $B(j)$, $\prod_{k=1}^{4} \biggfp{m_k+j}$ and $\prod_{k=1}^{4} \biggfp{m_k}$ modulo $p$. 
Using Corollary \ref{cor_pGammaCong1} (3) and Proposition \ref{prop_pGammaG} (1) we see that
\begin{align*}
A(j)&=\sum_{k=1}^{4} \Bigl(G_1\left(m_k+j\right)-G_1\left(1+j\right)\Bigr)\\
&\equiv \sum_{k=1}^{4} \Bigl(H_{rep(m_k)-1+j}^{(1)} - H_{j}^{(1)}\Bigr)-
\begin{cases}
0 & \text{if } 0\leq j\leq rep(m_4)-1, \\[6pt]
\frac{1}{p} &\text{if }  rep(m_4) \leq j\leq rep(m_3)-1,\\[6pt]
\frac{2}{p} & \text{if }  rep(m_3)\leq j\leq rep(m_2)-1,\\[6pt]
\frac{3}{p} & \text{if }  rep(m_2)\leq j\leq rep(m_1)-1,\\[6pt]
\frac{4}{p} & \text{if }  rep(m_1)\leq j\leq p-1,\\[6pt]
\end{cases}
\pmod{p}.
\end{align*}

\noindent Similarly, using Corollary \ref{cor_pGammaCong1} (3), (4) and Proposition \ref{prop_pGammaG} (2), we see that
\begin{align*}
&\sum_{k=1}^{4} \Bigl(G_1\left(m_k+j\right)^2-G_2\left(m_k+j\right)
-G_1\left(1+j\right)^2 +G_2\left(1+j\right)\Bigr)\\
&\equiv \sum_{k=1}^{4} \Biggl(H_{rep(m_k)+j-1}^{(2)}- \hspace{1pt} H_{j}^{(2)}\Biggr)+
\begin{cases}
0 & \text{if } 0\leq j\leq rep(m_4)-1, \\[6pt]
\frac{1}{p^2} & \text{if }  rep(m_4) \leq j\leq rep(m_3)-1,\\[6pt]
\frac{2}{p^2} & \text{if }  rep(m_3)\leq j\leq rep(m_2)-1,\\[6pt]
\frac{3}{p^2} & \text{if }  rep(m_2)\leq j\leq rep(m_1)-1,\\[6pt]
\frac{4}{p^2} & \text{if }  rep(m_1)\leq j\leq p-1,\\[6pt]
\end{cases}
\pmod{p}.
\end{align*}

\noindent Using Proposition \ref{prop_pGamma} we get that
\begin{align*}
\prod_{k=1}^{4} \biggfp{m_k+j}
&\equiv \prod_{k=1}^{4} \biggfp{rep(m_k)+j}\\
&\equiv \Biggl[\prod_{k=1}^{4} \left(rep(m_k)+j-1\right)!\Biggr]
\begin{cases}
1 & \text{if } 0\leq j\leq rep(m_4)-1, \\[6pt]
\frac{1}{p} & \text{if }  rep(m_4) \leq j\leq rep(m_3)-1,\\[6pt]
\frac{2}{p^2} & \text{if }  rep(m_3)\leq j\leq rep(m_2)-1,\\[6pt]
\frac{3}{p^3} & \text{if }  rep(m_2)\leq j\leq rep(m_1)-1,\\[6pt]
\frac{4}{p^4} & \text{if }  rep(m_1)\leq j\leq p-1,\\[6pt]
\end{cases}
\pmod{p}
\end{align*}
and
\begin{align*}
\prod_{k=1}^{4} \biggfp{m_k} &
= \prod_{k=1}^{2} (-1)^{rep(m_k)}=(-1)^{rep(m_1)+rep(m_2)}=\pm1\;.
\end{align*}

\noindent We now consider
\begin{multline*}
X(j):=\sum_{j=0}^{rep(m_4)-1} 
\frac{\prod_{k=1}^{4} \biggfp{m_k+j}}
{\prod_{k=1}^{4} \biggfp{m_k} j!^4}
\Bigl[1+2jA(j)+j^2B(j)\Bigr]
\\
+\sum_{j=rep(m_4)}^{rep(m_3)-1} 
\frac{\prod_{k=1}^{4} \biggfp{m_k+j}}
{\prod_{k=1}^{4} \biggfp{m_k} j!^4}
\Bigl[2j+j^2A(j)\Bigr]
+\sum_{j=rep(m_3)}^{rep(m_2)- 1} 
\frac{\prod_{k=1}^{4} \biggfp{m_k+j}}
{\prod_{k=1}^{4} \biggfp{m_k} j!^4}
\Bigl[j^2\Bigr]
\pmod{p}.
\end{multline*}

\noindent Again similar to the proof of Theorem \ref{thm_4G1} we will show $X(j)\equiv0 \pmod{p}$. Substituting for $A(j)$, $B(j)$, $\prod_{k=1}^{4} \biggfp{m_k+j}$ and $\prod_{k=1}^{4} \biggfp{m_k}$ modulo $p$ we have
\begin{multline*}
\pm X(j)
\equiv
\sum_{j=0}^{p- 1}  
\Biggl[\prod_{k=1}^{4} \ph{j+1}{rep(m_k)-1} \Biggr]
\Biggl[1+2j\sum_{k=1}^{4} \Bigl(H_{rep(m_k)-1+j}^{(1)} -H_{j}^{(1)}\Bigr)
\\
+\frac{j^2}{2} \Biggl[\left(\sum_{k=1}^{4} \Bigl(H_{rep(m_k)-1+j}^{(1)} - H_{j}^{(1)}\Bigr)\right)^2
-\sum_{k=1}^{4} \Bigl(H_{rep(m_k)-1+j}^{(2)} - H_{j}^{(2)}\Bigr)
\Biggr]
\Biggr]
\pmod{p}.
\end{multline*}
\noindent We define
\begin{equation*}
P(j):= \frac{d}{dj} \left[ j \hspace{1pt} \prod_{k=1}^{4} \ph{j+1}{rep(m_k)-1} \right] = \sum_{k=0}^{2(p-1)} a_k j^k \;.
\end{equation*}
and
\begin{equation*}
Q(j):= \frac{j}{2} \frac{d^2}{dj^2} \left[ j \hspace{1pt} \prod_{k=1}^{4} \ph{j+1}{rep(m_k)-1} \right] = \sum_{k=0}^{2(p-1)} b_k j^k \;.
\end{equation*}
Then
\begin{equation*}
\pm X(j)\equiv \sum_{j=1}^{p-1} P(j)+Q(j) \pmod{p}.
\end{equation*}
\noindent Applying (\ref{exp_sums}) in the usual way yields
\begin{equation*}
\sum_{j=0}^{p-1} P(j) \equiv 1 \pmod{p}.
\end{equation*}
and
\begin{equation*}
\sum_{j=0}^{p-1} Q(j) \equiv -1 \pmod{p}.
\end{equation*}
So
\begin{equation}\label{for_X02}
X(j)=\pm \left(\sum_{j=1}^{p-1} P(j)+Q(j) \right)\equiv 0 \pmod{p}.
\end{equation}
\noindent Accounting for (\ref{for_X02}) in (\ref{Resid1_4G2}) means we need only show
\begin{multline}\label{Resid2_4G2}
\sum_{j=0}^{\lfloor \frac{p-1}{d} \rfloor} 
\frac{\prod_{k=1}^{4} \biggfp{m_k+j}}
{\prod_{k=1}^{4} \biggfp{m_k} j!^4}
\Bigl[1+jA(j)\Bigr]
+ \sum_{j=\lfloor \frac{p-1}{d} \rfloor+1}^{\lfloor \frac{r(p-1)}{d} \rfloor} 
\frac{\prod_{k=1}^{4} \biggfp{m_k+j} }
{\prod_{k=1}^{4} \biggfp{m_k} j!^4}
\Bigl[j\Bigr]\\
\shoveleft - \sum_{j=\lfloor \frac{r(p-1)}{d} \rfloor+1}^{rep(m_3)-1} 
\prod_{k=1}^{4} \frac{\biggfp{m_k+j}}{\biggfp{m_k} j!^4}
\Bigl[\tfrac{1}{d}\Bigr]
+p \left[
\sum_{j=\lfloor \frac{p-1}{d} \rfloor+1}^{\lfloor \frac{r(p-1)}{d} \rfloor} 
\frac{\prod_{k=1}^{4} \biggfp{m_k+j}}
{\prod_{k=1}^{4} \biggfp{m_k} j!^4}
\Bigl[\Bigl(\tfrac{1}{d}\Bigr)\Bigl(1+jA(j)\Bigr)\Bigr]
\right. \\ \left.  \qquad
+\sum_{j=\lfloor \frac{r(p-1)}{d} \rfloor+1}^{rep(m_3)-1} 
\frac{\prod_{k=1}^{4} \biggfp{m_k+j}}
{\prod_{k=1}^{4} \biggfp{m_k} j!^4}
\Bigl[\tfrac{r}{d}-j-j^2A(j)\Bigr]
+\sum_{j=rep(m_3)}^{\lfloor (d-r)\frac{p-1}{d} \rfloor} 
\frac{\prod_{k=1}^{4} \biggfp{m_k+j}}
{\prod_{k=1}^{4} \biggfp{m_k} j!^4}
\Bigl[j\Bigl(\tfrac{1}{d}+\tfrac{r}{d}\Bigr)\Bigr]
\right. \\ \left.  \qquad
- \sum_{j=\lfloor (d-r)\frac{p-1}{d} \rfloor+1}^{rep(m_2) - 1} 
\prod_{k=1}^{4} \frac{\biggfp{m_k+j}}{\biggfp{m_k} j!^4}
\Bigl[\Bigl(\tfrac{1}{d}\Bigr) \Bigl(\tfrac{r}{d}\Bigr)+j^2\Bigr]
\right] 
\equiv s(p)
\pmod{p^2}.
\end{multline}
Using the same notation as in the proof of Theorem \ref{thm_4G1} and invoking Lemmas \ref{lem_ProdGammap2} and \ref{lem_SumGammap2} in a similar manner, (\ref{Resid2_4G2}) is equivalent to 
\begin{multline*}
\sum_{j=0}^{rep(m_2) - 1}
Bin \hspace{1pt}(j) \Bigl[1+j \mathcal{H}(j) \Bigr]
+\sum_{j=0}^{rep(m_2) - 1}
Bin \hspace{1pt}(j)
\Bigl[ j \mathcal{B}(j) -\mathcal{A}(j) - j \mathcal{A}(j) \mathcal{H}(j) \Bigr]\\
-\sum_{j=\lfloor \frac{r(p-1)}{d} \rfloor+1}^{rep(m_3)-1} 
Bin \hspace{1pt}(j)
\Biggl[ \Bigl[\tfrac{1}{p}\Bigr] \left[1-\mathcal{A}(j) + \tfrac{rep(m_1)-m_1}{p}\right]
\Bigl[\tfrac{1}{d}+j\Bigr]
+\Bigl[-\tfrac{r}{d}+\left(\tfrac{1}{d}+j\right)\left(1+j\mathcal{H}(j)-\tfrac{j}{p}\right)\Bigr]
\Biggr]
\\
-\sum_{j=\lfloor (d-r)\frac{p-1}{d} \rfloor+1}^{rep(m_2) - 1}
Bin \hspace{1pt}(j) \Bigl[\tfrac{1}{p}\Bigr]
\Bigl[j^2+\Bigl(\tfrac{1}{d}\Bigr)\Bigl(\tfrac{r}{d}\Bigr)+j \Bigl(\tfrac{1}{d}+\tfrac{r}{d}\Bigr)\Bigr]
\equiv s(p) 
\pmod{p^2}.
\end{multline*}

\noindent Recall that the second sum above is vacuous unless $p\equiv r \pmod{d}$ and $r^2\equiv 1\pmod{d}$ or $p\equiv d-r \pmod{d}$ and $r^2\equiv -1\pmod{d}$. In these cases the sum is over one value of $j=rep(m_3)-1 = \lfloor \tfrac{r(p-1)}{d} \rfloor +1$. In either case $\left(\tfrac{1}{d} +j\right) = \tfrac{rp}{d}$. Also $rep(m_1)-m_1=p\left(1-\tfrac{1}{d}\right)$. So we get 
\begin{multline*}
\sum_{j=\lfloor \frac{r(p-1)}{d} \rfloor+1}^{rep(m_3)-1} 
Bin \hspace{1pt}(j)
\Biggl[ \Bigl[\tfrac{1}{p}\Bigr] \left[1-\mathcal{A}(j) + \tfrac{rep(m_1)-m_1}{p}\right]
\Bigl[\tfrac{1}{d}+j\Bigr]
+\Bigl[-\tfrac{r}{d}+\left(\tfrac{1}{d}+j\right)\left(1+j\mathcal{H}(j)-\tfrac{j}{p}\right)\Bigr]
\Biggr]\\
=
\sum_{j=\lfloor \frac{r(p-1)}{d} \rfloor+1}^{rep(m_3)-1} 
Bin \hspace{1pt}(j)
\Bigl[\tfrac{rpj}{d}\mathcal{H}(j)-\tfrac{r}{d}\mathcal{A}(j)+\tfrac{r}{d}\tfrac{d-1}{d}+\tfrac{rp}{d}-\tfrac{rj}{d}
\Bigr].
\end{multline*}

\noindent Note $Bin \hspace{1pt}(j) \in p \mathbb{Z}_p$ for $j=rep(m_3)-1$. Thus $Bin \hspace{1pt}(j) \hspace{1pt}p\hspace{1pt} \mathcal{H}(j)\equiv Bin \hspace{1pt}(j) \hspace{1pt}p \left(\frac{1}{p}\right) \pmod{p^2}$, $Bin \hspace{1pt}(j) \mathcal{A}(j) \equiv Bin \hspace{1pt}(j) \frac{rep(m_1)-m_1}{p} \equiv Bin \hspace{1pt}(j) \frac{d-1}{d}\pmod{p^2}$ and 
\begin{multline*}
\sum_{j=\lfloor \frac{r(p-1)}{d} \rfloor+1}^{rep(m_3)-1} 
Bin \hspace{1pt}(j)
\Bigl[\tfrac{rpj}{d}\mathcal{H}(j)-\tfrac{r}{d}\mathcal{A}(j)+\tfrac{r}{d}\tfrac{d-1}{d}+\tfrac{rp}{d}-\tfrac{rj}{d}
\Bigr]
\equiv
\sum_{j=\lfloor \frac{r(p-1)}{d} \rfloor+1}^{rep(m_3)-1} 
Bin \hspace{1pt}(j)
\Bigl[\tfrac{rp}{d}
\Bigr]
\equiv 0
\pmod{p^2}.
\end{multline*}

\noindent Next we examine
\begin{multline*}
\sum_{j=\lfloor (d-r)\frac{p-1}{d} \rfloor+1}^{rep(m_2) - 1}
Bin \hspace{1pt}(j) \Bigl[\tfrac{1}{p}\Bigr]
\Bigl[j^2+\Bigl(\tfrac{1}{d}\Bigr)\Bigl(\tfrac{r}{d}\Bigr)+j \Bigl(\tfrac{1}{d}+\tfrac{r}{d}\Bigr)\Bigr]\\
=\sum_{j=\lfloor (d-r)\frac{p-1}{d} \rfloor+1}^{rep(m_2) - 1}
Bin \hspace{1pt}(j) \Bigl[\tfrac{1}{p}\Bigr]
\Bigl[\Bigl(j+\tfrac{1}{d}\Bigr)\Bigl(j+\tfrac{r}{d}\Bigr)\Bigr]
\end{multline*}
modulo $p^2$. Recall that this sum is vacuous unless $p\equiv d-1\pmod{d}$, $p \equiv r \pmod{d}$ and $r^2\equiv-1\pmod{d}$ or $p \equiv d-r \pmod{d}$ and $r^2\equiv1\pmod{d}$. Then the limits of summation are equal and the sum is over one value of $j=rep(m_2)-1=p-r\lfloor \tfrac{p-1}{d} \rfloor - \lfloor \frac{ar}{d} \rfloor -1$.
\noindent If $p\equiv d-1 \pmod{d}$,
$$j=p-r\left(\tfrac{p-d+1}{d}\right)-\lfloor \tfrac{dr-r}{d} \rfloor -1=p\left(1-\tfrac{r}{d}\right)+r-\tfrac{r}{d}-\lfloor r-\tfrac{r}{d} \rfloor -1=p\left(1-\tfrac{r}{d}\right)+\tfrac{d-r}{d}-1=p\left(1-\tfrac{r}{d}\right)-\tfrac{r}{d}.$$
Then
$$\left(j+\tfrac{r}{d}\right) = p\left(1-\tfrac{r}{d}\right) \in p \mathbb{Z}_p$$
and
$$\left(j+\tfrac{1}{d}\right) = p\left(1-\tfrac{r}{d}\right) -\tfrac{r}{d} +\tfrac{1}{d}\in  \mathbb{Z}_p.$$

\noindent If $p\equiv r \pmod{d}$ and $r^2 \equiv -1 \pmod{d}$,
$$j=p-r\left(\tfrac{p-r}{d}\right)-\lfloor \tfrac{r^2}{d} \rfloor -1=p\left(1-\tfrac{r}{d}\right)+\tfrac{r^2}{d}-\lfloor \tfrac{r^2}{d} \rfloor -1=p\left(1-\tfrac{r}{d}\right)+\tfrac{d-1}{d}-1=p\left(1-\tfrac{r}{d}\right)-\tfrac{1}{d}.$$
Then
$$\left(j+\tfrac{1}{d}\right) = p\left(1-\tfrac{r}{d}\right) \in p \mathbb{Z}_p$$
and
$$\left(j+\tfrac{r}{d}\right) = p\left(1-\tfrac{r}{d}\right) -\tfrac{1}{d} +\tfrac{r}{d}\in  \mathbb{Z}_p.$$

\noindent If $p\equiv d-r \pmod{d}$ and $r^2 \equiv 1 \pmod{d}$,
$$j=p-r\left(\tfrac{p-d+r}{d}\right)-\lfloor \tfrac{dr-r^2}{d} \rfloor -1=p\left(1-\tfrac{r}{d}\right)+\tfrac{dr-r^2}{d}-\lfloor \tfrac{dr-r^2}{d} \rfloor -1=p\left(1-\tfrac{r}{d}\right)+\tfrac{d-1}{d}-1=p\left(1-\tfrac{r}{d}\right)-\tfrac{1}{d}.$$
Then
$$\left(j+\tfrac{1}{d}\right) = p\left(1-\tfrac{r}{d}\right) \in p \mathbb{Z}_p$$
and
$$\left(j+\tfrac{r}{d}\right) = p\left(1-\tfrac{r}{d}\right) -\tfrac{1}{d} +\tfrac{r}{d}\in  \mathbb{Z}_p.$$

\noindent Therefore
\begin{equation*}
\sum_{j=\lfloor (d-r)\frac{p-1}{d} \rfloor+1}^{rep(m_2) - 1}
Bin \hspace{1pt}(j) \Bigl[\tfrac{1}{p}\Bigr]
\Bigl[\Bigl(j+\tfrac{1}{d}\Bigr)\Bigl(j+\tfrac{r}{d}\Bigr)\Bigr]
\equiv 0
\pmod{p^2}
\end{equation*}
as $Bin \hspace{1pt}(j) \in p^2\mathbb{Z}_p$ for $j=rep(m_2)-1$.
\noindent Finally, taking $m=rep(m_1)-1$, $n=rep(m_2)-1$, $C_1=rep(m_1)-m_1$ and $C_2=rep(m_2)-m_2$ (where applicable) in Corollaries \ref {Cor_BinHarId1} and \ref {Cor_BinHarId2} we show that
\begin{equation*}
\sum_{j=0}^{rep(m_2)-1}
Bin \hspace{1pt}(j) \Bigl[1+j \mathcal{H}(j) \Bigr]
\equiv s(p) 
\pmod{p^2}
\end{equation*}
and
\begin{equation*}
 \sum_{j=0}^{rep(m_2)-1}
Bin \hspace{1pt}(j)
\Bigl[ j \mathcal{B}(j) -\mathcal{A}(j) - j \mathcal{A}(j) \mathcal{H}(j) \Bigr]
\equiv 0 \pmod{p^2}
\end{equation*}
as required.
\end{proof}


\section{Proof of Theorem \ref{thm_DMCMain}}
One easily checks the result for primes $p=2,3$. Now let $p\geq7$ be a prime. Then by Theorem \ref{thm_4G2} with $d=5$ we have
\begin{align*}
{_{4}G} \left(\tfrac{1}{5} ,\hspace{2pt} \tfrac{2}{5},\hspace{2pt} \tfrac{3}{5} ,\hspace{2pt} \tfrac{4}{5}\right)_p
-s(p)\hspace{1pt} p
&\equiv
{{_{4}F_3} \left[ \begin{array}{cccc} \frac{1}{5}, & \frac{2}{5}, & \frac{3}{5}, & \frac{4}{5}\vspace{.05in}\\
\phantom{\frac{1}{d_1}} & 1, & 1, & 1 \end{array}
\Big| \; 1 \right]}_{p-1}
\pmod {p^3}
\end{align*}
where $s(p) = \gfp{\tfrac{1}{5}} \gfp{\tfrac{2}{5}}\gfp{\tfrac{3}{5}}\gfp{\tfrac{4}{5}}$. Therefore Theorem \ref{thm_DMCMain} will be established on proof of the following two results.
\begin{theorem}\label{thm_GStoMod}
For a prime $p \neq 5$, 
\begin{equation*}
-\frac{1}{p-1} \left[1+\frac{1}{p} \sum_{j=1}^{p-2}\; G_{-j}^{5} G_{5j} \; T^{-5j}(-5) \right]- s(p) \cdot p
= c(p),
\end{equation*}
where $s(p):= \gfp{\frac{1}{5}} \gfp{\frac{2}{5}} \gfp{\frac{3}{5}} \gfp{\frac{4}{5}}$, $T$ is a generator for the group of characters on $\mathbb{F}_p$ and $c(p)$ is as defined in (\ref{for_ModForm}).
\end{theorem}
\begin{cor}\label{cor_GStoMod}
For a prime $p \neq 5$,
\begin{align*}
{_{4}G} \left(\tfrac{1}{5} ,\hspace{2pt} \tfrac{2}{5},\hspace{2pt} \tfrac{3}{5} ,\hspace{2pt} \tfrac{4}{5}\right)_p
-s(p)\hspace{1pt} p
&=
c(p),
\end{align*}
where $s(p) := \gfp{\tfrac{1}{5}} \gfp{\tfrac{2}{5}}\gfp{\tfrac{3}{5}}\gfp{\tfrac{4}{5}}$ and $c(p)$ is as defined in (\ref{for_ModForm}).
\end{cor}

\noindent Then, using Proposition \ref{prop_GtoGHS}, the corollary below easily follows.
\begin{cor}\label{Cor_GHStoMod}
For a prime $p \equiv 1 \pmod 5$,
\begin{equation*}
-p^3 \; _4F_3 \left( \begin{array}{cccc} \chi_5, & \chi_5^2, & \chi_5^3, & \chi_5^4 \\
\phantom{\chi_5} & \varepsilon, & \varepsilon, & \varepsilon \end{array}
\big| \; 1 \right)_p -s(p)\hspace{1pt} p
= c(p),
\end{equation*}
where $\chi_5$ is a character of order 5 on $\mathbb{F}_p$, $s(p) := \gfp{\tfrac{1}{5}} \gfp{\tfrac{2}{5}}\gfp{\tfrac{3}{5}}\gfp{\tfrac{4}{5}}$ and $c(p)$ is as defined in (\ref{for_ModForm}).
\end{cor}

\begin{proof}[Proof of Theorem \ref{thm_GStoMod}]
From \cite[page 32]{Me} (following the work of Schoen \cite{S}), we have
\begin{equation}\label{for_ModtoPts}
c(p)=
\begin{cases}
p^3 + 25 p^2 -100 p + 1 - N_p & \text{if $p \equiv 1\phantom{,2} \pmod {5}$},\\
p^3 + \phantom{25} p^2 \: \;\phantom{-100 p} + 1 - N_p & \text{if $p \equiv 4\phantom{,2} \pmod {5}$}, \\
p^3 + \phantom{25} p^2 + \phantom{10}2 p + 1 - N_p & \text{if $p \equiv 2,3 \pmod {5}$} ,
\end{cases}
\end{equation}
where
$N_p$ is the number of points in $\mathbb{P}^4(\mathbb{F}_p)$ on $x_0^5+x_1^5+x_2^5+x_3^5+x_4^5-5x_0 x_1 x_2 x_3 x_4=0 \;.$
We will evaluate $N_p$ using (\ref{for_CtgPts}) and express our results in terms of Gauss sums using (\ref{for_AddtoGauss}). Using (\ref{for_CtgPtsSplit}) repeatedly we can write 
\begin{equation*}
N_p = N_p^{1}+N_p^2++N_p^3+N_p^4
\end{equation*}
 where
\begin{align*}
N_p^{1}&= \textup{number of points in }\mathbb{A}^1
(\mathbb{F}_p)\textup{ on } f_1(x_1) := 1+x_1^5=0,\\
N_p^{2}&= \textup{number of points in }\mathbb{A}^2
(\mathbb{F}_p)\textup{ on } f_2(x_1, x_2) := 1+x_1^5+x_2^5=0,\\
N_p^{3}&= \textup{number of points in }\mathbb{A}^3
(\mathbb{F}_p)\textup{ on } f_3(x_1, x_2, x_3) := 1+x_1^5+x_2^5+x_3^5=0,\textup{ and}\\
N_p^4&= \textup{number of points in }\mathbb{A}^4(\mathbb{F}_p)\textup{ on } f_4(x_1, x_2, x_3, x_4) := 1+x_1^5+x_2^5+x_3^5+x_4^5-5 x_1 x_2 x_3 x_4=0.
\end{align*}

\noindent Using (\ref{for_CtgPts}) we get that 
\begin{equation*}
p N_p^4= p^4 + \displaystyle \sum_{z \in \mathbb{F}_p^*} \sum_{x_1, x_2, x_3, x_4 \in \mathbb{F}_p}
\theta(z \: f_4(x_1, x_2, x_3, x_4)).
\end{equation*}

\noindent
Now
\begin{align*}
\sum_{x_1, x_2, x_3, x_4 \in \mathbb{F}_p} &\theta(z \: f_4(x_1, x_2, x_3, x_4)) = 
\theta(z) +
4 \sum_{x_1 \in \mathbb{F}_p^*} \theta(z \: (1+x_1^5)) + 
6 \sum_{x_1, x_2 \in \mathbb{F}_p^*} \theta(z \: (1+x_1^5 + x_2^5)) \\
&+4 \sum_{x_1, x_2, x_3 \in \mathbb{F}_p^*} \theta(z \: (1+x_1^5 + x_2^5 + x_3^5)) +
\sum_{x_1, x_2, x_3, x_4 \in \mathbb{F}_p^*} \theta(z \: f_4(x_1, x_2, x_3, x_4))
\end{align*}
\noindent and so
\begin{equation*}
p N_p^4= p^4 + \sum_{z \in \mathbb{F}_p^*} \theta(z) +4A +6B+4C+D,
\end{equation*}
where
\begin{align*}
&A= \sum_{z \in \mathbb{F}_p^*} \sum_{x_1 \in \mathbb{F}_p^*} \theta(z \: (1+x_1^5)),\\
&B= \sum_{z \in \mathbb{F}_p^*} \sum_{x_1, x_2 \in \mathbb{F}_p^*} \theta(z \: (1+x_1^5 + x_2^5)), \\
&C= \sum_{z \in \mathbb{F}_p^*}  \sum_{x_1, x_2, x_3 \in \mathbb{F}_p^*} \theta(z \: (1+x_1^5 + x_2^5 + x_3^5)),\\
\intertext{and}
&D= \sum_{z \in \mathbb{F}_p^*} \sum_{x_1, x_2, x_3, x_4 \in \mathbb{F}_p^*} \theta(z \: f_4(x_1, x_2, x_3, x_4)).
\end{align*}

\noindent By (\ref{sum_AddChar}) we have
\begin{align*}
\sum_{z \in \mathbb{F}_p^*} \theta(z) 
=-1\;.
\end{align*}

\noindent We evaluate $N_p^1$, $N_p^2$ and $N_p^3$ similarly and overall we get
\begin{equation}\label{for_Np}
pN_p=p^4+p^3+p^2+p-4+10A+10B+5C+D\;.
\end{equation}

\noindent Using properties (\ref{for_AddProp}) and (\ref{for_AddtoGauss}) of the additive character $\theta$ and the orthogonal relation (\ref{for_TOrthEl}) for $T$ we evaluate the terms $A, B, C$ and $D$.
\begin{align*}
A=\sum_{z, x_1 \in \mathbb{F}_p^*} \theta(z (1+x_1^5))
&= \sum_{z, x_1 \in \mathbb{F}_p^*} \theta(z) \theta(z x_1^5) \\
&=\frac{1}{(p-1)^2} \sum_{z, x_1 \in \mathbb{F}_p^*} \sum_{a,b = 0}^{p-2} G_{-a} G_{-b} T^{a}(z) T^b(z x_1^5) \\
&=\frac{1}{(p-1)^2} \sum_{x_1 \in \mathbb{F}_p^*} \sum_{a,b = 0}^{p-2} G_{-a} G_{-b} T^b(x_1^5) 
\sum_{z \in \mathbb{F}_p^*} T^{a+b}(z).
\end{align*}
We now apply (\ref{for_TOrthEl}) on the last summation. This yields $(p-1)$ if $T^{a+b} =\varepsilon$ which occurs if $a=b=0$ or $a=(p-1)-b$. Both of these cases are covered when $a=-b$ as $T^{(p-1)-b}=T^{-b}$. So
\begin{align*}
A&=\frac{1}{(p-1)} \sum_{x_1 \in \mathbb{F}_p^*} \sum_{b = 0}^{p-2} G_{b} G_{-b} T^b(x_1^5)\\
&=\frac{1}{(p-1)} \sum_{b = 0}^{p-2} G_{b} G_{-b} \sum_{x_1 \in \mathbb{F}_p^*} T^{5b}(x_1)\\
&=
\begin{cases}
\displaystyle \sum_{i = 0}^{4} G_{\frac{i(p-1)}{5}} G_{-\frac{i(p-1)}{5}} &  \textup{if } p \equiv 1 \pmod 5\\[18pt]
{G_0}^2 =1 & \textup{if } p \not\equiv 1 \pmod 5.
\end{cases}
\end{align*}
The last application of (\ref{for_TOrthEl}) yields $p-1$ if and only if $T^{5b} = \varepsilon$. If $p\equiv 1 \pmod 5$ this occurs if $b=0$ or a multiple of $\frac{p-1}{5}$ whereas if $p \not \equiv 1 \pmod5$ then only if $b=0$.
A similar application of (\ref{for_TOrthEl}), (\ref{for_AddProp}) and (\ref{for_AddtoGauss}) yields
\begin{equation*}
B=
\begin{cases}
\displaystyle \sum_{i,j = 0}^{4} G_{(i+j)t} G_{-it} G_{-jt}   &  \textup{if } p \equiv 1 \pmod 5,\\[18pt]
\displaystyle  {G_0}^3 = -1& \textup{if } p \not\equiv 1 \pmod 5,
\end{cases}
\end{equation*}
\begin{equation*}
C=
\begin{cases}
\displaystyle \sum_{i,j,k=0}^{4} G_{(i+j+k)t} G_{-it} G_{-jt} G_{-kt} 
& \textup{if } p \equiv 1 \pmod 5,\\[18pt]
\displaystyle {G_{0}}^4 =1
& \textup{if } p \not\equiv 1 \pmod 5,
\end{cases}
\end{equation*}
and
\begin{equation*}
D=
\begin{cases}
\displaystyle \frac{1}{(p-1)} \sum_{e= 0}^{p-2} \sum_{i,j,k=0}^4
G_{-e+(i+j+k)t} G_{-e-it} G_{-e-jt} G_{-e-kt} G_{-e} G_{5e} T^{-5e}(-5)
& \textup{if } p \equiv 1 \pmod 5,\\[18pt]
\displaystyle \frac{1}{(p-1)} \sum_{e= 0}^{p-2}  {G_{-e}^5}  G_{5e} T^{-5e}(-5)
& \textup{if } p \not\equiv 1 \pmod 5,
\end{cases}
\end{equation*}
where $t=\frac{p-1}{5}$. Noting that
\begin{equation}\label{for_Sp}
s(p)=
\begin{cases}
+1 & \text{if } p\equiv 1, 4 \pmod 5,\\
-1 & \text{if } p\equiv 2,3 \pmod 5,
\end{cases}
\end{equation}
and combining (\ref{for_ModtoPts}), (\ref{for_Np}) and the evaluations above we see that when $p \not\equiv 1 \pmod 5$
\begin{align*}
c(p)+s(p) \cdot p 
&= -\frac{1}{p}\Biggl[-4+10A+10B+5C+D\Biggr]\\
&=-\dfrac{1}{p}\left[1+\frac{1}{(p-1)} \sum_{e= 0}^{p-2} {G_{-e}^5} G_{5e} T^{-5e}(-5)\right]\\
&=-\dfrac{1}{p-1}\left[1+ \frac{1}{p} \sum_{e= 1}^{p-2} {G_{-e}^5} G_{5e} T^{-5e}(-5)\right]
\end{align*}
as required. We now consider the case when $p \equiv 1 \pmod 5$. Again combining (\ref{for_ModtoPts}), (\ref{for_Np}) and (\ref{for_Sp}) we get that
\begin{align*}
c(p)+s(p) \cdot p
&=24 p^2 -100 p -\tfrac{1}{p}\bigl[-4+10A+10B+5C+D\bigr].
\end{align*}

\noindent We first examine $D$. Recall
\begin{equation*}
D=\frac{1}{(p-1)} \sum_{e= 0}^{p-2} \sum_{i,j,k=0}^4
G_{-e+(i+j+k)t} G_{-e-it} G_{-e-jt} G_{-e-kt} G_{-e} G_{5e} T^{-5e}(-5)\:.
\end{equation*}

\noindent We split this sum into different cases depending on the values of $i,j$ and $k$ as follows (note that, as $G_{nt}=G_{mt}$ if $n\equiv m \pmod5$, when we refer to $-i,-j,-k$ and $i+j+k$ in the list below we are really considering them as basic representatives modulo $5$, i.e., elements of $\{0,1,2,3,4\}$):
\\[-9pt]
\begin{enumerate}
\item[(0)] $i=j=k=0$;\\[-6pt]
\item $i,j,k \geq1$ all distinct;\\[-6pt]
\item $i,j,k \geq 1$ and exactly three of $-i,-j,-k$ and $i+j+k$ are identical;\\[-6pt]
\item Exactly two of $-i,-j,-k$ and $i+j+k$ equal $0$;\\[-6pt]
\item $i,j,k \geq 1$, $i+j+k \neq 0$ less cases (1) and (2); and\\[-6pt]
\item Exactly one of $-i,-j,-k$ and $i+j+k$ equal $0$.\\[-6pt]
\end{enumerate}
These cases are mutually exclusive and cover all 125 possible triples $(i,j,k)$.
\noindent We let $D_n$ denote the part of $D$ covered in case $n$. Then
\begin{align*}
D_0=\frac{1}{(p-1)} \sum_{e= 0}^{p-2} G_{-e}^5 G_{5e} T^{-5e}(-5)
\end{align*}
and
\begin{align*}
-\frac{1}{p} \Biggl[1+D_0\Biggr]&=-\frac{1}{p}\left[1+\frac{1}{(p-1)} \sum_{e= 0}^{p-2} G_{-e}^5 G_{5e} T^{-5e}(-5)\right]\\
&=-\dfrac{1}{p-1}\left[1+ \frac{1}{p} \sum_{e= 1}^{p-2} {G_{-e}^5} G_{5e} T^{-5e}(-5)\right].
\end{align*}

\noindent So it suffices to show
\begin{equation*}
24 p^2 -100 p-\tfrac{1}{p}\Bigl[-5+10A+10B+5C+D- D_0\Bigr]=0.
\end{equation*}

\noindent By Theorem \ref{thm_HD}, with $m=5$ and $\psi=T^e$, we see that
\begin{equation*}
G_e G_{e+t} G_{e+2t} G_{e+3t} G_{e+4t} = G_{5e} T^{-5e}(5) G_{t} G_{2t} G_{3t} G_{4t}\:.
\end{equation*}
\noindent Therefore, using (\ref{for_GaussConj}) and noting that $t$ is even, we get
\begin{equation*}
D
=\frac{1}{p^2(p-1)} \sum_{e= 0}^{p-2} \sum_{i,j,k=0}^4
G_{-e+(i+j+k)t} G_{-e-it} G_{-e-jt} G_{-e-kt} G_{-e} G_e G_{e+t} G_{e+2t} G_{e+3t} G_{e+4t} T^{-5e}(-1).\\ 
\end{equation*}

\noindent We now evaluate $D_1$ to $D_5$.\\[3pt]
\noindent (1) There are 24 different triples $(i,j,k)$ in this case. If $i,j$ an $k$ are all distinct then $i+j+k \not\equiv 0 \pmod5$. Also $i+j+k$ is distinct from $-i,-j$ and $-k$ when all are considered as basic representatives modulo $5$. Therefore
\begin{equation*}
G_{-e+(i+j+k)t} G_{-e-it} G_{-e-jt} G_{-e-kt} = G_{-e+t} G_{-e+2t} G_{-e+3t} G_{-e+4t}
\end{equation*}
and
\begin{equation*}
D_1=
 \frac{24}{p^2(p-1)} \sum_{e= 0}^{p-2} 
G_{-e+t} G_{-e+2t} G_{-e+3t} G_{-e+4t} G_{-e} G_e G_{e+t} G_{e+2t} G_{e+3t} G_{e+4t} T^{-5e}(-1) \:.
\end{equation*}
By (\ref{for_GaussConj}), $G_{-e+nt} G_{e+(5-n)t}= T^{-e+nt}(-1) \:p$ if $T^{-e+nt} \neq \varepsilon$.
So, for $a\in\{0,1,2,3,4\}$,
\begin{align*}
D_1
&=\frac{24}{p^2(p-1)} \left[\sum_{\substack{e= 0\\e\neq at}}^{p-2} p^5 \; T^{-5e+10t}(-1) T^{-5e}(-1) 
+\sum_{\substack{e= 0\\e=at}}^{p-2} p^4 \;T^{10t}(-1)\right]\\
&=24 p^2 \left[p-5  \right]\\
\end{align*}
as $T^{10t}=T^{2(p-1)}=\varepsilon$.\\

\noindent (2) There are 16 different triples $(i,j,k)$ in this case. However only 4 of these give different values for the set $\{-i,-j,-k,i+j+k\}$ when the elements are expressed as basic representatives modulo $5$ and these sets can be represented by $\{k, k, k, -3k\}$ for $1\leq k\leq 4$. Therefore
\begin{equation*}
D_2=
 \frac{4}{p^2(p-1)} \sum_{e= 0}^{p-2} \sum_{k=1}^4 
G_{-e+kt}^3 G_{-e-3kt} G_{-e} G_e G_{e+t} G_{e+2t} G_{e+3t} G_{e+4t} T^{-5e}(-1) \:.
\end{equation*}

\noindent Note we can replace $G_{e+t} G_{e+2t} G_{e+3t} G_{e+4t}$ with $G_{e+kt} G_{e+2kt} G_{e+3kt} G_{e+4kt}$, as $T^{nt}=T^{mt}$ when $n\equiv m \pmod 5$ and $\mathbb{F}_5^*$ is a multiplicative group. We now split the summation depending on whether $e=at$ or not, for $a\in\{0,1,2,3,4\}$, and make this replacement in the first part to get
\begin{multline*}
D_2=
 \frac{4}{p^2(p-1)}\left[\sum_{\substack{e= 0\\e\neq at}}^{p-2} \sum_{k=1}^4 
G_{-e+kt}^3 G_{-e-3kt} G_{-e} G_e G_{e+kt} G_{e+2kt} G_{e+3kt} G_{e+4kt} T^{-5e}(-1)\right.\\
\left.+\sum_{\substack{e= 0\\e= at}}^{p-2} \sum_{k=1}^4
G_{-e+kt}^3 G_{-e-3kt} G_{-e} G_e G_{e+t} G_{e+2t} G_{e+3t} G_{e+4t} \right] \:.
\end{multline*}

\noindent We use (\ref{for_GaussConj}) to simplify this expression as follows.
\begin{align*}
D_2
&= \frac{4}{p^2(p-1)}\left[\sum_{\substack{e= 0\\e\neq at}}^{p-2} \sum_{k=1}^4 
p^3 G_{-e+kt}^2   G_{e+kt} G_{e+2kt} T^{-3e-2kt}(-1)  T^{-5e}(-1)\right. \\
& \phantom{= \frac{4}{p^2(p-1)}}\left. \qquad \qquad \qquad\qquad \quad\quad \quad \quad  -\sum_{\substack{e= 0\\e= at}}^{p-2} \sum_{k=1}^4
p^2 G_{-e+kt}^3 G_{-e-3kt} G_{-e}  \right] \\
&= \frac{4}{(p-1)}\left[p  \sum_{\substack{e= 0\\e\neq at}}^{p-2} \sum_{k=1}^4 
G_{-e+kt}^2   G_{e+kt} G_{e+2kt}  -\sum_{\substack{e= 0\\e= at}}^{p-2} \sum_{k=1}^4
 G_{-e+kt}^3 G_{-e-3kt} G_{-e} \right] .\\
\end{align*}

\noindent We can simplify the first summation and express it in terms of generalised Jacobi sums using Corollaries \ref{cor_GaussTsum} and \ref{cor_JacRedSmall} and Proposition \ref{prop_JactoGauss} to get
\begin{align*}
\sum_{\substack{e= 0\\e\neq at}}^{p-2} \sum_{k=1}^4 G_{-e+kt}^2   G_{e+kt} G_{e+2kt}
&= \sum_{k=1}^4 \sum_{e= 0}^{p-2}G_{-e+kt}^2 G_{e+kt} G_{e+2kt}
- \sum_{k=1}^4 \sum_{\substack{e= 0\\e=at}}^{p-2} G_{-e+kt}^2   G_{e+kt} G_{e+2kt}\\
&= \sum_{k=1}^4 -p(p-1)
-p \sum_{k=1}^4 \sum_{\substack{e= 0\\e=at}}^{p-2} J\left(T^{-e+kt},T^{-e+kt},T^{e+kt}\right)\\
&=-4p(p-1)-p\sum_{k=1}^4 \sum_{\substack{e= 0\\e=akt}}^{p-2} J\left(T^{-e+kt},T^{-e+kt},T^{e+kt}\right)\\
&=-4p(p-1)-p\sum_{k=1}^4 \sum_{a=0}^4 J\left(T^{(1-a)kt},T^{(1-a)kt},T^{(1+a)kt}\right)\\
&=-4p(p-1)-p\sum_{k=1}^4 \left[J\left(T^{kt},T^{kt},T^{kt}\right)+
 J\left(\varepsilon,\varepsilon,T^{2kt}\right) \right.\\ 
&\left. + J\left(T^{4kt},T^{4kt},T^{3kt}\right) +J\left(T^{3kt},T^{3kt},T^{4kt}\right)
+J\left(T^{2kt},T^{2kt},\varepsilon \right)\right]\:.\\
\end{align*}

\noindent Similarly,
\begin{align*}
\sum_{\substack{e= 0\\e= at}}^{p-2} \sum_{k=1}^4 G_{-e+kt}^3 G_{-e-3kt} G_{-e}
&=\sum_{k=1}^4 \sum_{a=0}^4 G_{(1-a)kt}^3 G_{(-a-3)kt} G_{-akt}\\
&=\sum_{k=1}^4 \left[-G_{kt}^3 G_{2kt} - G_{kt} G_{4kt} - G_{4kt}^3 G_{3kt} \right.\\
&\qquad\qquad \qquad \quad\qquad \qquad \qquad \; \;  \left. + G_{3kt}^3 G_{4kt} G_{2kt}+G_{2kt}^3 G_{3kt} G_{kt}\right]\\
&=\sum_{k=1}^4 \left[-p\:J\left(T^{kt},T^{kt},T^{kt}\right) - p  - p\: J\left(T^{4kt},T^{4kt},T^{3kt}\right) \right.\\
&\qquad \qquad \qquad \left.- p^2 J\left(T^{3kt},T^{3kt},T^{4kt}\right) -p^2 J\left(T^{2kt},T^{2kt},T^{kt}\right)\right].\\
\end{align*}

\noindent By Propositions \ref {prop_JacBasic} and \ref{prop_JacRed} and Corollary \ref{cor_JacRedSmall},
\begin{equation*}
J\left(T^{4kt},T^{4kt},T^{3kt}\right)=-J\left(T^{4kt},T^{4kt},T^{3kt},T^{4kt} \right)=J\left(T^{4kt},T^{4kt},T^{4kt} \right),
\end{equation*}
\begin{equation*}
J\left(T^{2kt},T^{2kt},\varepsilon \right)=-J\left(T^{2kt},T^{2kt},T^{kt},\varepsilon \right)
=J\left(T^{2kt},T^{2kt},T^{kt} \right)
\end{equation*}
and
\begin{equation*}
J\left(\varepsilon,\varepsilon,T^{2kt}\right)=J\left(\varepsilon,T^{2kt}\right) J\left(\varepsilon,T^{2kt}\right)
=(-1)^2=1.
\end{equation*}
Applying Lemma \ref{lem_JacConsMul} with $a=b=c=4$ and $r=4$ we see that
\begin{equation*}
\sum_{k=1}^4 J\left(T^{4kt},T^{4kt},T^{4kt}\right)=\sum_{k=1}^4 J\left(T^{16kt},T^{16kt},T^{16kt}\right)
=\sum_{k=1}^4 J\left(T^{kt},T^{kt},T^{kt}\right)\;.
\end{equation*}

\noindent Combining these yields
\begin{align*}
D_2 &= \frac{4}{(p-1)}
\left[ -4p^2(p-1)+(p-p^2)\sum_{k=1}^4 \left(1+2J\left(T^{kt},T^{kt},T^{kt}\right)\right) \right]\\
&= -8p
\left[ 2p+2+\sum_{k=1}^4 J\left(T^{kt},T^{kt},T^{kt}\right) \right]\:.\\
\end{align*}

\noindent (3) There are 24 different triples $(i,j,k)$ in this case. However only 4 of these give different values for the set $\{-i,-j,-k,i+j+k\}$ when the elements are expressed as basic representatives modulo $5$ and these sets can be represented by $\{0,0 , k, -k\}$ for $1\leq k\leq 4$. Therefore
\begin{align*}
D_3&=
 \frac{6}{p^2(p-1)} \sum_{e= 0}^{p-2} \sum_{k=1}^4
G_{-e+kt} G_{-e-kt} G_{-e}^3 G_e G_{e+t} G_{e+2t} G_{e+3t} G_{e+4t} T^{-5e}(-1).
\end{align*}
Letting $e'=e+kt$ we see that 
\begin{multline*}
D_3=
\frac{6}{p^2(p-1)} \sum_{k=1}^4 \sum_{e'= kt}^{p-2+kt} 
G_{-e'+2kt} G_{-e'} G_{-e'+kt}^3 G_{e'-kt} \\G_{e'+(1-k)t} G_{e'+(2-k)t} G_{e'+(3-k)t} G_{e'+(4-k)t} T^{-5e'+5kt}(-1) .
\end{multline*}

\noindent Now $G_{e'-kt} G_{e'+(1-k)t} G_{e'+(2-k)t} G_{e'+(3-k)t} G_{e'+(4-k)t} = G_{e'} G_{e'+t} G_{e'+2t} G_{e'+3t} G_{e'+4t}$ as $T^{nt}=T^{mt}$ when $n\equiv m \pmod 5$ and $\mathbb{F}_5$ is an additive group. Similarly $T^{a}=T^{b}$ when $a\equiv b \pmod {p-1}$ so we can change the limits of summation on $e'$ to get
\begin{align*}
D_3&=
\frac{6}{p^2(p-1)} \sum_{k=1}^4 \sum_{e'= 0}^{p-2} 
G_{-e'+2kt} G_{-e'} G_{-e'+kt}^3 G_{e'} G_{e'+t} G_{e'+2t} G_{e'+3t} G_{e'+4t}T^{-5e'}(-1)\\[6pt]
&=\frac{6}{4} \cdot D_2\\
&= -12p
\left[ 2p+2+\sum_{k=1}^4 J\left(T^{kt},T^{kt},T^{kt}\right) \right].
\end{align*}

\noindent (4) There are 12 different triples $(i,j,k)$ in this case. However only 2 of these give different values for the set $\{-i,-j,-k,i+j+k\}$ when the elements are expressed as basic representatives modulo $5$ and these sets can be represented by $\{
k, k, -k, -k\}$ for $1\leq k\leq 2$. Therefore, using (\ref{for_GaussConj}),
\begin{align*}
D_4
&= \frac{3}{p^2(p-1)} \sum_{e= 0}^{p-2} \sum_{k=1}^4
G_{-e+kt}^2 G_{-e-kt}^2 G_{-e} G_e G_{e+t} G_{e+2t} G_{e+3t} G_{e+4t} T^{-5e}(-1)\\
&= \frac{3}{(p-1)} \left[p \sum_{\substack{e= 0\\e\neq at}}^{p-2} \sum_{k=1}^4
G_{-e+kt} G_{-e-kt}  G_{e+2kt} G_{e+3kt}
- \sum_{\substack{e= 0\\e=at}}^{p-2} \sum_{k=1}^4
G_{-e+kt}^2 G_{-e-kt}^2 G_{-e} \right] \; .\\
\end{align*}

\noindent Similar to case (2) we now simplify these summations and express them in terms of generalised Jacobi sums using Corollaries \ref{cor_GaussTsum} and \ref{cor_JacRedSmall} and Proposition \ref{prop_JactoGauss} to get
\begin{align*}
\sum_{\substack{e= 0\\e\neq at}}^{p-2} \sum_{k=1}^4 &G_{-e+kt} G_{-e-kt}  G_{e+2kt} G_{e+3kt}\\
&= \sum_{k=1}^4 \sum_{e= 0}^{p-2}G_{-e+kt} G_{-e+4kt}  G_{e+2kt} G_{e+3kt}
- \sum_{k=1}^4 \sum_{\substack{e= 0\\e=at}}^{p-2}G_{-e+kt} G_{-e+4kt}  G_{e+2kt} G_{e+3kt}\\
&= \sum_{k=1}^4 -p(p-1)
+p \sum_{k=1}^4 \sum_{\substack{e= 0\\e=at}}^{p-2} J\left(T^{-e+kt},T^{-e+4kt},T^{e+2kt},T^{e+3kt}\right)\\
&= -4p(p-1)
- p \sum_{k=1}^4 \sum_{a= 0}^{4}  J\left(T^{(1-a)kt}, T^{(4-a)kt}, T^{(2+a)kt}\right)\\
&=-4p(p-1)-p\sum_{k=1}^4 \left[J\left(T^{kt},T^{4kt},T^{2kt}\right)+
 J\left(\varepsilon,T^{3kt},T^{3kt}\right) \right.\\ 
&\left. \phantom{=}+ J\left(T^{4kt},T^{2kt},T^{4kt}\right) +J\left(T^{3kt},T^{kt},\varepsilon\right)
+J\left(T^{2kt},\varepsilon,T^{kt} \right)\right]\:.\\
\end{align*}
Similarly,
\begin{align*}
\sum_{\substack{e= 0\\e=at}}^{p-2} \sum_{k=1}^4 G_{-e+kt}^2 G_{-e-kt}^2 G_{-e}
&=\sum_{k=1}^4 \sum_{a= 0}^{4} G_{(1-a)kt}^2 G_{(-1-a)kt}^2 G_{-akt}\\
&=\sum_{k=1}^4 \left[ -p^2 -p\: J\left(T^{3kt},T^{3kt},T^{4kt}\right) 
-p^2 \: J\left(T^{4kt},T^{4kt},T^{2kt}\right)\right.\\
&\left.\qquad \qquad \qquad \qquad \quad
-p^2 J\left(T^{3kt},T^{kt},T^{kt}\right) -p \: J\left(T^{2kt},T^{2kt},T^{kt}\right)\right]\:.\\
\end{align*}

\noindent By Propositions \ref {prop_JacBasic} and \ref{prop_JacRed} and Corollary \ref{cor_JacRedSmall},
\begin{equation*}
J\left(\varepsilon,T^{3kt},T^{3kt} \right)=-J\left(\varepsilon,T^{3kt},T^{3kt}, T^{4kt} \right)
=J\left(T^{3kt},T^{3kt}, T^{4kt} \right),
\end{equation*}
\begin{equation*}
J\left(T^{3kt},T^{kt}, \varepsilon \right)=-J\left(T^{3kt},T^{kt}, \varepsilon, T^{kt} \right)
=J\left(T^{3kt},T^{kt}, T^{kt} \right),
\end{equation*}
\begin{equation*}
J\left(T^{2kt}, \varepsilon,T^{kt} \right)=-J\left(T^{2kt},T^{2kt}, \varepsilon, T^{kt} \right)
=J\left(T^{2kt},T^{2kt}, T^{kt} \right)
\end{equation*}
and
\begin{equation*}
J\left(T^{kt},T^{4kt},T^{2kt} \right)=-p\:J\left(T^{kt},T^{4kt} \right)=p.
\end{equation*}
Combining these yields
\begin{align*}
D_4 &= \frac{3}{(p-1)}
\left[ -4p^2(p-1)+(p-p^2) \sum_{k=1}^4 
\left(p + J\left(T^{3kt},T^{3kt},T^{4kt}\right)+J\left(T^{2kt},T^{2kt},T^{kt}\right)\right)\right]\\
&= -3p
\left[ 8p+\sum_{k=1}^4 
\left(J\left(T^{3kt},T^{3kt},T^{4kt}\right)+J\left(T^{2kt},T^{2kt},T^{kt}\right)\right)\right].\\
\end{align*}

\noindent Applying Lemma \ref{lem_JacConsMul} with $a=b=3,c=4$ and $r=4$, we see that
\begin{equation*}
\sum_{k=1}^4 J\left(T^{3kt},T^{3kt},T^{4kt}\right)=\sum_{k=1}^4 J\left(T^{12kt},T^{12kt},T^{16kt}\right)
=\sum_{k=1}^4 J\left(T^{2kt},T^{2kt},T^{kt}\right)
\end{equation*}
and
\begin{equation*}
D_4 = -6p \left[ 4p+\sum_{k=1}^4 J\left(T^{2kt},T^{2kt},T^{kt}\right)\right].
\end{equation*}\\

\noindent (5) There are 48 different triples $(i,j,k)$ in this case. However only 4 of these give different values for the set $\{-i,-j,-k,i+j+k\}$ when the elements are expressed as basic representatives modulo $5$ and these sets can be represented by $\{0, k , 2k, 2k\}$ for $1\leq k\leq 4$. Therefore
\begin{align*}
D_5&=
 \frac{12}{p^2(p-1)} \sum_{e= 0}^{p-2} \sum_{k=1}^4
G_{-e+2kt}^2 G_{-e+kt} G_{-e}^2 G_e G_{e+t} G_{e+2t} G_{e+3t} G_{e+4t} T^{-5e}(-1).
\end{align*}
Letting $e'=e-kt$, we see that 
\begin{multline*}
D_5=
\frac{12}{p^2(p-1)} \sum_{k=1}^4 \sum_{e'= -kt}^{p-2-kt} 
G_{-e'+kt}^2 G_{-e'} G_{-e'-kt}^2 G_{e'+kt} \\G_{e'+(1+k)t} G_{e'+(2+k)t} G_{e'+(3+k)t} G_{e'+(4+k)t} T^{-5e'-5kt}(-1) .
\end{multline*}

\noindent Now $G_{e'+kt} G_{e'+(1+k)t} G_{e'+(2+k)t} G_{e'+(3+k)t} G_{e'+(4+k)t} = G_{e'} G_{e'+t} G_{e'+2t} G_{e'+3t} G_{e'+4t}$ as $T^{nt}=T^{mt}$ when $n\equiv m \pmod 5$ and $\mathbb{F}_5$ is an additive group. Similarly $T^{a}=T^{b}$ when $a\equiv b \pmod {p-1}$ so we can change the limits of summation on $e'$ to get
\begin{align*}
D_5&=
\frac{12}{p^2(p-1)} \sum_{k=1}^4 \sum_{e'= 0}^{p-2} 
G_{-e'+kt}^2 G_{-e'} G_{-e'-kt}^2 G_{e'} G_{e'+t} G_{e'+2t} G_{e'+3t} G_{e'+4t}T^{-5e'}(-1)\\[6pt]
&=4 \cdot D_4\\[6pt]
&=-24p \left[ 4p+\sum_{k=1}^4 J\left(T^{2kt},T^{2kt},T^{kt}\right)\right]\;.
\end{align*}

\noindent Putting all these cases together we get
\begin{align*}
D-D_0 &=
24p^3 -280p^2-40p
-20p \sum_{k=1}^4 J\left(T^{kt},T^{kt},T^{kt}\right)-30p \sum_{k=1}^4 J\left(T^{2kt},T^{2kt},T^{kt}\right)\; .
\end{align*}

\noindent We now examine $A$, $B$ and $C$ in a similar manner. First,
\begin{align*}
A=\sum_{i = 0}^{4} G_{it} G_{-it}&=
{G_0}^2 + \sum_{i = 1}^{4} G_{it} G_{-it}\\
&=1 +4p & \text{(using (\ref{for_GaussConj}))}.
\end{align*}
\noindent Note from the evaluation of $A$ that
\begin{equation*}
A^*:=\sum_{i= 1}^{4} G_{it} G_{-it}=4p \:,
\end{equation*}
which we will use below. Now,
\begin{align*}
B&=\sum_{i,j = 0}^{4} G_{(i+j)t} G_{-it} G_{-jt}\\
&= G_0^3 + 2\sum_{k= 1}^{4} G_0 G_{kt} G_{-kt} + \sum_{i,j = 1}^{4} G_{(i+j)t} G_{-it} G_{-jt}\\
&= -1 + 3\sum_{k= 1}^{4} G_0 G_{kt} G_{-kt} + \sum_{\substack{i,j = 1\\i+j \neq 5}}^{4} G_{(i+j)t} G_{-it} G_{-jt}\\
&=-1-3 A^* -p \sum_{\substack{i,j = 1\\i+j \neq 5}}^{4}
J\left(T^{(i+j)t}, T^{-it}, T^{-jt}\right)\\
&=-1-12 p -3p \sum_{k = 1}^{4} J\left(T^{kt}, T^{2kt}, T^{2kt}\right)
\end{align*}
as the triples $(-i,-j,i+j)$, for $1\leq i,j \leq 4$ with $i+j\neq5$, can be represented by $(k,2k,2k)$,  $1\leq k \leq 4$, with multiplicity $3$ when considered as basic representatives modulo $5$. Note from the evaluation of $B$ that
\begin{equation*}
B^*:=\sum_{i,j = 1}^{4} G_{(i+j)t} G_{-it} G_{-jt}=-4p -3p \sum_{k = 1}^{4} J\left(T^{kt}, T^{2kt}, T^{2kt}\right)
\end{equation*}
and so
\begin{align*}
C&= \sum_{i,j,k=0}^{4} G_{(i+j+k)t} G_{-it} G_{-jt} G_{-kt} \\
&=G_0^4 + 3\sum_{k= 1}^{4} G_0^2 G_{kt} G_{-kt} +3\sum_{i,j = 1}^{4} G_0 G_{(i+j)t} G_{-it} G_{-jt}
+\sum_{i,j,k=1}^{4} G_{(i+j+k)t} G_{-it} G_{-jt} G_{-kt}\\
&=1+3A^*-3B^*
+\sum_{\substack{i,j,k=1\\i+j+k \equiv 0\; (5)}}^{4} G_{(i+j+k)t} G_{-it} G_{-jt} G_{-kt}
+\sum_{\substack{i,j,k=1\\i+j+k \not\equiv 0\; (5)}}^{4} G_{(i+j+k)t} G_{-it} G_{-jt} G_{-kt}\\
&=1+3A^*-3B^*
-\sum_{\substack{i,j,=1\\i \not\equiv -j\; (5)}}^{4} G_{-it} G_{-jt} G_{(i+j)t}
+\sum_{\substack{i,j,k=1\\i+j+k \not\equiv 0\; (5)}}^{4} G_{(i+j+k)t} G_{-it} G_{-jt} G_{-kt}\\
&=1+3A^*-3B^*
-\bigl(B^*-\sum_{\substack{i,j,=1\\i \equiv -j\; (5)}}^{4} G_{-it} G_{-jt} G_{(i+j)t}\bigr)
+\sum_{\substack{i,j,k=1\\i+j+k \not\equiv 0\; (5)}}^{4} G_{(i+j+k)t} G_{-it} G_{-jt} G_{-kt}\\
&=1+3A^*-4B^*
+\sum_{i=1}^{4} G_{-it} G_{it} G_{0}
+\sum_{\substack{i,j,k=1\\i+j+k \not\equiv 0\; (5)}}^{4} G_{(i+j+k)t} G_{-it} G_{-jt} G_{-kt}\\
&=1+3A^*-4B^*-A^*
+\sum_{\substack{i,j,k=1\\i+j+k \not\equiv 0\; (5)}}^{4} G_{(i+j+k)t} G_{-it} G_{-jt} G_{-kt}\\
&=1+2A^*-4B^*
+\sum_{\substack{i,j,k=1\\i+j+k \not\equiv 0\; (5)}}^{4} G_{(i+j+k)t} G_{-it} G_{-jt} G_{-kt}.
\end{align*}
There are 52 triples $(i,j,k)$ satisfying the conditions of the last summation above, which can be split into three distinct cases as follows:\\[3pt]
(1) $i,j,k$ all distinct. This covers 24 triples $(i,j,k)$. In this case $-i, -j -k$ and $i+j+k$ are all distinct when considered as basic representatives modulo $5$ and so the set $\{-i,-j-k,i+j+k\}$ can be represented by $\{1,2,3,4\}$ for each triple.\\[3pt]
(2) Exactly three of $-i,-j,-k$ and $i+j+k$ are identical when considered as basic representatives modulo $5$. There are 16 different triples $(i,j,k)$ in this case. However only 4 of these give different values for the set $\{-i,-j,-k,i+j+k\}$ when the elements are expressed as basic representatives modulo $5$ and these sets can be represented by $\{k, k, k, 2k\}$ for $1\leq k\leq 4$.\\[3pt]
(3) This cases covers the remaining 12 triples. However, only 2 of these give different values for the set $\{-i,-j,-k,i+j+k\}$ when the elements are expressed as basic representatives modulo $5$ and these sets can be represented by $\{k, k, -k, -k\}$ for $1\leq k\leq 2$.\\

\noindent Therefore, using (\ref{for_GaussConj}), Proposition \ref{prop_JactoGauss} and Corollary \ref{cor_JacRedSmall},
\begin{align*}
\sum_{\substack{i,j,k=1\\i+j+k \not\equiv 0\; (5)}}^{4} G_{(i+j+k)t} G_{-it} G_{-jt} G_{-kt}
&=24\hspace{2pt}G_{t}G_{2t}G_{3t}G_{4t} + 4\sum_{k=1}^4 G_{kt}^3 G_{2kt} + 3\sum_{k=1}^4 G_{kt}^2 G_{-kt}^2\\
&=24\hspace{2pt}p^2 -4p\sum_{k=1}^4 J\left(T^{kt}, T^{kt}, T^{kt}, T^{2kt} \right)+3\sum_{k=1}^4 p^2\\
&=36p^2 +4p\sum_{k=1}^4 J\left(T^{kt}, T^{kt}, T^{kt} \right)
\end{align*}
and
\begin{align*}
C=1+ 36p^2+24p +12p \sum_{k = 1}^{4} J\left(T^{kt}, T^{2kt}, T^{2kt}\right)  +4p\sum_{k=1}^4 J\left(T^{kt}, T^{kt}, T^{kt} \right).
\end{align*}
Overall we get  
\begin{align*}
\frac{1}{p}\Bigl[-5+&10A+10B+5C+D- D_0\Bigr]\\
&=\frac{1}{p} \left[-5+10(1+4p)+10\left(-1-12 p -3p \sum_{k = 1}^{4} J\left(T^{kt}, T^{2kt}, T^{2kt}\right)\right)\right.\\
&\left. \qquad \quad
+5\left(1+ 36p^2+24p +12p \sum_{k = 1}^{4} J\left(T^{kt}, T^{2kt}, T^{2kt}\right)  +4p\sum_{k=1}^4 J\left(T^{kt}, T^{kt}, T^{kt} \right)\right)\right.\\
&\left. \qquad \qquad
+ 24p^3 -280p^2-40p
-20p \sum_{k=1}^4 J\left(T^{kt},T^{kt},T^{kt}\right)-30p \sum_{k=1}^4 J\left(T^{2kt},T^{2kt},T^{kt}\right)\right]\\
&=24p^2-100p
\end{align*}
as required.
\end{proof}

\begin{proof}[Proof of Corollary \ref{cor_GStoMod}]
From Definition \ref{def_GFn} we see that
\begin{equation*}
{_{4}G} \left(\tfrac{1}{5} ,\hspace{2pt} \tfrac{2}{5},\hspace{2pt} \tfrac{3}{5} ,\hspace{2pt} \tfrac{4}{5}\right)_p
= \frac{-1}{p-1} \sum_{k=0}^{4} (-p)^k \sum_{j=\lfloor m_k r_k \rfloor +1}^{\lfloor m_{k+1}r_{k+1}\rfloor} 
\frac{{\biggfp{\tfrac{j}{p-1}}}^{4}}{\biggfp{1-\tfrac{j}{p-1}}}
\frac
{\prod_{h=0}^{4} \Biggfp{\frac{k+1-\frac{5j}{p-1}+h}{5}}}
{\prod_{h=1}^{4} \biggfp{\frac{h}{5}}} \; ,
\end{equation*}
where $r_i=\frac{p-1}{d_i}$ for $0 \leq i \leq 5$, $m_0=-1$, $m_{5}=p-2$, $d_0=d_{5}=p-1$ and $m_i=i$, $d_i=5$ for $1\leq i \leq4$. Applying Proposition \ref{prop_pGamma} (4) with $m=5$ and $x= k+1-\frac{5j}{p-1}$ we have
\begin{equation*}
{_{4}G} \left(\tfrac{1}{5} ,\hspace{2pt} \tfrac{2}{5},\hspace{2pt} \tfrac{3}{5} ,\hspace{2pt} \tfrac{4}{5}\right)_p
= \frac{-1}{p-1} \sum_{k=0}^{4} (-p)^k \sum_{j=\lfloor m_k r_k \rfloor +1}^{\lfloor m_{k+1}r_{k+1}\rfloor} 
\frac{{\biggfp{\tfrac{j}{p-1}}}^{4}}{\biggfp{1-\tfrac{j}{p-1}}}
\biggfp{k+1-\tfrac{5j}{p-1}}
\omega(5^{-5j}).
\end{equation*}
We now split off the terms when $j=0$ and $j$ is a multiple of $t=\frac{p-1}{5}$, where applicable, and use Theorem \ref{thm_GrossKoblitz} to convert to an expression involving Gauss sums. If $p \not \equiv 1 \pmod{5}$ we get
\begin{equation*}
{_{4}G} \left(\tfrac{1}{5} ,\hspace{2pt} \tfrac{2}{5},\hspace{2pt} \tfrac{3}{5} ,\hspace{2pt} \tfrac{4}{5}\right)_p
= \frac{-1}{p-1} \left[1+ \sum_{j=1}^{p-2} \frac{G(\omega^{-j})^4}{G(\omega^j)} \hspace{2pt} G(\omega^{5j})\hspace{2pt}  \omega(5^{-5j}) \right].
\end{equation*}
If $p\equiv 1 \pmod{5}$ then
\begin{equation*}
{_{4}G} \left(\tfrac{1}{5} ,\hspace{2pt} \tfrac{2}{5},\hspace{2pt} \tfrac{3}{5} ,\hspace{2pt} \tfrac{4}{5}\right)_p
= \frac{-1}{p-1} \left[1+ \sum_{\substack{j=1\\j\neq at}}^{p-2} \frac{G(\omega^{-j})^4}{G(\omega^j)} \hspace{2pt} G(\omega^{5j})\hspace{2pt}  \omega(5^{-5j}) + \sum_{a=1}^{4} \frac{G(\omega^{-at})^4}{G(\omega^{at})}  \right].
\end{equation*}
In each case, applying Proposition \ref{prop_GaussConj} yields
\begin{equation*}
{_{4}G} \left(\tfrac{1}{5} ,\hspace{2pt} \tfrac{2}{5},\hspace{2pt} \tfrac{3}{5} ,\hspace{2pt} \tfrac{4}{5}\right)_p
= -\frac{1}{p-1} \left[1+\frac{1}{p} \sum_{j=1}^{p-2}\; G(\omega^{-j})^{5} G(\omega^{5j}) \; \omega^{-5j}(-5) \right].
\end{equation*}
Taking $T=\omega$ in Theorem \ref{thm_GStoMod} completes the proof. 
\end{proof}

\section*{acknowledgements}
The author would like to thank  
the UCD Ad Astra Research Scholarship programme for its financial support
and Robert Osburn for his advice during the preparation of this paper. This work was partially supported by Science Foundation Ireland 08/RFP/MTH1081.

\end{document}